\documentclass[10pt,reqno]{amsart}
\usepackage{amsmath,amsthm,amssymb,mathrsfs,stmaryrd,color}
\usepackage[all]{xy}
\usepackage{url}
\usepackage{geometry}
\usepackage[utf8]{inputenc}
\usepackage[T1]{fontenc}

\usepackage{graphicx}

\usepackage{relsize}
\usepackage[bbgreekl]{mathbbol}
\usepackage{amsfonts}
\DeclareSymbolFontAlphabet{\mathbb}{AMSb}
\DeclareSymbolFontAlphabet{\mathbbl}{bbold}
\newcommand{\Prism}{{\mathlarger{\mathbbl{\Delta}}}}

\usepackage{enumitem}
\usepackage[colorlinks=true,hyperindex, linkcolor=magenta, pagebackref=false, citecolor=cyan,pdfpagelabels]{hyperref}
\usepackage[capitalize]{cleveref}
\setlist[enumerate]{itemsep=2pt,parsep=2pt,before={\parskip=2pt}}

\newcommand{\cosimp}[3]{\xymatrix@1{#1 \ar@<.4ex>[r] \ar@<-.4ex>[r] & {\ }#2 \ar@<0.8ex>[r] \ar[r] \ar@<-.8ex>[r] & {\ } #3 \ar@<1.2ex>[r] \ar@<.4ex>[r] \ar@<-.4ex>[r] \ar@<-1.2ex>[r] & \cdots }}

\newcommand{\adjunction}[4]{\xymatrix@1{#1{\ } \ar@<0.3ex>[r]^{ {\scriptstyle #2}} & {\ } #3 \ar@<0.3ex>[l]^{ {\scriptstyle #4}}}}

\begin{document}

\newtheorem{theorem}{Theorem}[section]
\newtheorem*{theorem*}{Theorem}
\newtheorem*{definition*}{Definition}
\newtheorem{proposition}[theorem]{Proposition}
\newtheorem{lemma}[theorem]{Lemma}
\newtheorem{corollary}[theorem]{Corollary}

\theoremstyle{definition}
\newtheorem{definition}[theorem]{Definition}
\newtheorem{question}[theorem]{Question}
\newtheorem{remark}[theorem]{Remark}
\newtheorem{warning}[theorem]{Warning}
\newtheorem{example}[theorem]{Example}
\newtheorem{notation}[theorem]{Notation}
\newtheorem{convention}[theorem]{Convention}
\newtheorem{construction}[theorem]{Construction}
\newtheorem{claim}[theorem]{Claim}
\newtheorem{assumption}[theorem]{Assumption}

\crefname{assumption}{assumption}{assumptions}
\crefname{construction}{construction}{constructions}

\newcommand{\qc}{q-\mathrm{\crys}}

\newcommand{\Shv}{\mathrm{Shv}}
\newcommand{\et}{{\acute{e}t}}
\newcommand{\crys}{\mathrm{crys}}
\newcommand{\dR}{\mathrm{dR}}
\renewcommand{\inf}{\mathrm{inf}}
\newcommand{\Hom}{\mathrm{Hom}}
\newcommand{\Sch}{\mathrm{Sch}}
\newcommand{\Spf}{\mathrm{Spf}}
\newcommand{\Spa}{\mathrm{Spa}}
\newcommand{\Spec}{\mathrm{Spec}}
\newcommand{\perf}{\mathrm{perf}}
\newcommand{\qsyn}{\mathrm{qsyn}}
\newcommand{\perfd}{\mathrm{perfd}}
\newcommand{\arc}{{\rm arc}}
\newcommand{\conj}{\mathrm{conj}}
\newcommand{\rad}{\mathrm{rad}}

\setcounter{tocdepth}{1}

\title{Prismatic $F$-crystals and crystalline Galois representations}
\author{Bhargav Bhatt}
\author{Peter Scholze}
\begin{abstract}
Let $K$ be a complete discretely valued field of mixed characteristic $(0,p)$ with perfect residue field. We prove that the category of prismatic $F$-crystals on $\mathcal O_K$ is equivalent to the category of lattices in crystalline $G_K$-representations. 
\end{abstract}

\maketitle
\setcounter{tocdepth}{1}
\tableofcontents


\section{Introduction}

Let $K$ be a complete discretely valued field of mixed characteristic $(0,p)$ with ring of integers $\mathcal{O}_K$, perfect residue field $k$,  completed algebraic closure $C$, and absolute Galois group $G_K$.  

\subsection{The main  theorem}

A fundamental discovery of Fontaine \cite{FontaineBigRingAnnals} in the study of finite dimensional $\mathbf{Q}_p$-representations of $G_K$ is the property of being {\em crystalline}; this notion is an (extremely successful) attempt at  capturing the property of ``having good reduction'' for such representations, analogous to the property of ``being unramified'' in the $\ell$-adic case\footnote{While the literal definition of unramifiedness certainly makes sense for $p$-adic representations of $G_K$, it does not capture the desired phenomena, e.g., the $p$-adic cyclotomic character and its nonzero powers are all infinitely ramified, whence, for $X/K$ smooth projective and $i > 0$, any nonzero $H^i(X_C, \mathbf{Q}_p)$ is also infinitely ramified.}. For instance, to each crystalline $G_K$-representation $V$, Fontaine has attached an $F$-isocrystal $D_{\crys}(V)$ over $k$ of the same rank as $V$, thus providing a notion of ``special fibre'' for such representations, thereby solving Grothendieck's problem of the mysterious functor. The motivating example here comes from\footnote{The assertions in this example were conjectured by Fontaine in \cite{FontaineBigRingAnnals} as the ``crystalline comparison conjecture'', and were proven by Tsuji \cite{TsujiCrys}. In fact, his work built on the work of several other people, and many different proofs of the crystalline comparison conjecture have been given since then as well; we refer to \cite[\S 1.1]{BMS1} for more on the history of this conjecture.} algebraic geometry: given a proper smooth algebraic variety $X/K$ with good reduction determined by  a proper smooth scheme $\mathcal{X}/\mathcal{O}_K$ extending $X/K$, each $G_K$-representation $H^i(X_C,\mathbf{Q}_p)$ is crystalline and $D_{\crys}(H^i(X_C,\mathbf{Q}_p)) \simeq H^i_{\crys}(\mathcal{X}_k)_{\mathbf{Q}_p}$. Moreover, for abelian varieties, one has a converse, giving a $p$-adic Neron-Ogg-Shafarevich criterion \cite{ColemanIovitaNeronOgg,MokraneNeronOgg}. 

In a different direction, to any $p$-adic formal scheme $Y$, we functorially attached in \cite{BhattScholzePrisms} a ringed site $(Y_\Prism,\mathcal{O}_\Prism)$ called the (absolute) prismatic site of $Y$ (Definition~\ref{AbsPrismSite}); the sheaf $\mathcal{O}_\Prism$ comes equipped with a Frobenius lift $\varphi$ as well as an invertible ideal sheaf $\mathcal{I}_\Prism \subset \mathcal{O}_\Prism$.  The original motivation for the construction of the prismatic site was the resulting prismatic cohomology $R\Gamma(Y_\Prism,\mathcal{O}_\Prism)  \ \rotatebox[origin=c]{-270}{$\circlearrowright$} \ \varphi$, which yields a mechanism to interpolate between different $p$-adic cohomology theories attached to $Y$ (such as the $p$-adic \'etale cohomology of the generic fibre or the crystalline cohomology of the special fibre). Turning attention now towards coefficients, one can study the following category of modules (whose definition is inspired by the structures seen on relative prismatic cohomology via the basic theorems of \cite{BhattScholzePrisms}, themselves inspired by the previous works \cite{BMS1,BMS2}):

\begin{definition}[Prismatic $F$-crystals, Definition~\ref{PrismaticFCrysDef}]
\label{PrismaticFCrysIntro}
A {\em prismatic $F$-crystal (of vector bundles) on $Y$} consists of a pair $(\mathcal{E},\varphi_{\mathcal{E}})$, where $\mathcal{E}$ is a vector bundle on the ringed site $(Y_\Prism, \mathcal{O}_\Prism)$ and $\varphi_{\mathcal{E}}$ is an isomorphism $(\varphi^* \mathcal{E}) [1/\mathcal{I}_\Prism] \simeq \mathcal{E}[1/\mathcal{I}_\Prism]$. We write $\mathrm{Vect}^{\varphi}(Y_\Prism, \mathcal{O}_\Prism)$ for the category of such $F$-crystals.
\end{definition}

Any prismatic $F$-crystal $(\mathcal{E},\varphi_{\mathcal{E}})$ has an \'etale realization $T(\mathcal{E},\varphi_{\mathcal{E}}) \in \mathrm{Loc}_{\mathbf{Z}_p}(Y_\eta)$ as a $\mathbf{Z}_p$-local system on the generic fibre $Y_\eta \subset Y$ (Construction~\ref{EtaleRealize}), as well as a crystalline realization as an $F$-crystal $D_{\crys}(\mathcal{E},\varphi_{\mathcal{E}})$ on the special fibre $Y_{p=0} \subset Y$ (Construction~\ref{CrysdRRealize}); both realizations have the same rank as $\mathcal{E}$. The picture here is that the prescence of a prismatic $F$-crystal on $Y$ lifting a given $\mathbf{Z}_p$-local system on $Y_\eta$ can be regarded as a candidate definition of (and in fact a witness for) ``a good reduction'' of the local system on $Y_\eta$. The key example again comes from geometry (Example~\ref{ex:GM}).

Specializing now to $Y=\mathrm{Spf}(\mathcal{O}_K)$, we obtain two candidate notions of ``good reduction'' for a local system on $Y_\eta$ (or equivalently for $p$-adic $G_K$-representations): one via Fontaine's theory of crystalline $G_K$-representations, and the other via the notion of prismatic $F$-crystals on $\mathrm{Spf}(\mathcal{O}_K)$. The main theorem of this paper is that these perspectives are equivalent:

\begin{theorem}[Theorem~\ref{MainThm}]
\label{MainThmIntro}
The \'etale realization functor gives an equivalence 
\begin{equation}
\label{equivintro}
 T_{\mathcal{O}_K}:\mathrm{Vect}^{\varphi}(\mathrm{Spf}(\mathcal{O}_K)_\Prism, \mathcal{O}_\Prism) \simeq \mathrm{Rep}^{\crys}_{\mathbf{Z}_p}(G_K),
 \end{equation}
where the target denotes the category of finite free $\mathbf{Z}_p$-modules $T$ equipped with a continuous $G_K$-action such that $T[1/p]$ is a crystalline representation.
\end{theorem}

There were multiple motivations for pursuing Theorem~\ref{MainThmIntro}. First, the data of a prismatic $F$-crystal on $\mathrm{Spf}(\mathcal{O}_K)$ is a rather elaborate piece of structure, so Theorem~\ref{MainThmIntro} can be regarded  as an elucidation of some new properties of crystalline $G_K$-representations. Secondly, the property of being crystalline for a $G_K$-representation is essentially a rational concept that carries little meaning, e.g., for $\mathbf{Z}/p^n$-representations. On the other hand, there is a perfectly sensible notion of a prismatic $F$-crystal over $\mathrm{Spf}(\mathcal{O}_K)$ with $\mathcal{O}_\Prism/p^n$-coefficients; accordingly, we expect Theorem~\ref{MainThmIntro} to be the first step\footnote{Actually, the notion of prismatic $F$-crystals over $\mathcal{O}_K$, while adequate for  the equivalence of Theorem~\ref{MainThmIntro}, is not quite the correct notion in more general settings, e.g., with $\mathbf{Z}/p^n$-coefficients or in the derived category. A basic problem is that a prismatic $F$-crystal $(\mathcal{E},\varphi_{\mathcal{E}})$ does not come equipped with bounds on the zeroes/poles of the isomorphism $\varphi^* \mathcal{E} [1/\mathcal{I}_\Prism] \simeq \mathcal{E}[1/\mathcal{I}_\Prism]$ along the locus $\mathcal{I}_\Prism = 0$. Instead, we expect that the correct objects to consider are perfect (or pseudocoherent) complexes over Drinfeld's stack $\Sigma''$ (or, rather, its variant over $\mathcal{O}_K$) from \cite{DrinfeldPrismatization}. Those complexes, whose definition draws inspiration from the Fontaine-Jannsen theory of $F$-gauges \cite{FontaineJannsen}, can be roughly regarded as prismatic $F$-crystals in perfect complexes where one has bounded the zeroes/poles of the Frobenius by also keeping track of the Nygaard filtration.} towards a well-behaved theory of torsion crystalline $G_K$-representations.

We prove Theorem~\ref{MainThm} by quasi-syntomic descent along the map $Y=\mathrm{Spf}(\mathcal{O}_C) \to X=\mathrm{Spf}(\mathcal{O}_K)$.  The essential ingredient (which is the subject of \S \ref{sec:EssSurj}) is the construction of a prismatic $F$-crystal $\mathfrak{M}(L)$ over $X$ lifting a given crystalline $\mathbf{Z}_p$-representation $L$ under the functor $T_{\mathcal{O}_K}$. This construction involves two steps. First, we construct $\mathfrak{N} := \mathfrak{M}(L)|_Y$ by mimicking arguments from \cite{Kisin}, thus relying crucially on Kedlaya's slope filtrations theorem \cite{KedlayaMonodromy} through an observation of Berger \cite{BergerModuleWach}. Having constructed $\mathfrak{N}$, we construct a descent isomorphism $p_1^* \mathfrak{N} \simeq p_2^* \mathfrak{N}$ over $Y \times Y$ satisfying the cocycle condition; this descent isomorphism is the primary new structure on crystalline representations constructed here. Its construction is complicated by the very inexplicit nature of the ring $\Prism_{\mathcal{O}_C \widehat{\otimes}_{\mathcal{O}_K} \mathcal{O}_C}$ over which the descent isomorphism occurs. An essential input here is the Beilinson fibre sequence \cite{AMMNBFS}. The version of the Beilinson fibre sequence that we use states that for a torsionfree $p$-complete ring $R$, there is a fibre sequence of spectra
\[
\mathrm{TC}(R)^\wedge_p[\tfrac 1p]\to \mathrm{TC}(R/p)^\wedge_p[\tfrac 1p]\to \mathrm{HC}(R)^\wedge_p[\tfrac 1p].
\]
If, for example, $R=\mathcal O_{\mathbb C_p}$, then all three spectra that occur are concentrated in even degrees; via the comparisons between topological cyclic homology and $p$-adic Hodge theory, applying $\pi_{2n}$ then gives a short exact sequence
\[
0\to \mathbb Q_p(n)\to (B_{\mathrm{crys}}^+)^{\varphi=p^n}\to B_{\mathrm{dR}}^+/\mathrm{Fil}^n\to 0,
\]
i.e.~the usual fundamental exact sequence of $p$-adic Hodge theory. Our use of the Beilinson fibre sequence will be for quasiregular semiperfectoid rings $R$, where again the terms can be made explicit in terms of prismatic cohomology.

\subsection{Relation to theory of Breuil-Kisin modules}
Theorem~\ref{MainThmIntro} has some antecedents in the literature, most notably in the work of Kisin \cite{Kisin}. To explain the connection to Kisin's results, fix a uniformizer $\pi \in \mathcal{O}_K$, giving rise to a $W(k)$-algebra surjection $\mathfrak{S} := W(k)\llbracket u \rrbracket \xrightarrow{\varphi} \mathcal{O}_K$ with kernel $I = (E(u))$ generated by an Eisenstein polynomial $E(u)$; endow $\mathfrak{S}$ with the unique Frobenius lift $\varphi:\mathfrak{S} \to \mathfrak{S}$ determined by $\varphi(u) = u^p$. Then Kisin proved the following:

\begin{theorem}[Kisin's {\cite[Theorem 0.1]{Kisin}}, Theorem~\ref{KisinFullyFaithfulBK}]
\label{KisinFFIntro}
There is a fully faithful embedding
\[ D_{\mathfrak{S}}:\mathrm{Rep}^{\crys}_{\mathbf{Z}_p}(G_K) \to \mathrm{Vect}^{\varphi}(\mathfrak{S}) \]
where the target is the category of $F$-crystals over $(\mathfrak{S},I)$, defined as in Definition~\ref{PrismaticFCrysIntro} after replacing $(\mathcal{O}_\Prism,\varphi,\mathcal{I}_\Prism)$ with $(\mathfrak{S},\varphi,I)$ and denoted $\mathrm{Mod}^{\varphi}_{\mathfrak{S}}$ in \cite{Kisin}.
\end{theorem}

The connection between Theorem~\ref{KisinFFIntro} and Theorem~\ref{MainThmIntro} is the following: the pair $(\mathfrak{S},I)$, which is an example of a Breuil-Kisin prism, gives an object of $\mathrm{Spf}(\mathcal{O}_K)_\Prism$ that covers the final object (Example~\ref{BKAinfPrismOK}). Thus, there is a natural faithful ``evaluation'' functor
\[  \mathrm{Vect}^{\varphi}(\mathrm{Spf}(\mathcal{O}_K)_\Prism, \mathcal{O}_\Prism) \xrightarrow{\mathrm{ev}_{(\mathfrak{S},I)}} \mathrm{Vect}^{\varphi}(\mathfrak{S}).\]
The functor $D_{\mathfrak{S}}$ in Theorem~\ref{KisinFFIntro} is then obtained by composing the inverse of the equivalence $T_{\mathcal{O}_K}$ in Theorem~\ref{MainThmIntro} with the evaluation functor $\mathrm{ev}_{(\mathfrak{S},I)}$ above\footnote{In fact, our methods allow us to prove full faithfulness assertion in Theorem~\ref{KisinFFIntro} as well. But this not a completely new proof of the latter as we use some of the same ideas as \cite{Kisin}. In particular, like \cite{Kisin}, we rely on Berger's observation \cite{BergerModuleWach} translating weak admissibility of filtered $\varphi$-modules to a boundedness condition for an auxiliary module over an extended Robba ring; unlike \cite{Kisin}, we apply this observation directly over $A_{\inf}$.}. Realizing the functor $D_{\mathfrak{S}}$ via the equivalence $T_{\mathcal{O}_K}$ in this fashion has some concrete consequences. For instance, the target of the functor $D_{\mathfrak{S}}$ depends on the choice of the uniformizer $\pi$, while the equivalence $T_{\mathcal{O}_K}$ is completely canonical; factoring the former over the latter then immediately yields a certain ``independence of the uniformizer'' result for the latter that was shown by Liu \cite{LiuKisinCompat} (see Remark~\ref{LiuCompat}). More conceptually, lifting an object of $\mathrm{Vect}^{\varphi}(\mathfrak{S})$ to a prismatic $F$-crystal over $\mathcal{O}_K$ has an intrinsic meaning in the prismatic theory: as $(\mathfrak{S},I)$ covers the final object of $\mathrm{Spf}(\mathcal{O}_K)_\Prism$, such a lift determines and is determined by  descent data over the self-coproduct $\mathfrak{S}^{(1)}$ of $(\mathfrak{S},I)$ with itself in $\mathrm{Spf}(\mathcal{O}_K)_\Prism$. This interpretation of the essential image of $D_{\mathfrak{S}}$ can be roughly regarded as an integral avatar of Kisin's result describing the essential image of $D_{\mathfrak{S}}[1/p]$ in terms of the existence of certain logarithmic connections on the open unit disc; see \S \ref{ss:LogConnBK} for more on the relation between the two notions.

Kisin used Theorem~\ref{KisinFFIntro} to establish the following conjecture of Breuil \cite{BreuilKinfinity} on Galois representations.

\begin{theorem}[Kisin's {\cite[Theorem 0.2]{Kisin}}, Theorem~\ref{KisinBreuil}]
\label{KisinBreuilIntro}
Let $K_\infty/K$ be the extension of $K$ obtained by adjoining a compatible system of $p$-power roots of the uniformizer $\pi$ inside $C$.  The restriction functor
\[ \mathrm{Rep}^{\crys}_{\mathbf{Q}_p}(G_K) \to \mathrm{Rep}_{\mathbf{Q}_p}(G_{K_\infty})\]
is fully faithful.
\end{theorem}

We give an alternative proof of this result in this paper. Granting what was already explained above, the essential remaining point is to show that the base change functor
\[  \mathrm{Vect}^{\varphi}(\mathfrak{S})  \xrightarrow{- \otimes_{\mathfrak{S}} \mathfrak{S}[1/u]^{\wedge}_p} \mathrm{Vect}^{\varphi}(\mathfrak{S}[1/u]^{\wedge}_p) \]
is fully faithful, where the target is the category of finite free $\mathfrak{S}[1/u]^{\wedge}_p$-modules $M$ equipped with an isomorphism $\varphi^* M \simeq M$ (where we note that $E(u)$ is invertible in $\mathfrak{S}[1/u]^{\wedge}_p$).  We give a direct argument (Theorem~\ref{EtaleRealizeBK}) for this full faithfulness by comparing $\mathfrak{S}$ to Fontaine's period ring $A_{\inf}$; the main new idea is to observe and exploit a certain orthogonality property of the Kummer tower $K_\infty/K$ with the cyclotomic extension $K(\mu_{p^\infty})/K$ that manifests itself via the behaviour of certain elements of $A_{\inf}$ coming from each of these towers (Lemmas~\ref{Ainfentire} and \ref{BoundPoles}).

We expect that Theorem~\ref{MainThmIntro} should extend to the semistable case, provided one replaces the prismatic site with the log prismatic site of Koshikawa \cite{KoshikawaLogPrismaticI}, and moreover that the resulting theorem ought to be related to \cite{LiuNoteLattice} in the same way Theorem~\ref{MainThmIntro} was related to \cite{Kisin}. However, we do not pursue this direction here\footnote{Since the appearance of the first version of this paper, this goal was realized in \cite{DuLiuLogPrismFCrys}, which also gives a new proof of Theorem~\ref{MainThmIntro}.}.

\subsection*{Acknowledgements}

This project has a long history. In fact, we conjectured Theorem~\ref{MainThmIntro} almost immediately after the discovery of the prismatic site in the summer of 2017. After quick progress on all aspects except the boundedness of the descent isomorphism (Proposition~\ref{ACDescent}), we were stuck for a few years (as recorded in the second author's Fields Medal video for the 2018 ICM, where he can be seen staring at a blackboard, unsuccessfully contemplating the structure of $\Prism_{\mathcal O_C \widehat{\otimes}_{\mathcal O_K} \mathcal O_C}$), and moved on to other projects. Partially inspired by Drinfeld's \cite{DrinfeldPrismatization} as well as ongoing discussions with Toby Gee and Jacob Lurie, we returned to the project and were able to prove the desired boundedness using \cite{AMMNBFS} in the summer of 2020. In this period, we have benefitted from conversations with many mathematicians, including Vladimir Drinfeld, Toby Gee, Shizhang Li, Jacob Lurie, Akhil Mathew, Matthew Morrow and Andrew Snowden. We also thank Mark Kisin and the referee for useful comments on a previous version of this paper.

During the preparation of this paper, Bhatt was partially supported by the NSF (\#1801689, \#1952399, \#1840234), a Packard fellowship, and the Simons Foundation (\#622511). Scholze was supported by a DFG Leibniz Prize, and by the DFG under the Excellence Strategy – EXC-2047/1 – 390685813.

\newpage
\section{The absolute prismatic site and the quasi-syntomic site}
\label{sec:PrismQsyn}

In this section, we introduce the absolute prismatic site (Definition~\ref{AbsPrismSite}), and explain how to describe vector bundles (with respect to various sheaves of rings) on this site explicitly in terms of modules over prisms (Proposition~\ref{VBDescentPrism}). For future use, it will also be important to have a description of these bundles in terms of the quasi-syntomic site (Definition~\ref{QSynSite}), so we record such a description in Proposition~\ref{ComparePrismQSyn}.

\begin{notation}[Vector bundles and perfect complexes on a ringed topos]
Say $(\mathcal{X},\mathcal{O})$ is a ringed topos. 

A {\em vector bundle} on $(\mathcal{X},\mathcal{O})$ is an $\mathcal{O}$-module $E$ such that there exists a cover $\{U_i\}$ of $\mathcal{X}$ and finite projective $\mathcal{O}(U_i)$-modules $P_i$ such that $E|_{U_i} \simeq P_i \otimes_{\mathcal{O}(U_i)} \mathcal{O}_{U_i}$. Write $\mathrm{Vect}(\mathcal{X},\mathcal{O})$ for the category of all vector bundles. Note that $\mathrm{Vect}(\mathcal{X},\mathcal{O})$ can also be defined as the global sections over $\mathcal{X}$ of the sheafification of the presheaf of categories on $\mathcal{X}$ given by $U \mapsto \mathrm{Vect}(\mathcal{O}(U))$.

A {\em perfect complex} on $(\mathcal{X},\mathcal{O})$ is an object $E \in D(\mathcal{X},\mathcal{O})$  such that there exists a cover $\{U_i\}$ of $\mathcal{X}$ and perfect complexes $P_i \in D(\mathcal{O}(U_i))$ such that $E|_{U_i} \simeq P_i \otimes_{\mathcal{O}(U_i)} \mathcal{O}_{U_i}$. Write $D_{\perf}(\mathcal{X},\mathcal{O}) \subset D(\mathcal{X},\mathcal{O})$ for the full subcategory of perfect complexes. If $\mathcal{X}$ admits a basis of $\mathcal{O}$-acyclic objects (i.e., objects $V$ with $R\Gamma(V,\mathcal{O}) \simeq \mathcal{O}(V)$), then $D_{\perf}(\mathcal{X},\mathcal{O})$ can also be defined as the global sections over $\mathcal{X}$ of the sheafification of the presheaf of $\infty$-categories on $\mathcal{X}$ given by $U \mapsto D_{\perf}(\mathcal{O}(U))$; this will be the case in all our applications.
\end{notation}

To understand vector bundles with respect to certain ``Banach-style'' sheaves on the prismatic site, it will be quite convenient to use the following general descent theorem, generalizing known descent results for coherent sheaves in rigid geometry.

\begin{theorem}[Drinfeld-Mathew, {\cite[Theorem 5.8]{MathewDescent}}]
\label{DMDescent}
Let $R$ be a connective $E_\infty$-ring, and let $I \subset \pi_0(R)$ be a finitely generated ideal.  Consider the following functors defined on the $\infty$-category of connective $E_\infty$-$R$-algebras:
\begin{enumerate}
\item $S \mapsto  D_{\perf}(\mathrm{Spec}(S^{\wedge}_I) - V(IS))$. 
\item $S \mapsto  D^{-}_{coh}(\mathrm{Spec}(S^{\wedge}_I) - V(IS))$. 
\item $S \mapsto  \mathrm{Vect}(\mathrm{Spec}(S^{\wedge}_I) - V(IS))$. 
\end{enumerate}
Each of these functors is a sheaf for the $I$-completely flat topology. 
\end{theorem}

We remark that the same result also applies to $D_{\perf}(S^{\wedge}_I)$, $D^{-}_{coh}(S^{\wedge}_I)$ and $\mathrm{Vect}(S^{\wedge}_I)$, i.e.~without inverting $I$. In that case, the result is easy, as all $\infty$-categories are equivalent to the limit of the corresponding $\infty$-categories for quotients of $S$ on which $I$ is nilpotent (where one can apply usual faithfully flat descent); see \cite[Appendix A]{ALBPrismaticDieudonne} for the proof in case of vector bundles.

The main object of study in this article is the following site, introduced in \cite[Remark 4.7]{BhattScholzePrisms}.

\begin{definition}[The absolute prismatic site]
\label{AbsPrismSite}
Given a $p$-adic formal scheme $X$, we write $X_\Prism$ for the opposite of the category of bounded prisms $(A,I)$ equipped with a map $\mathrm{Spf}(A/I) \to X$; we endow $X_\Prism$ with the topology induced by the flat topology on prisms and refer to it as the {\em absolute prismatic site of $X$}. Write $\mathcal{O}_\Prism$ for the structure sheaf, and $\mathcal{I}_\Prism \subset \mathcal{O}_\Prism$ for the ideal sheaf of the Hodge-Tate divisor. 
\end{definition}

\begin{remark}
\label{PrismaticReplete}
The topos $\mathrm{Shv}(X_\Prism)$ is replete in the sense of \cite{BhattScholzeProetale}: this follows as an inductive limit of faithfully flat maps of prisms is a faithfully flat map of prisms. In particular, derived inverse limits behave well, so one has $\mathcal{O}_\Prism \simeq R\lim \mathcal{O}_\Prism/(p^n, \mathcal{I}_\Prism^n)$ since we have $A \simeq R\lim A/(p^n,I^n)$ for any bounded prism $(A,I)$; similar assertions hold true for the variants  $\mathcal{O}_\Prism[1/p]^{\wedge}_{\mathcal{I}_\Prism}$ or $\mathcal{O}_\Prism[1/\mathcal{I}_\Prism]^{\wedge}_p$.
\end{remark}

\begin{example}[The prism of a qrsp ring]
If $X=\mathrm{Spf}(R)$ for $R$ quasiregular semiperfectoid as in \cite[Definition 4.20]{BMS2}, then $X_\Prism$ has final object given by the prism $\Prism_R$ of $R$ by \cite[Proposition 7.10]{BhattScholzePrisms}.
\end{example}

\begin{example}[The Breuil-Kisin and $A_{\inf}$-prisms over $\mathcal{O}_K$]
\label{BKAinfPrismOK}
Let $K/\mathbf{Q}_p$ be a discretely valued extension with perfect residue field $k$. We shall construct two examples of objects in  $\mathrm{Spf}(\mathcal{O}_K)_\Prism$; both examples give covers of the final object of the topos.

\begin{enumerate}
\item (Breuil-Kisin prisms) Choose a uniformizer $\pi \in \mathcal{O}_K$. Writing $W = W(k)$ and $\mathfrak{S} = W\llbracket u\rrbracket$, we obtain a surjection $\mathfrak{S} \to \mathcal{O}_K$ with kernel generated by an Eisenstein polynomial $E(u)$. Endowing $\mathfrak{S}$ with the $\delta$-structure determined by the Witt vector Frobenius on $W(k)$ and $\varphi(u)=u^p$, the pair $(A,I) = (\mathfrak{S},(E(u))$ gives an object of $\mathrm{Spf}(\mathcal{O}_K)_\Prism$.

Moreover, we claim that $(A,I)$ covers the final object of $\mathrm{Shv}(\mathrm{Spf}(\mathcal{O}_K)_\Prism)$; one can deduce this by mapping $(A,I)$ to the $A_{\inf}$-prism and deducing the claim from the analogous property for the $A_{\inf}$-prism (see part (2) of this example), but we give a direct argument. Fix an object $(B,J) \in \mathrm{Spf}(\mathcal{O}_K)_\Prism$ with structure map $\mathcal{O}_K \to B/J$. We shall construct a faithfully flat map $(B,J) \to (C,JC)$ of prisms such that there exists a map  $(A,I) \to (C,JC)$ in $\mathrm{Spf}(\mathcal{O}_K)_\Prism$; this will prove the claim.  By standard deformation theory, there is a unique $W$-algebra structure on all objects in sight. Pick $v \in B$ lifting the image of $\pi \in \mathcal{O}_K$ under the structure map $\mathcal{O}_K \to B/J$; note that $\pi$ generates $\mathcal{O}_K$ as a $W$-algebra.  We then define $C$ as a suitable prismatic envelope:
\[ C = (A \otimes_W B)\{\frac{u-v}{J}\}^{\wedge}_{(p,J)} = B[u]\{\frac{u-v}{J}\}^{\wedge}_{(p,J)}.\]
By \cite[Proposition 3.13]{BhattScholzePrisms}, this gives a $(p,J)$-completely flat $\delta$-$B$-algebra, so $(C,JC)$ is a flat cover $(B,J)$; we shall check that this does the job. Since $u\equiv v \mod JC$, we also have $E(u) \equiv E(v) \mod JC$, whence $E(u) \in JC$ since $E(v) \in JB$ as $E(\pi) = 0$ in $\mathcal{O}_K$. By the irreducibility lemma (\cite[Lemma 2.24]{BhattScholzePrisms}) for distinguished elements, it follows that $E(u)C = JC$.  Thus, the natural map $A \to C$ extends to a map of prisms $(A,I) \to (C,JC)$.  Moreover, the two resulting compositions $\mathcal{O}_K \simeq A/I \to C/JC$ and $\mathcal{O}_K \to B/J \to C/JC$ are the same (they carry the generator $\pi \in \mathcal{O}_K$ to the same element $u=v \in C/JC$), so the map $(A,I) \to (C,JC)$ is indeed a map  in $\mathrm{Spf}(\mathcal{O}_K)_\Prism$, proving the desired covering property. In fact, one can check that the object $(C,JC)$ thus constructed is the coproduct of $(A,I)$ and $(B,J)$ in the category $\mathrm{Spf}(\mathcal{O}_K)_\Prism$.

\item (The $A_{\inf}$-prism) Fix a completed algebraic closure $C/K$. Then $\mathcal{O}_C$ is perfectoid, so $\Prism_{\mathcal{O}_C}$ determines an object of $\mathrm{Spf}(\mathcal{O}_K)_\Prism$ via restriction of scalars along $\mathcal{O}_K \to \mathcal{O}_C$. Using the fact that $\mathcal{O}_K \to \mathcal{O}_C$ is a quasi-syntomic cover as well as the lifting result in \cite[Proposition 7.11]{BhattScholzePrisms}, it follows that $\Prism_{\mathcal{O}_C}$ covers the final object of the topos $\mathrm{Shv}(\mathrm{Spf}(\mathcal{O}_K)_\Prism)$. Note that this construction will be used (and in fact generalized and elucidated) in the proof of Proposition~\ref{ComparePrismQSyn}.
\end{enumerate}
Both the above examples will feature prominently in the rest of the paper. 
\end{example}

\begin{proposition}[Describing vector bundles and perfect complexes on $X_\Prism$ explicitly]
\label{VBDescentPrism}
Let $X$ be a $p$-adic formal scheme. There is a natural equivalence
\begin{equation}
\label{eq:VBDescent}
 \lim_{(A,I) \in X_\Prism} \mathrm{Vect}(A) \simeq \mathrm{Vect}(X_\Prism, \mathcal{O}_\Prism).
 \end{equation}
A similar assertion holds true with $D_{\perf}(-)$ replacing $\mathrm{Vect}(-)$. Moreover, analogous statements hold true if $\mathcal{O}_\Prism$ is replaced by $\mathcal{O}_\Prism[1/p]^{\wedge}_{\mathcal{I}_\Prism}$ or $\mathcal{O}_\Prism[1/\mathcal{I}_\Prism]^{\wedge}_p$. 
\end{proposition}
\begin{proof}
The object on the right in \eqref{eq:VBDescent} can regarded as the global sections of the stackification of the assignment $(A,I) \mapsto \mathrm{Vect}(A)$ on $(A,I) \in X_\Prism$. Thus, we must show this assignment is already a sheaf for the flat topology. This follows by $(p,I)$-completely faithfully flat descent for vector bundles, see the lines following Theorem~\ref{DMDescent}. The same argument also applies to $D_{\perf}(-)$ replacing $\mathrm{Vect}(-)$.

Next, let us explain the descent claim for $\mathcal{O}_\Prism[1/p]^{\wedge}_{\mathcal{I}_\Prism}$; the claim for $\mathcal{O}_\Prism[1/\mathcal{I}_\Prism]^{\wedge}_p$ is analogous. For the functor $D_{\perf}(-)$, we may argue by devissage modulo powers of $\mathcal{I}_\Prism$, so it suffices to show 
\[ \lim_{(A,I) \in X_\Prism} D_{\perf}(A/I[1/p]) \simeq D_{\perf}(X_\Prism, \mathcal{O}_\Prism/\mathcal{I}_\Prism[1/p]).\]
Arguing as in the previous paragraph, the claim follows from Theorem~\ref{DMDescent} (1). The statement for vector bundles is then deduced by observing that given a commutative ring $R$ that is derived $J$-complete with respect to a finitely generated ideal $J$, an object $E \in D_{\perf}(R)$ lies in $\mathrm{Vect}(R)$ if and only if $E \otimes_R^L R/J \in D_{\perf}(R/J)$ lies in $\mathrm{Vect}(R/J)$. 
\end{proof}

\begin{example}[Describing vector bundles on $\mathrm{Spf}(\mathcal{O}_K)_\Prism$ via Breuil-Kisin prisms]
Keep notation as in Example~\ref{BKAinfPrismOK}.  Then we have a descent equivalence
\[ \mathrm{Vect}(\mathrm{Spf}(\mathcal{O}_K)_\Prism,\mathcal{O}_\Prism) \simeq \lim \mathrm{Vect}(A^\bullet) \simeq \lim \left(\xymatrix@1{\mathrm{Vect}(A^0) \ar@<.4ex>[r] \ar@<-.4ex>[r] & {\ } \mathrm{Vect}(A^1) \ar@<0.8ex>[r] \ar[r] \ar@<-.8ex>[r] & {\ } \mathrm{Vect}(A^2)   } \right) \]
where $(A^\bullet,I^\bullet)$ is the cosimplicial object of $\mathrm{Spf}(\mathcal{O}_K)_\Prism$ obtained by taking the Cech nerve of $(A,I)$ over the initial object. As $A^0=A$, this equivalence shows that lifting a finite projective $A$-module $M$ to a crystal of vector bundles on $\mathrm{Spf}(\mathcal{O}_K)_\Prism$ entails specifying a descent isomorphism $p_1^* M \simeq p_2^* M$ over $A^1$ satisfying the cocycle condition over $A^2$; see also Construction~\ref{PrismaticCechNerve} for an explicit description of $A^1$ (denoted $\mathfrak{S}^{(1)}$ there).
\end{example}

We next relate the prismatic site to the quasi-syntomic one, see also \cite{ALBPrismaticDieudonne}.

\begin{definition}[The quasi-syntomic site]
\label{QSynSite}
For a quasi-syntomic $p$-adic formal scheme $X$, write $X_{qsyn}$ for the opposite of the category of quasi-syntomic maps $\eta:\mathrm{Spf}(R) \to X$, endowed with the quasi-syntomic topology (see \cite[Definition 4.1, Variant 4.35]{BMS2}); we call this {\em the quasi-syntomic site} of $X$. Write $X_{qrsp} \subset X_{qsyn}$ for the full subcategory spanned by $\eta:\mathrm{Spf}(R) \to X$ with $R$ semiperfect modulo $p$ and admitting a map from a perfectoid ring (or equivalently with $R$ qrsp, see \cite[Definition 4.20, Variant 4.35]{BMS2}). By \cite[Proposition 4.31, Variant 4.35]{BMS2}, restricting sheaves induces an equivalence 
\[ \mathrm{Shv}(X_{qsyn}) \simeq \mathrm{Shv}(X_{qrsp}),\]
which we can use to define sheaves on $X_{qsyn}$. In particular, the assignment carrying $(\eta:\mathrm{Spf}(R) \to X) \in X_{qrsp}$ to $\Prism_R$ gives a sheaf $\Prism_\bullet$ of rings on $X_{qsyn}$; write $I \subset \Prism_\bullet$ for the ideal sheaf of the Hodge-Tate divisor.
\end{definition}

\begin{remark}
The topos $\mathrm{Shv}(X_{qsyn})$ is replete: this follows as an inductive limit of quasi-syntomic covers is a quasi-syntomic cover. Thus, one has analogs of the assertions in Remark~\ref{PrismaticReplete} with the triple $(X_\Prism,\mathcal{O}_\Prism, \mathcal{I}_\Prism)$ replaced by $(X_{qsyn},\Prism_\bullet, I)$.
\end{remark}

\begin{remark}
If $X$ is a quasi-syntomic $p$-adic formal scheme and $R \in X_{qsyn}$, then $R$ is itself a quasi-syntomic ring. In particular, if $R$ is qrsp, then $\Prism_R$ is a bounded prism. 
\end{remark}

\begin{example}
If $X=\mathrm{Spf}(R)$ for $R$ quasiregular semiperfectoid, then $X_{qrsp}$ has a final object given by $R$ itself.
\end{example}

\begin{proposition}[Describing vector bundles and perfect complexes on $X_{qsyn}$ explicitly]
\label{VBDescentQSyn}
Let $X$ be a quasi-syntomic $p$-adic formal scheme. Then there is a natural equivalence
\[ \lim_{R \in X_{qrsp}} \mathrm{Vect}(\Prism_R) \simeq \mathrm{Vect}(X_{qsyn},\Prism_\bullet).\]
A similar assertion holds true with $D_{\perf}(-)$ replacing $\mathrm{Vect}(-)$. Moreover, analogous statements hold true if $\Prism_\bullet$ is replaced with $\Prism_\bullet[1/p]^{\wedge}_I$ or $\Prism_\bullet[1/I]^{\wedge}_p$.
\end{proposition}
\begin{proof}
As $X_{qsyn}$ and $X_{qrsp}$ give the same topos, we may replace $X_{qsyn}$ with $X_{qrsp}$ in all assertions. In this case, the claim follows by similar arguments to the ones in Proposition~\ref{VBDescentPrism}.
\end{proof}

\begin{proposition}[Relating $X_\Prism$ and $X_{qsyn}$]
\label{ComparePrismQSyn}
Let $X$ be a quasi-syntomic $p$-adic formal scheme. Then there is a natural equivalence
\begin{equation}
\label{eq:VBCompareQSynPrism}
 \mathrm{Vect}(X_\Prism,\mathcal{O}_\Prism) \simeq \mathrm{Vect}(X_{qsyn},\Prism_\bullet).
 \end{equation}
A similar assertion holds true with $D_{\perf}(-)$ replacing $\mathrm{Vect}(-)$. Moreover, analogous statements hold true if $\mathcal{O}_\Prism$ is replaced by $\mathcal{O}_\Prism[1/p]^{\wedge}_{\mathcal{I}_\Prism}$ or $\mathcal{O}_\Prism[1/\mathcal{I}_\Prism]^{\wedge}_p$, and  $\Prism_\bullet$ is correspondingly replaced with $\Prism_\bullet[1/p]^{\wedge}_I$ or $\Prism_\bullet[1/I]^{\wedge}_p$.
\end{proposition}
\begin{proof}
We give the argument for \eqref{eq:VBCompareQSynPrism}; the rest of the statements follow similarly. Using \cite[Proposition 4.31, Variant 4.35]{BMS2}, it suffices to prove the assertions with $X_{qsyn}$ replaced by $X_{qrsp}$.  By Zariski descent for both sides, we may assume $X = \mathrm{Spf}(R)$ is affine. In the rest of the proof, we use affine notation, so $R_{qrsp}$ is a category of $R$-algebras (rather than its opposite), etc. 

There is a natural functor $\Prism(-):R_{qrsp} \to R_\Prism$ determined by $S \mapsto \Prism_S$. Using the initiality of $\Prism_S \in S_\Prism$ for qrsp $S$, this functor satisfies the following universal property: for any $(A,I) \in R_\Prism$, we have
\[ \mathrm{Hom}_{R-\text{alg}}(S, A/I) \simeq \mathrm{Hom}_{R_\Prism}(\Prism_S, (A,I)).\]
In particular, the functor $\Prism(-)$ commutes with finite non-empty coproducts. Moreover, it follows from \cite[Proposition 7.11]{BhattScholzePrisms} that if $S \in R_{qrsp}$ is a cover of $R$ (i.e., $p$-completely faithfully flat over $R$), then $\Prism_S \in R_\Prism$ determines a cover of the final object of $\mathrm{Shv}(X_\Prism)$.

We can now prove the proposition. Choose a quasi-syntomic cover $R \to S$ with $S$ qrsp. Write $S^\bullet$ for the Cech nerve of $R \to S$, so we have
\[ \mathrm{Vect}(R_{qsyn},\Prism_\bullet) \simeq \lim \mathrm{Vect}(\Prism_{S^\bullet})\] 
by Cech theory in $R_{qsyn}$. Moreover, the observations in the preceding paragraph and Cech theory in $R_\Prism$ then show that the right side is also $\mathrm{Vect}(R_\Prism, \mathcal{O}_\Prism)$, giving the equivalence in the theorem. To check this equivalence is independent of the choice of $S$, one can either use that the category of all $R \to S$ as above has finite non-empty coproducts and is thus sifted, or one can note that the remarks in the previous paragraph show that $\Prism(-)$ determines a map
\[ \nu:(\mathrm{Shv}(X_\Prism), \mathcal{O}_\Prism) \to (\mathrm{Shv}(X_{qsyn}), \Prism_\bullet) \]
of ringed topoi, with the desired equivalence being induced by $\nu^*$.
\end{proof}

\newpage
\section{Local systems on the generic fibre via the prismatic site}

The goal of this section is to explain how $\mathbf{Z}_p$-local systems on the generic fibre of a $p$-adic formal scheme $X$ can be regarded as certain $F$-crystals on $X_\Prism$ (Corollary~\ref{LocSysLaurentFCrysIsog}). In fact, this relationship, which ultimately comes from Artin-Schreier theory, is quite robust and extends to an equivalence of derived categories (Corollary~\ref{LocSysLaurentFCrys}).

\begin{notation}
For a bounded $p$-adic formal scheme $X$,\footnote{Recall that ``bounded'' means that it is locally of the form $\mathrm{Spf}(R)$ where $R$ is bounded in the sense that the $p$-primary torsion $R[p^\infty]$ has bounded exponent; cf.~\cite[Definition 3.2 (2)]{BhattScholzePrisms}.} we write $X_\eta$ for the generic fibre of $X$, regarded as a presheaf on perfectoid spaces over $\mathbf{Q}_p$ (that is in fact a locally spatial diamond, see \cite[Section 15]{ECoD}). We shall use $D^{(b)}_{lisse}(X_\eta,\mathbf{Z}_p)$ to denote the full subcategory of $D(X_{\eta,proet}, \mathbf{Z}_p)$ spanned by locally bounded objects which are derived $p$-complete and whose mod $p$ reduction has cohomology sheaves that are locally constant with finitely generated stalks. We will use that the association $X\mapsto D^{(b)}_{lisse}(X_\eta,\mathbf{Z}_p)$ defines a sheaf of $\infty$-categories for the quasisyntomic topology on $X$, as follows from the v-descent results in \cite{ECoD} (and the observation that any quasisyntomic cover of $X$ induces a v-cover of $X_\eta$).
\end{notation}

For the following definition, recall that for a prism $(A,I)$, even while $I$ is not $\varphi$-stable, there is still a Frobenius on the $p$-adic completion $A[1/I]^\wedge_p$ -- for this it suffices that $\varphi(I)\equiv I^p$ modulo $p$.

\begin{definition}[Laurent $F$-crystals]
\label{LaurentFCrysDef}
Fix a bounded $p$-adic formal scheme $X$. We define
\[ D_{\perf}(X_\Prism, \mathcal{O}_\Prism[1/\mathcal{I}_\Prism]^\wedge_p)^{\varphi=1} = \lim_{(A,I) \in X_\Prism} D_{\perf}(A[1/I]^\wedge_p)^{\varphi=1}.\]
Thanks to Proposition~\ref{VBDescentPrism}, this is the $\infty$-category of pairs $(E,\varphi_E)$, where $E$ is a crystal of perfect complexes  on $(X_\Prism,\mathcal{O}_\Prism[1/\mathcal{I}_\Prism]^{\wedge}_p)$ and $\varphi_E:\varphi^* E \simeq E$ is an isomorphism. Similarly, we set
\[ \mathrm{Vect}(X_\Prism, \mathcal{O}_\Prism[1/\mathcal{I}_\Prism]^\wedge_p)^{\varphi=1} = \lim_{(A,I) \in X_\Prism} \mathrm{Vect}(A[1/I]^{\wedge}_p)^{\varphi=1}\]
to be the category of crystals of vector bundles $E$ on $(X_\Prism,\mathcal{O}_\Prism[1/\mathcal{I}_\Prism]^{\wedge}_p)$ equipped with isomorphisms $\varphi_E:\varphi^* E \simeq E$. We shall informally refer to these objects as Laurent $F$-crystals on $X_\Prism$; the name $F$-crystal will be reserved for a more restrictive notion. 
\end{definition}

\begin{remark}
If $X$ is a scheme of characteristic $p$, then any prism $(A,I) \in X_\Prism$ must have $I=(p)$, whence $A[1/I]^{\wedge}_p = 0$. Thus, $\mathcal{O}_\Prism[1/\mathcal{I}_\Prism]^{\wedge}_p = 0$ and $D_{\perf}(X_\Prism, \mathcal{O}_\Prism[1/\mathcal{I}_\Prism]^\wedge_p)^{\varphi=1}=0$ as well. 
\end{remark}

To study Laurent $F$-crystals, we use what is often called Artin--Schreier--Witt theory, in the following form.

\begin{proposition}[The Riemann-Hilbert correspondence for $\mathbf{F}_p$-local systems in characteristic $p$]
\label{RHLocSys}
Let $S$ be a commutative $\mathbf{F}_p$-algebra. Then extension of scalars along $\mathbf{F}_p \to \mathcal{O}_{\mathrm{Spec}(S),et}$ and taking Frobenius fixed points give mutually inverse equivalences
\[ D^b_{lisse}(\mathrm{Spec}(S), \mathbf{F}_p) \simeq D_{\perf}(S)^{\varphi=1}.\]
\end{proposition}

This proposition is well-known, but we do not know a reference written in the above generality. The special case for vector bundles was proven by Katz \cite[Proposition 4.1.1]{Katzpadic} (for some reason, he makes the assumption that the ring is normal, but that is never used).

\begin{proof}
For full faithfulness, by finite \'etale descent for both sides, we reduce to the checking the statement for endomorphisms of the constant sheaf, which follows from the Artin-Schreier sequence. To show essential surjectivity, fix $(E,\alpha) \in D_{\perf}(S)^{\varphi=1}$, i.e., $E \in D_{\perf}(S)$ is a perfect $S$-complex and $\alpha:\varphi^* E \simeq E$ is an isomorphism. We must show that $(E,\alpha)$ lies in the essential image of the extension of scalars functor considered in the proposition. 

We claim that each $H^i(E)$ is a vector bundle. By descending induction, it suffices to show the claim for the highest non-zero cohomology group $H^n(E)$, which is finitely presented as $E$ is a perfect complex. Now the claim is \cite[Proposition 3.2.13]{KedlayaLiu}, but for the reader's convenience we include an argument. By noetherian approximation, we may assume $S$ is noetherian. As the property of being a vector bundle can be detected locally, we may assume $S$ is a local noetherian ring. The map $\alpha$ induces an isomorphism $\varphi^* H^n(E) \simeq H^n(E)$. But then any Fitting ideal $J \subset S$ of $H^n(E)$ satisfies $J = \varphi(J)S$, which forces $J \in \{0,S\}$ by Krull's intersection theorem. It follows that $H^n(E)$ is a vector bundle, as asserted.

As each $H^i(E)$ is a vector bundle and we have already shown full faithfulness, it suffices to prove the essential surjectivity claim in the proposition for each cohomology group separately, so we may assume $E$ is itself a vector bundle, in which case the result follows from (the proof of) \cite[Proposition 4.1.1]{Katzpadic}, see also \cite[Proposition 3.2.7]{KedlayaLiu}.
\end{proof}

\begin{example}[Laurent $F$-crystals over a qrsp]
\label{PerfdLaurentCrys}
Let $X = \mathrm{Spf}(R)$ with $R$ a qrsp ring. Then $X_\Prism$ has an initial object determined by the prism $(\Prism_R,I)$ of $R$. Consequently, we learn that
\[D_{\perf}(X_\Prism, \mathcal{O}_\Prism[1/\mathcal{I}_\Prism]^\wedge_p)^{\varphi=1} = D_{\perf}(\Prism_R[1/I]^\wedge_p)^{\varphi=1}.\]
If we further assume that $R$ is perfectoid, then we can identify $\Prism_R[1/I]^{\wedge}_p \simeq W(R^\flat[1/I])$. The Riemann-Hilbert correspondence in Proposition~\ref{RHLocSys} as well as the tilting equivalence imply that the $\infty$-category above identifies with the derived $\infty$-category $D^b_{lisse}(X_\eta, \mathbf{Z}_p)$ of lisse $\mathbf{Z}_p$-sheaves on the generic fibre $X_\eta = \mathrm{Spa}(R[1/p],R)$. 
\end{example}

\begin{proposition}[Invariance of unit $F$-crystals under completed perfections]
\label{BanachPerfFrobMod}
Let $R$ be a ring of characteristic $p$ containing an element $t$ such that $R$ is derived $t$-complete. Let $S = (R_\perf)^\wedge_t$ denote the $t$-completed perfection of $R$. The base change functors
\[ D_{\perf}(R[1/t])^{\varphi=1} \xrightarrow{a} D_{\perf}(R_\perf[1/t])^{\varphi=1} \xrightarrow{b} D_{\perf}(S[1/t])^{\varphi=1}\]
are equivalences. 
\end{proposition}
\begin{proof}
The functor $a$ is an equivalence by Proposition~\ref{RHLocSys} and the topological invariance of the \'etale site. The full faithfulness of $b \circ a$ (and thus $b$) follows from \cite[Lemma 9.2]{BhattScholzePrisms} as perfect complexes over $R$ are automatically derived $t$-complete. Using Proposition~\ref{RHLocSys} for $S$ and the full faithfulness of $b \circ a$, it suffices to check that every lisse $\mathbf{F}_p$-sheaf on $\mathrm{Spec}(S[1/t])$ is pulled back from $\mathrm{Spec}(R[1/t])$. But this follows form the Elkik(-Gabber-Ramero) approximation \cite[Proposition 5.4.54]{GabberRamero}, which ensures that the $R[1/t]$ and $S[1/t]$ have isomorphic fundamental groups.
\end{proof}

\begin{corollary}[$\mathbf{Z}_p$-local systems as Laurent $F$-crystals]
\label{LocSysLaurentFCrys}
Let $X$ be a bounded $p$-adic formal scheme over $\mathbf{Z}_p$ with generic fibre $X_\eta$. Then extension of scalars gives an equivalence
\[ D_{\perf}(X_\Prism, \mathcal{O}_\Prism[1/\mathcal{I}_\Prism]^\wedge_p)^{\varphi=1} \simeq D_{\perf}(X_\Prism, \mathcal{O}_{\Prism,\perf}[1/\mathcal{I}_\Prism]^\wedge_p)^{\varphi=1}.\]
Moreover, the target identifies with the locally constant derived category
\[ D^{(b)}_{lisse}(X_\eta,\mathbf{Z}_p)\]
of the generic fibre $X_\eta$. 
\end{corollary}
\begin{proof}
For the first part, by the definition of both sides as limits, it suffices to show the following: for any bounded prism $(A,I)$ with perfection $(B,J)$, the pullback
\[ D_{\perf}(A[1/I]^\wedge_p)^{\varphi=1} \to D_{\perf}(B[1/J]^\wedge_p)^{\varphi=1}\]
is an equivalence. By devissage, it is enough to show the same mod $p$, where it follows from Proposition~\ref{BanachPerfFrobMod}. The identification with $D^{(b)}_{lisse}(X_\eta,\mathbf{Z}_p)$ then follows from Example~\ref{PerfdLaurentCrys} and descent.
\end{proof}

\begin{corollary}[$\mathbf{Q}_p$-local systems as Laurent $F$-crystals up to isogeny]
\label{LocSysLaurentFCrysIsog}
Let $X$ be a bounded $p$-adic formal scheme with generic fibre $X_\eta$. Then there is a natural equivalence
\[ \mathrm{Vect}(X_\Prism, \mathcal{O}_\Prism[1/\mathcal{I}_\Prism]^\wedge_p)^{\varphi=1}\simeq \mathrm{Loc}_{\mathbf{Z}_p}(X_\eta).\]
Inverting $p$ gives a natural equivalence
\[ \mathrm{Vect}(X_\Prism, \mathcal{O}_\Prism[1/\mathcal{I}_\Prism]^\wedge_p)^{\varphi=1} \otimes_{\mathbf{Z}_p} \mathbf{Q}_p \simeq \mathrm{Loc}_{\mathbf{Z}_p}(X_\eta) \otimes_{\mathbf{Z}_p} \mathbf{Q}_p.\]
\end{corollary}
\begin{proof}
The second part follows from the first one by inverting $p$. The first part follows by the same argument used in Corollary~\ref{LocSysLaurentFCrys}.
\end{proof}

\begin{remark}
In the second equivalence of Corollary~\ref{LocSysLaurentFCrys}, both sides admit natural enlargements:
\begin{enumerate}
\item We have a fully faithful embedding
\[ \mathrm{Vect}(X_\Prism, \mathcal{O}_\Prism[1/\mathcal{I}_\Prism]^\wedge_p)^{\varphi=1} \otimes_{\mathbf{Z}_p} \mathbf{Q}_p \subset \mathrm{Vect}(X_\Prism,  \mathcal{O}_\Prism[1/\mathcal{I}_\Prism]^\wedge_p[1/p])^{\varphi=1}.\]
If $X$ is an $\mathbf{F}_p$-scheme, this embedding is an equivalence as both categories are trivial. In the mixed characteristic case, however, this embedding is essentially never an equivalence. Indeed, consider $X=\mathrm{Spf}(\mathcal{O}_C)$ for $C/\mathbf{Q}_p$ a complete and algebraically closed extension. In this case, $X_\Prism$ has an initial object $(A=A_{\inf}(\mathcal{O}_C),I)$ given by the perfect prism corresponding to $\mathcal{O}_C$. By Corollary~\ref{LocSysLaurentFCrys}, the category of Laurent $F$-crystals is trivial: taking $\varphi$-fixed points identifies the category $\mathrm{Vect}(X_\Prism, \mathcal{O}_\Prism[1/\mathcal{I}_\Prism]^\wedge_p)^{\varphi=1}$ with $\mathrm{Vect}(\mathbf{Z}_p)$. Consequently, the category $\mathrm{Vect}(X_\Prism, \mathcal{O}_\Prism[1/\mathcal{I}_\Prism]^\wedge_p)^{\varphi=1} \otimes_{\mathbf{Z}_p} \mathbf{Q}_p$ identifies with $\mathrm{Vect}(\mathbf{Q}_p)$. On the other hand, the category $\mathrm{Vect}(X_\Prism,  \mathcal{O}_\Prism[1/\mathcal{I}_\Prism]^\wedge_p[1/p])^{\varphi=1}$ identifies with the category $\mathrm{Vect}(W(C^\flat)[1/p])^{\varphi=1}$ of $F$-isocrystals on $C^\flat$; the latter is described by the Dieudonn\'e-Manin classification, and is much larger than $\mathrm{Vect}(\mathbf{Q}_p)$. In fact, the essential image of the latter inside $\mathrm{Vect}(W(C^\flat)[1/p])^{\varphi=1}$ is exactly the full subcategory of ``\'etale'' $F$-isocrystals, i.e., those $F$-isocrystals where all the Frobenius eigenvalues have slope $0$.

\item We have a fully faithful embedding
\[\mathrm{Loc}_{\mathbf{Z}_p}(X_\eta) \otimes_{\mathbf{Z}_p} \mathbf{Q}_p \subset \mathrm{Loc}_{\mathbf{Q}_p}(X_\eta)\]
into the category of (pro-\'etale) $\mathbf{Q}_p$-local systems on the generic fibre $X_\eta$. If $X=\mathrm{Spf}(\mathcal{O}_K)$ for a complete extension $K/\mathbf{Q}_p$, then this embedding is an equivalence as an \'etale cover of $\mathrm{Spa}(K,\mathcal{O}_K)$ admits sections over a finite \'etale cover of $K$. In general, however, this embedding is not essentially surjective. For instance, if $X$ is the $p$-completion of $\mathbf{P}^1_{\mathbf{Z}_p}$, then de Jong has constructed \cite[\S 7]{deJongEtalePi1} a $\mathbf{Q}_p$-local system $L$ on $X_\eta$ with monodromy group $\mathrm{SL}_2(\mathbf{Q}_p)$; no such $L$ can arise by inverting $p$ in a $\mathbf{Z}_p$-local system.

\end{enumerate}
We do not know a natural generalization of Corollary~\ref{LocSysLaurentFCrysIsog} that accommodates the preceding enlargements.
\end{remark}

\begin{remark}
For $X=\mathrm{Spf}(\mathcal{O}_K)$ with $K/\mathbf{Q}_p$ a discretely valued extension with perfect residue field, Corollary~\ref{LocSysLaurentFCrysIsog} was proven independently by Zhiyou Wu \cite{WuPrismaticGalois}. This proof was extended to the case of smooth formal schemes over $\mathcal{O}_K$ by \cite{MinWangLaurent}.
\end{remark}

\begin{remark}[Local systems on the special fibre]
A simpler version of the reasoning used to establish Corollary~\ref{LocSysLaurentFCrys} shows that there is a natural equivalence
\[ D_{\perf}(X_\Prism,\mathcal{O}_\Prism)^{\varphi=1} \simeq D^b_{lisse}(X_{p=0}, \mathbf{Z}_p),\]
where the scheme $X_{p=0} := X \times_{\mathrm{Spf}(\mathbf{Z}_p)} \mathrm{Spec}(\mathbf{F}_p)$.
\end{remark}

\newpage
\section{Prismatic $F$-crystals}

Given a $p$-adic formal scheme $X$, following Proposition~\ref{VBDescentPrism}, we have the notion of a vector bundle on $(X_\Prism,\mathcal{O}_\Prism)$:
\[ \mathrm{Vect}(X_\Prism, \mathcal{O}_\Prism) := \lim_{(A,I) \in X_\Prism} \mathrm{Vect}(A).\]
In this section, we introduce the notion of a prismatic $F$-crystal (Definition~\ref{PrismaticFCrysDef}), which is a vector bundle as above equipped with Frobenius structure that is reminiscent of the notion of a shtuka. We then give some examples, and construct certain ``realization'' functors (Constructions~\ref{EtaleRealize} and \ref{CrysdRRealize}), extracting concrete invariants from the somewhat elaborate structure of a prismatic $F$-crystal.

\begin{definition}[Prismatic $F$-crystals of vector bundles]
\label{PrismaticFCrysDef}
For any $p$-adic formal scheme $X$, let $\mathrm{Vect}^{\varphi}(X_\Prism,\mathcal{O}_\Prism)$ denote the category of {\em prismatic $F$-crystals (of vector bundles) on $X_\Prism$}, i.e., vector bundles $\mathcal{E}$ on $(X_\Prism,\mathcal{O}_\Prism)$ equipped with an identification $\varphi_{\mathcal{E}}: \varphi^* \mathcal{E}[1/\mathcal{I}_\Prism] \simeq \mathcal{E}[1/\mathcal{I}_\Prism]$. If $\varphi_{\mathcal{E}}$ carries $\varphi^* \mathcal{E}$ into $\mathcal{E}$, then we say that $(\mathcal{E},\varphi_{\mathcal{E}})$ is {\em effective}. Taking tensor products yields a (rigid) symmetric monoidal structure on $\mathrm{Vect}^{\varphi}(X_\Prism,\mathcal{O}_\Prism)$.

More generally, we make a similar definition for $\mathrm{Vect}^{\varphi}(X_\Prism, \mathcal{O}')$ where $\mathcal{O}'$ is a sheaf of $\mathcal{O}_\Prism$-algebras equipped with a compatible Frobenius. Similary, if $(A,I)$ is a prism, one has an evident category $\mathrm{Vect}^{\varphi}(A)$ of prismatic $F$-crystals of vector bundles on $A$. 
\end{definition}

\begin{remark}[Prismatic $F$-crystals of perfect complexes]
For a $p$-adic formal scheme $X$, there is an evident notion of a prismatic $F$-crystal of perfect complexes on $X_\Prism$: it is given by an object $\mathcal{E} \in D_{\perf}(X_\Prism,\mathcal{O}_\Prism)$ equipped with an identification $\varphi_{\mathcal{E}}: \varphi^* \mathcal{E}[1/\mathcal{I}_\Prism] \simeq \mathcal{E}[1/\mathcal{I}_\Prism]$. Any prismatic $F$-crystal of vector bundles gives a prismatic $F$-crystal in perfect complexes; but the derived notion also accommodates other examples, such as prismatic $F$-crystals of vector bundles on $(X_\Prism,\mathcal{O}_\Prism/p^n)$ when $X$ is $p$-torsionfree. While this notion does not play a serious role in this paper, we shall use it to make some remarks.
\end{remark}

\begin{example}[Breuil-Kisin modules]
Let $K/\mathbf{Q}_p$ be a discretely valued extension with perfect residue field $k$. Choose a uniformizer $\pi \in \mathcal{O}_K$. As in Example~\ref{BKAinfPrismOK} (1), we obtain a Breuil-Kisin prism  $(A,I) = (\mathfrak{S},(E(u))) \in \mathrm{Spf}(\mathcal{O}_K)_\Prism$. An object of the category $\mathrm{Vect}^{\varphi}(A)$ is traditionally called a Breuil-Kisin module, and was studied in depth in \cite{Kisin}. There is an obvious realization functor from prismatic $F$-crystals over $\mathcal{O}_K$ towards Breuil-Kisin modules. This functor as well as the relationship of either side with crystalline Galois representations will be discussed further in \S \ref{sec:BKCrys}. 
\end{example}

\begin{example}[Laurent $F$-crystals as prismatic $F$-crystals]
The category $\mathrm{Vect}^{\varphi}(X_\Prism, \mathcal{O}_\Prism[1/\mathcal{I}_\Prism]^{\wedge}_p)$ coincides with the category $\mathrm{Vect}(X_\Prism,\mathcal{O}_\Prism[1/\mathcal{I}_\Prism]^{\wedge}_p)^{\varphi=1}$ from Definition~\ref{LaurentFCrysDef}. Indeed, $\mathcal I_\Prism$ is already inverted in the sheaf of rings $\mathcal{O}_\Prism[1/\mathcal{I}_\Prism]^{\wedge}_p$.
\end{example}

\begin{example}[Breuil-Kisin twists]
\label{ex:BKT}
In this example, we explain the notion of a Breuil--Kisin twist, which plays the role of the Tate twist in the world of prismatic $F$-crystals. For more details including explicit constructions over certain perfectoid rings, we refer to \cite[Example 4.2, Example 4.24]{BMS1}, \cite[\S 4.9]{DrinfeldPrismatization}, and \cite[\S 2]{BhattLurieAPC}.

Take $X=\mathrm{Spf}(\mathbf{Z}_p)$, so $X_\Prism$ is the category of all bounded prisms. There is a naturally defined prismatic $F$-crystal $(\mathcal{O}_\Prism\{1\},\varphi)$, called the Breuil-Kisin twist, on $X$ whose underlying $\mathcal{O}_\Prism$-module is invertible; by pullback, one obtains a similar prismatic $F$-crystal over any bounded $p$-adic formal scheme. Informally, the invertible $\mathcal{O}_\Prism$-module is given by the following formula:
\[ \mathcal{O}_\Prism\{1\} := \mathcal{I}_\Prism \otimes \varphi^* \mathcal{I}_\Prism \otimes (\varphi^2)^* \mathcal{I}_\Prism \otimes .... = \bigotimes_{i \geq 0} (\varphi^i)^* \mathcal{I}_\Prism.\]
More precisely, to make sense of the infinite tensor product, one observes that for any bounded prism $(A,I)$, one first checks\footnote{Alternately, one can also show that $(\varphi^n)^* I$ is canonically trivial modulo $I^{p^n}$ using Joyal's operations; see the footnote appearing in the proof of Lemma~\ref{PrismaticLogCech}.} that for $n\geq r$, the invertible $A$-module $(\varphi^n)^* I$ is canonically trivialized (by the generator $p$) after base change along $A \to A/I_r$, where $I_r = \prod_{i = 0}^{r-1} \varphi^{i}(I) \subset A$; the infinite tensor product then makes sense in view of the natural equivalence $\mathcal{P}\mathrm{ic}(A) \simeq \lim_r \mathcal{P}\mathrm{ic}(A/I_r)$ of groupoids (which can be proven as in \cite[Proposition 2.2.12]{BhattLurieAPC}). It follows that there is a natural isomorphism
\[ \varphi^* \mathcal{O}_\Prism\{1\} \simeq \mathcal{I}_\Prism^{-1} \mathcal{O}_\Prism\{1\},\]
which gives the $F$-crystal structure. 
\end{example}

\begin{example}[Gauss-Manin prismatic $F$-crystals]
\label{ex:GM}
Let $f:X \to Y$ be a proper smooth map of $p$-adic formal schemes. Then $\mathcal{E}_X := Rf_* \mathcal{O}_\Prism$ is naturally an effective prismatic $F$-crystal in perfect complexes on $Y_\Prism$: the perfect crystal property follows from the Hodge-Tate comparison \cite[Theorem 1.8 (2)]{BhattScholzePrisms} and flat base change for coherent cohomology, while the effective $F$-crystal structure comes from the isogeny theorem \cite[Theorem 1.15 (4)]{BhattScholzePrisms}.  One can then pass to cohomology and obtain prismatic $F$-crystals in vector bundles under favorable conditions. For instance, if $Y = \mathrm{Spf}(\mathcal{O}_K)$ for a complete extension $K/\mathbf{Q}_p$ with residue field $k$  and if the special fibre $X_k$ has torsionfree crystalline cohomology, then each cohomology sheaf of $\mathcal{E}_X$  gives a prismatic $F$-crystal in vector bundles by \cite[Theorem 14.5]{BMS1}. The $F$-crystal from Example~\ref{ex:BKT} then admits a simple geometric description: we have $\mathcal{O}_\Prism\{-1\} \simeq H^2(\mathcal{E}_{\mathbf{P}^1})$ via the first Chern class map, see \cite{BhattLurieAPC}.
\end{example}

\begin{example}[Prismatic $F$-crystals on schemes of characteristic $p$]
\label{ex:CrystallineFCrys}
Let $X$ be a quasi-syntomic $\mathbf{F}_p$-scheme. Then it is known that the category of crystals of vector bundles on $X_\Prism$ and $X_{\crys}$ are naturally identified in a Frobenius equivariant fashion: one reduces by descent to the case where $X=\mathrm{Spf}(R)$ for $R$ qrsp, where the claim follows from the crystalline comparison $\Prism_R \simeq A_{\crys}(R)$ for prismatic cohomology. Moreover, we have $\mathcal{I}_\Prism = p\mathcal{O}_\Prism$ by the irreducibility lemma on distinguished elements. Consequently, the notion of a $F$-crystal on $X_\Prism$ coincides with the classical notion of an $F$-crystal on $X_{\crys}$.
\end{example}

\begin{construction}[The \'etale realization]
\label{EtaleRealize}
For any $p$-adic formal scheme $X$, the $p$-completed base change gives a symmetric monoidal functor
\[ T:\mathrm{Vect}^{\varphi}(X_\Prism,\mathcal{O}_\Prism) \to \mathrm{Vect}(X_\Prism,\mathcal{O}_\Prism[1/\mathcal{I}_\Prism]^\wedge_p)^{\varphi=1} \simeq \mathrm{Loc}_{\mathbf{Z}_p}(X_\eta),\]
where we use Corollary~\ref{LocSysLaurentFCrys} for the last isomorphism. We refer to this functor as the \'etale realization functor. For future reference, we remark that this functor makes sense not only for prismatic $F$-crystals of vector bundles, but in fact for prismatic $F$-crystals of perfect complexes provided we replace the target with the derived category from Corollary~\ref{LocSysLaurentFCrys}.
\end{construction}

\begin{example}[Relating Breuil-Kisin and Tate twists]
Using the $q$-logarithm \cite{ALBqLog} and descent (or directly the prismatic logarithm \cite{BhattLurieAPC}), one can show that the \'etale realization carries the Breuil-Kisin twists from Example~\ref{ex:BKT} to usual Tate twists, i.e., that $T(\mathcal{O}_\Prism\{i\}) = \mathbf{Z}_p(i)$ for any $i$ and any bounded $p$-adic formal scheme $X$. More generally, in the context of Example~\ref{ex:GM}, the Artin-Schreier sequence on $X_{\eta,proet}$ can be used to show that $T$ commutes with proper smooth pushforwards.
\end{example}

\begin{example}[$F$-crystals over a qrsp ring]
\label{FCrysLaurentFCrysPerfectoid}
Let $X = \mathrm{Spf}(R)$ for $R$ a qrsp ring. Then $X_\Prism$ has an initial object determined by the prism $(\Prism_R,I)$ of $R$. Consequently, we can identify the category of $F$-crystals explicitly:
\[  \mathrm{Vect}^{\varphi}(X_\Prism,\mathcal{O}_\Prism) = \mathrm{Vect}^{\varphi}(\Prism_R) := \{ (E,\varphi) \mid E \in \mathrm{Vect}(\Prism_R), \  \varphi:(F^* E)[1/I] \simeq E[1/I] \}.\]
Combining this with the analogous equivalence for Laurent $F$-crystals in Example~\ref{PerfdLaurentCrys} as well as Lemma~\ref{PrismQRSPTrans} below, we conclude that the \'etale realization functor
\[ T:\mathrm{Vect}^{\varphi}(X_\Prism,\mathcal{O}_\Prism) \to \mathrm{Vect}(X_\Prism,\mathcal{O}_\Prism[1/\mathcal{I}_\Prism]^\wedge_p)^{\varphi=1} \simeq \mathrm{Loc}_{\mathbf{Z}_p}(X_\eta)\]
is faithful if $R$ is $p$-torsionfree.
\end{example}

\begin{lemma}
\label{PrismQRSPTrans}
Let $R$ be a $p$-torsionfree qrsp ring with prism $(\Prism_R,I)$. Then $\Prism_R \to \Prism_R[1/I]^{\wedge}$ is injective, and hence the same holds true on tensoring with any finite projective $\Prism_R$-module $M$.
\end{lemma}
\begin{proof}
By using the derived $p$-completeness of the cofibre of the map and derived Nakayama, it is enough to show that
\[ \Prism_R/p \to \Prism_R/p[1/I]\]
is injective. This follows from the fact that $(\Prism_R,I)$ is a transversal prism (i.e., $(p,I)$ form a regular sequence, see \cite{ALBPrismaticDieudonne} for the terminology), which itself is a consequence of the Hodge-Tate comparison and the $p$-torsionfreeness of $R$.
\end{proof}

\begin{construction}[The crystalline and de Rham realizations]
\label{CrysdRRealize}
Let $X$ be a quasi-syntomic $p$-adic formal scheme, assumed $\mathbf{Z}_p$-flat for simplicity. By Example~\ref{ex:CrystallineFCrys}, the category of prismatic $F$-crystals on $X_{p=0}$ identifies with the category $\mathrm{Vect}^{\varphi}(X_{p=0,crys})$ of $F$-crystals on the crystalline site $X_{p=0,crys}$. Pullback along $X_{p=0} \to X$ thus yields the ``crystalline realization'' functor
\[ D_{\crys}:\mathrm{Vect}^{\varphi}(X_\Prism,\mathcal{O}_\Prism) \to \mathrm{Vect}^{\varphi}(X_{p=0,crys}).\]
Noting that crystals on $X_{\crys}$ and $X_{p=0,crys}$ are identified (as $X_{p=0} \subset X$ is a pro-PD thickening), forgetting the Frobenius also yields the ``de Rham realization'' functor
\[ D_{\dR}:\mathrm{Vect}^{\varphi}(X_\Prism,\mathcal{O}_\Prism) \to \mathrm{Vect}(X_{\crys}).\]
If $X$ is formally smooth over some base $Y$, then we can further pass to relative crystalline sites to obtain the ``relative de Rham realization'' functor
\[ D_{dR,Y}:\mathrm{Vect}^{\varphi}(X_\Prism,\mathcal{O}_\Prism) \to \mathrm{Vect}((X/Y)_{\crys}) \simeq \mathrm{Vect}^{\nabla}_{nil}(X/Y),\]
where the target denotes the category of vector bundles on $X$ equipped with a flat connection relative to $Y$ such that the connection is topologically quasi-nilpotent. 
\end{construction}

\newpage
\section{Prismatic $F$-crystals over $\mathrm{Spf}(\mathcal{O}_K)$: formulation of the main theorem}
\label{sec:FCrysOK}

Fix a complete discretely valued extension $K/\mathbf{Q}_p$ with perfect residue field $k$.   In this section, we first explain why prismatic $F$-crystals over $\mathrm{Spf}(\mathcal{O}_K)$ give rise to crystalline Galois representations upon \'etale realization (Proposition~\ref{FCrystoCrysGal}). Using this, we then formulate our main comparison theorem with crystalline Galois representations (Theorem~\ref{MainThm}) as an equivalence of categories; the full faithfulness is proven in this section using the easy direction of Fargues' classification of $F$-crystals over $\Prism_{\mathcal{O}_C}$ (Theorem~\ref{FarguesShtukaClassify}; see also Remark~\ref{FFviaBK}), while the essential surjectivity is the subject of \S \ref{sec:EssSurj}.

\begin{notation}[The $A_{\inf}$-prism and associated period rings]
Fix a completed algebraic closure $C/K$. Using the unique $F$-equivariant splitting\footnote{To construct this splitting explictly, choose $n \gg 0$ such that the $n$-fold Frobenius on $\mathcal{O}_K/p$ factors as $\mathcal{O}_K/p \xrightarrow{can} k \xrightarrow{\varphi_n} \mathcal{O}_K/p$; the splitting is then given by $k \xrightarrow{\varphi_n \circ \varphi^{-n}} \mathcal{O}_K/p$.} $k \to \mathcal{O}_K/p$ of $\mathcal{O}_K/p \to k$, we regard $\mathcal{O}_K$ and $\mathcal{O}_C$ (as well as all subsequently appearing rings) as $W(k)$-algebras. Let $X = \mathrm{Spf}(\mathcal{O}_K)$ and $Y=\mathrm{Spf}(\mathcal{O}_C)$; write $Y^{\bullet/X}$ denote the Cech nerve of $Y \to X$.  Associated to this data, one has some standard period rings and elements, that we introduce next. 

\begin{itemize}
\item  Write $A_{\inf} = A_{\inf}(\mathcal{O}_C)$. Choose a compatible system $\underline{\epsilon}=(1,\epsilon_p,\epsilon_{p^2},...)$ of $p$-power roots of $1$ in $C$ gives rise to the standard elements $q=[\epsilon]$, $\mu=q-1$, and $\tilde{\xi} = [p]_q$ of $A_{\inf}$. Moreover, we have a natural surjection $\tilde{\theta}:A_{\inf} \to \mathcal{O}_C$ with kernel $(\tilde{\xi})$,  that we use to identify $(A_{\inf},(\tilde{\xi}))$ as the perfect prism attached to $\mathcal{O}_C$, so $A_{\inf} = \Prism_{\mathcal{O}_C}$. We shall refer to the point of $\mathrm{Spec}(A_{\inf})$ determined by $A_{\inf} \xrightarrow{\tilde{\theta}} \mathcal{O}_C \to C$ as the Hodge-Tate point. Note that $\mu \in A_{\inf}$ maps to the non-zero element $\epsilon_p - 1 \in C$ in the residue field at the Hodge-Tate point.

\item Write $A_{\crys} = D_{(\varphi^{-1}([p]_q))}(A_{\inf}) = A_{\inf}\{\frac{[p]_q}{p}\}$ for the PD-envelope of $\theta=\tilde{\theta}\circ \varphi$, $B_{\crys}^+ = A_{\crys}[1/p]$, and $B_{\crys} = A_{\crys}[1/\mu]$\footnote{Observe that $B_{\crys} = B_{\crys}^+[1/\mu]$, i.e., $p$ is invertible in $A_{\crys}[1/\mu]$. Indeed, we have $(q-1)^{p-1} \equiv [p]_q \mod p A_{\inf}$, so $(q-1)^{p-1}/p \in A_{\crys}$, whence inverting $\mu=q-1$ also inverts $p$.}; in prismatic terms, the $A_{\inf}$-algebra $A_{\crys}$ identifies with the $\Prism_{\mathcal{O}_C}$-algebra $ \Prism_{\mathcal{O}_C}\{\frac{[p]_q}{p}\} = \Prism_{\mathcal{O}_C/p}$, see e.g.~\cite[Corollary 2.39, Proposition 7.10]{BhattScholzePrisms}.

\item Write $B_{\dR}^+$ for the formal completion of $A_{\inf}[1/p]$ at $\tilde{\theta}$; this is a complete DVR with residue field $C$, uniformizer $[p]_q$ and fraction field $B_{\dR} = B_{\dR}^+[1/[p]_q]$.  We have a natural $A_{\inf}$-algebra map $\varphi^\ast B_{\crys}^+ \to B_{\dR}^+$. This maps $\mu\in B_{\crys}^+$ to $\varphi(\mu)\in B_{\dR}^+$; since $\tilde{\theta}(\mu) \neq 0 \in C$, the image $\varphi(\mu)=[p]_q \mu$ has the form $[p]_q u$ for a unit $u$, so we also have an induced $A_{\inf}$-algebra map $\varphi^\ast B_{\crys} \to B_{\dR}$.
\end{itemize}
\end{notation}

Our goal is to understand the category $\mathrm{Vect}^{\varphi}(X_\Prism,\mathcal{O}_\Prism)$ of prismatic $F$-crystals on $X$. We shall do so via descent along $Y \to X$. For this reason, we shall need the easy direction of the following result:

\begin{theorem}[Fargues]
\label{FarguesShtukaClassify}
The category $\mathrm{Vect}^{\varphi}(Y_\Prism,\mathcal{O}_\Prism)$ is identified with the category of pairs $(T,F)$, where $T$ is a finite free $\mathbf{Z}_p$-module and $F$ is a $B_{\dR}^+$-lattice in $T \otimes_{\mathbf{Z}_p} B_{\dR}$.
\end{theorem}

We content ourselves by explaining the construction of $(T,F)$ attached to $\mathcal{E} \in \mathrm{Vect}^{\varphi}(Y_\Prism,\mathcal{O}_\Prism)$. For an elementary proof of full faithfulness (which is the only part we use), we refer to \cite[Remark 4.29]{BMS1}. For the essential surjectivity, see \cite[Lecture XIV]{Berkeley}.

\begin{proof}[Construction of the functor in Fargues' theorem]
Fix $\mathcal{E} \in \mathrm{Vect}^{\varphi}(Y_\Prism,\mathcal{O}_\Prism)$. As $Y_\Prism$ has a final object determined by $A_{\inf} = \Prism_{\mathcal{O}_C}$, we can regard $\mathcal{E}$ as a BKF-module $(M,\varphi_M)$ over $A_{\inf}$ in the sense of \cite[\S 4.3]{BMS1}. The \'etale realization $T(\mathcal{E})$ is the finite free $\mathbf{Z}_p$-module $T := (M \otimes_{A_{\inf}} W(C^\flat))^{\varphi=1}$. By Artin-Schreier theory (see Corollary~\ref{LocSysLaurentFCrysIsog}), one has a natural isomorphism $M \otimes_{A_{\inf}} W(C^\flat) \simeq T \otimes_{\mathbf{Z}_p} W(C^\flat)$ extending the identity on $T$ on $(-)^{\varphi=1}$. Moreover, one can show (\cite[Lemma 4.26]{BMS1}) that this isomorphism restricts to an isomorphism $M[1/\mu] \simeq T \otimes_{\mathbf{Z}_p} A_{\inf}[1/\mu]$. As $\mu$ is invertible in the residue field at the Hodge-Tate point, we obtain an isomorphism $M \otimes_{A_{\inf}} B_{\dR}^+ \simeq T \otimes_{\mathbf{Z}_p} B_{\dR}^+$. The Frobenius on $M$ then yields the $B_{\dR}^+$-lattice $F := \varphi^* M \otimes_{A_{\inf}} B_{\dR}^+$ in $T \otimes_{\mathbf{Z}_p} B_{\dR}$. Thus, we obtain the pair $(T,F)$ mentioned in the theorem.
\end{proof}

We begin  our study of prismatic $F$-crystals by noting that the resulting Galois representations are crystalline.

\begin{proposition}[Prismatic $F$-crystals on $\mathcal O_K$ give crystalline Galois representations]
\label{FCrystoCrysGal}
Let $\mathcal{E} \in \mathrm{Vect}^{\varphi}(X_\Prism,\mathcal{O}_\Prism)$. Then $T(\mathcal{E})[1/p]$ is a crystalline Galois representation.
\end{proposition}
\begin{proof}
Fix a prismatic $F$-crystal $\mathcal{E}$ over $X$. Pulling back to $Y$ and then further to $Y_{p=0}$, we have a natural identification
\[ \mathcal{E}(\Prism_{\mathcal{O}_C}) \otimes_{\Prism_{\mathcal{O}_C}} \Prism_{\mathcal{O}_C/p} \simeq \mathcal{E}(\Prism_{\mathcal{O}_C/p}).\]
Identifying $\Prism_{\mathcal{O}_C} = A_{\inf}$ and $\Prism_{\mathcal{O}_C/p} = A_{\crys}$, we can write this as
\begin{equation}
\label{eq:lat1}
\mathcal{E}(\Prism_{\mathcal{O}_C}) \otimes_{A_{\inf}} A_{\crys} \simeq \mathcal{E}(\Prism_{\mathcal{O}_C/p}).
 \end{equation}
Inverting $\mu$ in the tensor base $A_{\inf}$ gives
\[ \mathcal{E}(\Prism_{\mathcal{O}_C})[1/\mu]\otimes_{A_{\inf}[1/\mu]} B_{\crys} \simeq \mathcal{E}(\Prism_{\mathcal{O}_C/p})[1/\mu].\]
As explained after Theorem~\ref{FarguesShtukaClassify}, by \cite[Lemma 4.26]{BMS1}, the object $\mathcal{E}(\Prism_{\mathcal{O}_C})[1/\mu]$ identifies with $T(\mathcal{E}) \otimes_{\mathbf{Z}_p} A_{\inf}[1/\mu]$, so we can rewrite the above as a natural isomorphism
\[T(\mathcal{E})\otimes_{\mathbf{Z}_p} B_{\crys} \simeq  \mathcal{E}(\Prism_{\mathcal{O}_C/p})[1/\mu].\]
The rest follows from Dwork's Frobenius trick. More precisely, observe that $\mathcal{O}_K/p \to k$ is an infinitesimal thickening in characteristic $p$, so the $n$-fold Frobenius on $\mathcal{O}_K/p$ factors over $k$. It follows that if $\mathcal{E'}$ is any crystalline crystal on $\mathrm{Spec}(\mathcal{O}_K/p)$, then $(\varphi^n)^* \mathcal{E'}$ identifies with the pulback of $\mathcal{E'}|_{\mathrm{Spec}(k)_\Prism}$ along the map $\varphi_n:k \to \mathcal{O}_K/p$ induced by $\varphi^n$ on $\mathcal{O}_K/p$ for $n \gg 0$. Pulling back further to $\mathcal{O}_C/p$ and using the $F$-crystal structure to drop the Frobenius pullback after inverting $p$, we learn that the right side in the isomorphism above can be rewritten as 
\[ \mathcal{E}(\Prism_{\mathcal{O}_C/p})[1/\mu] \simeq \mathcal{E}(W(k)) \otimes_{W(k)} B_{\crys}.\]
Putting everything together, we get a canonical isomorphism
\[ T(\mathcal{E}) \otimes_{\mathbf{Z}_p} B_{\crys} \simeq \mathcal{E}(W(k)) \otimes_{W(k)} B_{\crys}.\]
As this isomorphism is $G_K$-equivariant, it follows that $T(\mathcal{E})[1/p]$ is a crystalline $G_K$-representation.
\end{proof}

\begin{remark}[A compatibility of lattices]
\label{CompatLattice}
Keep notation as in the proof of Proposition~\ref{FCrystoCrysGal}. Base changing the final isomorphism along $\varphi^\ast B_{\crys} \to B_{\dR}$ gives an identification
\[ T(\mathcal{E}) \otimes_{\mathbf{Z}_p} B_{\dR} \simeq \varphi^\ast \mathcal{E}(W(k)) \otimes_{W(k)} B_{\dR}.\]
Thus, $F' := \varphi^\ast \mathcal{E}(W(k)) \otimes_{W(k)} B_{\dR}^+$ gives a $B_{\dR}^+$-lattice in $T(\mathcal{E}) \otimes_{\mathbf{Z}_p} B_{\dR}$. On the other hand, we also have the lattice $F = \varphi^* \mathcal{E}(\Prism_{\mathcal{O}_C}) \otimes_{\Prism_{\mathcal{O}_C}} B_{\dR}^+$, as explained following Theorem~\ref{FarguesShtukaClassify}. It follows from the constructions that these lattices are identical.
\end{remark}

Using Theorem~\ref{FarguesShtukaClassify}, we can construct $G_K$-equivariant $F$-crystals on $Y$ starting with de Rham Galois representations.

\begin{construction}[From de Rham Galois representations to $F$-crystals over $\mathcal{O}_C$]
\label{dRGaloisShtuka}
Let $T$ be a $\mathbf{Z}_p$-lattice in a de Rham Galois representation $V=T[1/p]$ of $G_K$. Then we have a natural $G_K$-equivariant filtered isomorphism
\[ V \otimes_{\mathbf{Q}_p} B_{\dR} \simeq D_{\dR}(V) \otimes_K B_{\dR}.\]
In particular, $D_{\dR}(V) \otimes_K B_{\dR}^+$ is a $B_{\dR}^+$-lattice in $V \otimes_{\mathbf{Q}_p} B_{\dR}$. By Theorem~\ref{FarguesShtukaClassify} as well as the compatibility in Remark~\ref{CompatLattice}, the assignment
\[ T \mapsto (T, D_{\dR}(V) \otimes_K B_{\dR}^+)\]
gives the middle horizontal arrow in a commutative diagram
\[ \xymatrix{\mathrm{Vect}^{\varphi}(X_\Prism,\mathcal{O}_\Prism) \ar[r] \ar[d] &  \mathrm{Vect}^{\varphi}(Y_\Prism,\mathcal{O}_\Prism) \ar@{=}[d] \\
 \mathrm{Rep}_{\mathbf{Z}_p}^{\dR}(G_K) \ar[r]^-{T \mapsto (T, D_{\dR}(T[1/p]) \otimes_K B_{\dR}^+)} \ar[d] & \{ (T,F) \mid T \in \mathrm{Vect}(\mathbf{Z}_p),\ F \subset T \otimes_{\mathbf{Z}_p} B_{\dR}  \text{ a }B_{\dR}^+{-lattice} \} \ar[d] \\
  \mathrm{Rep}_{\mathbf{Z}_p}(G_K) \simeq  \mathrm{Vect}(X_\Prism,\mathcal{O}_\Prism[1/\mathcal{I}_\Prism]^\wedge_p)^{\varphi=1} \ar[r] & \mathrm{Vect}(Y_\Prism,\mathcal{O}_\Prism[1/\mathcal{I}_\Prism]^\wedge_p)^{\varphi=1} }\]
In particular, as the bottom vertical arrow on the left is fully faithful, it follows that for any pair $\mathcal{M},\mathcal{N} \in \mathrm{Vect}^{\varphi}(X_\Prism,\mathcal{O}_\Prism)$ and any map $g:T(\mathcal{M}) \to T(\mathcal{N})$ in $\mathrm{Vect}(X_\Prism,\mathcal{O}_\Prism[1/\mathcal{I}_\Prism]^\wedge_p)^{\varphi=1}$, there is a unique map $f_Y:\mathcal{M}|_Y \to \mathcal{N}|_Y$ in  $\mathrm{Vect}^{\varphi}(Y_\Prism,\mathcal{O}_\Prism)$ with $T(f_Y) = g|_Y$; note that this assertion only uses the full faithfulness of Theorem~\ref{FarguesShtukaClassify} (i.e., full faithfulness of the equivalence appearing as the top right vertical arrow above).
\end{construction}

We can now formulate our main theorem:

\begin{theorem}[The main theorem]
\label{MainThm}
The \'etale realization functor
\[ T:\mathrm{Vect}^{\varphi}(X_\Prism,\mathcal{O}_\Prism) \to \mathrm{Rep}^{\crys}_{\mathbf{Z}_p}(G_K)\]
coming from Proposition~\ref{FCrystoCrysGal} is an equivalence of categories.
\end{theorem}

We prove the full faithfulness here using the preceding discussion on Galois representations attached to prismatic $F$-crystals; the essential surjectivity is proven in \S \ref{sec:EssSurj}.

\begin{proof}[Proof of full faithfulness in Theorem~\ref{MainThm}]
Faithfulness of $T$ over $X$ reduces to the analogous assertion over $Y$, which in turn follows by observing that the canonical map $M \to M \otimes_{A_{\inf}} A_{\inf}[1/\ker(\theta)]^\wedge_p = M\otimes_{A_{\inf}} W(C^\flat)$ is injective for any finite projective $A_{\inf}$-module $M$ by Lemma~\ref{PrismQRSPTrans}. Note that this argument also proves that the \'etale realization $T$ is faithful over any term of $Y^{\bullet/X}$.

For fullness, fix prismatic $F$-crystals $\mathcal{M}$ and $\mathcal{N}$ together with a map $g:T(\mathcal{M}) \to T(\mathcal{N})$ of the corresponding  objects in $\mathrm{Vect}(X_\Prism,\mathcal{O}_\Prism[1/\mathcal{I}_\Prism]^\wedge_p)^{\varphi=1}$ (or equivalenty the corresponding Galois representations). We must find a map $f:\mathcal{M} \to \mathcal{N}$ such that $T(f) = g$. First, by the last sentence of Construction~\ref{dRGaloisShtuka}, the map $g|_Y$ extends (necessarily uniquely) to a map $f_Y:\mathcal{M}|_Y \to \mathcal{N}|_Y$. The map $f_Y$ induces two a priori distinct maps $a,b:\mathcal{M}|_{Y \times_X Y} \to \mathcal{N}|_{Y \times_X Y}$ via pulback along either projection. We must show that $a=b$. But the induced maps $T(a),T(b):T(\mathcal{M})|_{Y \times_X Y} \to T(\mathcal{N})|_{Y \times_X Y}$ are the same as $T(f_Y)$ is induced from $g$ via pullback; the desired equality now follows from faithfulness of $T$ over $Y \times_X Y$.
\end{proof}

\begin{remark}[Full faithfulness via Breuil-Kisin prisms]
\label{FFviaBK}
The proof of full faithfulness in Theorem~\ref{MainThm} explained above relied on certain properties of Galois representations. An alternative purely prismatic argument can be given using Theorem~\ref{EtaleRealizeBK}, as we now briefly sketch. Let $\mathfrak{S}^{(\bullet)}$ denote the cosimplicial $\delta$-ring obtained by taking the Cech nerve of a Breuil-Kisin prism $(\mathfrak{S},E(u)) \in X_{\Prism}$; see Notation~\ref{NotBK} and Construction~\ref{PrismaticCechNerve}. Then we have compatible descent equivalences
\[ \mathrm{Vect}^{\varphi}(X_\Prism,\mathcal{O}_\Prism) \simeq \lim \mathrm{Vect}^{\varphi}(\mathfrak{S}^{(\bullet)}) \quad \text{and} \quad \mathrm{Vect}^{\varphi}(X_\Prism,\mathcal{O}_\Prism[1/\mathcal{I}_\Prism]^{\wedge}_p) \simeq \lim \mathrm{Vect}^{\varphi}(\mathfrak{S}^{(\bullet)}[1/I]^{\wedge}_p), \]
so it suffices to show that the base change functor
\[  \mathrm{Vect}^{\varphi}(\mathfrak{S}^{(\bullet)})  \to \mathrm{Vect}^{\varphi}(\mathfrak{S}^{(\bullet)}[1/I]^{\wedge}_p)\]
of cosimplicial categories induces a fully faithful in the inverse limit. Theorem~\ref{EtaleRealizeBK} implies that the above functor is fully faithful in cosimplicial degree $0$. Also, the induced functor in each cosimplicial degree is faithful as the map $\mathfrak{S}^{(i)} \to \mathfrak{S}^{(i)}[1/I]^{\wedge}_p$ is injective. It then follows formally that the limiting functor is indeed fully faithful.
\end{remark}

\begin{remark}[Full faithfulness fails with mod $p$ coefficients]
The analog of Theorem~\ref{MainThm} with mod $p$ coefficients is false. In fact,  the full faithfulness fails. The proof given above used crucially (via the full faithfulness in Theorem~\ref{FarguesShtukaClassify}) Kedlaya's theorem \cite{KedlayaAinf} that sections of vector bundles on $\mathrm{Spec}(A_{\inf})$ do not change if we remove the closed point; this assertion fails for vector bundles on $\mathrm{Spec}(A_{\inf}/p)$, which causes the problem.  In the next paragraph, we shall give an explicit example of an invertible object $(E,\varphi) \in \mathrm{Pic}^{\varphi}(X_\Prism,\mathcal{O}_\Prism/p)$ such that the \'etale realization of $E$ is the trivial local system while the underlying invertible $\mathcal{O}_\Prism/p$-module $E$ is non-trivial. In particular, full faithfulness of the \'etale realization fails. 

Take $K=\mathbf{Q}_p$. Consider the $F$-crystal $E := \mathcal{O}_\Prism\{p-1\}/p \in \mathrm{Vect}^{\varphi}(X_\Prism, \mathcal{O}_\Prism/p)$ coming from Example~\ref{ex:BKT}. The \'etale realization $T$ is symmetric monoidal, so $T(E) = \mathbf{Z}/p(p-1) \simeq \mathbf{Z}/p$ as $G_K$-representations. However,  $E$ is not isomorphic to the trivial $F$-crystal $\mathcal{O}_\Prism/p$. Indeed, we have
\[ \mathcal{O}_\Prism\{p\}/p \simeq \varphi^* \mathcal{O}_\Prism\{1\}/p \simeq \mathcal{I}^{-1}_\Prism/p \otimes \mathcal{O}_\Prism\{1\}/p,\]
whence $E \simeq \mathcal{I}_\Prism^{-1}/p$, so it suffices to show that $\mathcal{I}_\Prism/p$ is a non-trivial invertible $\mathcal{O}_\Prism/p$-sheaf. But this is already true after restriction to the Hodge-Tate locus, i.e., $\mathcal{I}_\Prism/p \otimes_{\mathcal{O}_\Prism/p} \overline{\mathcal{O}_\Prism}/p$ and $\overline{\mathcal{O}_\Prism}/p$ are not isomorphic as sheaves of $\mathcal{O}_\Prism$-modules; indeed, one can show that the Sen operator acts with weight $1$ on $\mathcal{I}_\Prism/p \otimes_{\mathcal{O}_\Prism/p} \overline{\mathcal{O}_\Prism}/p$ and weight $0$ on $\overline{\mathcal{O}_\Prism}/p$ (see \cite[\S 3.5]{BhattLurieAPC}).
\end{remark}

\newpage
\section{From crystalline Galois representations to prismatic $F$-crystals}
\label{sec:EssSurj}

The goal of this section is to prove the essential surjectivity assertion in Theorem~\ref{MainThm}, so we continue with the notation from \S \ref{sec:FCrysOK}. However, instead of working with the prismatic site $(X_\Prism,\mathcal{O}_\Prism)$, we shall switch to the quasi-syntomic site  $(X_{qsyn},\Prism_\bullet)$ (or equivalently the qrsp site  $(X_{qrsp},\Prism_\bullet)$)  from \S \ref{sec:PrismQsyn}; accordingly, appropriate period sheaves on the quasi-syntomic site are introduced in \S \ref{ss:periodqsyn}. Next, in \S \ref{ss:filtcrys}, we explain how a filtered $\varphi$-module gives rise to an $F$-crystal over the ``open unit disc'' version of $\Prism_\bullet$. In the critical \S \ref{ss:BoundDescent}, using the Beilinson fibre sequence from \cite{AMMNBFS}, we explain why this $F$-crystal over the ``open unit disc'' version of $\Prism_\bullet$ satisfies a boundedness condition at the boundary in the presence of some auxiliary lifting data that will ultimately be provided by the weak admissibility of the filtered $\varphi$-module. Armed with this boundedness, we can prove the promised theorem in \S \ref{ss:PfEssSur}.

\begin{example}
The $\mathcal{O}_K$-algebra $\mathcal{O}_C$ gives an object of $X_{qrsp}$. More generally, as $X_{qrsp}$ admits finite non-empty products (computed by $p$-completed tensor products of the underlying $\mathcal{O}_K$-algebras), the terms of the $p$-completed Cech nerve of $\mathcal{O}_K \to \mathcal{O}_C$ lie in $X_{qrsp}$; the rest  of this section only uses the first three terms of this Cech nerve.
\end{example}

\subsection{Some period sheaves on the quasi-syntomic site}
\label{ss:periodqsyn}

In this subsection, we introduce certain period sheaves on $X_{qrsp}$ that shall be important to our arguments. Roughly, the idea is to view $\Prism_\bullet$ as a structured collection of prisms analogous to $A_{\inf}$.

\begin{construction}[Period sheaves on the quasi-syntomic site that carry a Frobenius]
On $X_{qrsp}$, we shall use the following sheaves:
\begin{itemize}
\item The prismatic structure sheaf $\Prism_\bullet$: this is given by
\[ R \mapsto \Prism_R.\]
This sheaf carries an ideal sheaf $I \subset \Prism_\bullet$ given by passing to the ideal of the prism. Vector bundles over $\Prism_\bullet$ identify with prismatic crystals of vector bundles on $\mathrm{Spf}(\mathcal{O}_K)$.

\item The crystalline structure sheaf $A_{\crys}$, corresponding to the special fibre: this is given by
\[R \mapsto A_{\crys}(R) = \Prism_R \{I/p\}.\] 
Note that we have $\Prism_R\{I/p\} = \Prism_{R/p}$, explaining the name. Vector bundles over $A_{\crys}(\bullet)=\Prism_\bullet\{I/p\}$ identify with prismatic crystals of vector bundles on $\mathrm{Spec}(\mathcal{O}_K/p)$. 

\item The rational localization $\Prism_\bullet \langle I/p \rangle$: this is given by
\[ R \mapsto \Prism_R[I/p]^{\wedge}_p.\]
Modules over this will be closely related to filtered $F$-crystals on $X$ (Construction~\ref{FiltPhiCrys}).

\item The sheaf $\Prism_\bullet \{\varphi(I)/p\}$: this is given by
\[R \mapsto \Prism_R\{\varphi(I)/p\},\]
and can thus be regarded as the $\varphi$-pullback of the crystalline structure sheaf over $\Prism_\bullet$.

\item The \'etale structure sheaf $\Prism_\bullet[1/I]^{\wedge}_p$: this is given by
\[ R \mapsto \Prism_R[1/I]^{\wedge}_p.\]
Modules over this sheaf will be closely related to $\mathbf{Z}_p$-local systems on $R[1/p]$ (Remark~\ref{QSynSheavesModules}).

\end{itemize}
We have natural maps
\[ \Prism_\bullet[1/I]^{\wedge}_p \gets \Prism_\bullet \to \Prism_\bullet\{\varphi(I)/p\} \to \Prism_\bullet \langle I/p \rangle \to A_{\crys} = \Prism_\bullet \{I/p\}.\]
Moreover, the Frobenius $\varphi:\Prism_\bullet \to \Prism_\bullet$ carries $I$ to $\varphi(I)$, and hence induces a map 
\[ \tilde{\varphi}:A_{\crys} = \Prism_\bullet\{I/p\} \to \Prism_\bullet\{\varphi(I)/p\} \]
linear over $\varphi:\Prism_\bullet \to \Prism_\bullet$. In particular, one obtains an induced Frobenius endomorphism\footnote{Despite the notation, the endomorphism $\varphi$ of $\Prism \langle I/p \rangle$ defined as above does not underlie a $\delta$-structure, i.e., we need not have $\varphi = \mathrm{Frob} \mod p \Prism \langle I/p \rangle$; in fact, formally enlarging $\Prism_\bullet \langle I/p \rangle$ by forcing this congruence leads to $\Prism_\bullet \{I/p\}$. This issue does not arise for the other sheaves.} of all presheaves introduced above, also denoted  by $\varphi$. Mimicking Definition~\ref{PrismaticFCrysDef}, one then obtains a notion of an $F$-crystal over any of these rings. For instance, an $F$-crystal over $\Prism_\bullet$ is a vector bundle $\mathcal{M}$ over $\Prism_\bullet$ equipped with an isomorphism $\varphi_{\mathcal{M}}:(\varphi^* M) [1/I] \simeq M[1/I]$; we write $\mathrm{Vect}^{\varphi}(\Prism_\bullet)$ for the category of such $F$-crystals, and similarly over the other rings as well.
\end{construction}

\begin{remark}
\label{QSynSheavesModules}
The comparison result in Proposition~\ref{ComparePrismQSyn} gives equivalences
\[ \mathrm{Vect}^{\varphi}(X_\Prism, \mathcal{O}_\Prism) \simeq \mathrm{Vect}^{\varphi}(X_{qrsp},\Prism_\bullet) \quad \text{and} \quad \mathrm{Vect}^{\varphi}(X_\Prism, \mathcal{O}_\Prism[1/\mathcal{I}_\Prism]^{\wedge}_p) \simeq \mathrm{Vect}^{\varphi}(X_{qrsp},\Prism_\bullet[1/I]^{\wedge}_p).\]
In particular, we have
\[ \mathrm{Vect}^{\varphi}(X_{qrsp},\Prism_\bullet[1/I]^{\wedge}_p) \simeq \mathrm{Rep}_{\mathbf{Z}_p}(G_K)\]
by Corollary~\ref{LocSysLaurentFCrysIsog}. Theorem~\ref{MainThm} can then be formulated as the assertion that the base change functor
\[ \mathrm{Vect}^{\varphi}(X_{qrsp},\Prism_\bullet) \to \mathrm{Vect}^{\varphi}(X_{qrsp},\Prism_\bullet[1/I]^{\wedge}_p) \simeq \mathrm{Rep}_{\mathbf{Z}_p}(G_K)\]
is fully faithful with image given by $\mathbf{Z}_p$-lattices in crystalline $\mathbf{Q}_p$-representations of $G_K$. In fact, we have already proven full faithfulness in \S \ref{sec:FCrysOK}; in the rest of the section, we shall check the assertion regarding the essential image.
\end{remark}

\begin{construction}[de Rham period sheaves on the prismatic site]
We define de Rham period sheaves on $X_{qrsp}$ as follows:
\[ \mathbb{B}_{\dR}^+ = \left(\Prism_\bullet[1/p]\right)^\wedge_I \quad \text{and}  \quad \mathbb{B}_{\dR} = \mathbb{B}_{\dR}^+[1/I].\]
Note that $\mathbb{B}_{\dR}^+({\mathcal{O}_C}) = \mathbb{B}_{\dR}^+(C)$ is the usual de Rham period ring attached to $C$ (and thus a discrete valuation ring). More generally,  a similar assertion holds true for any $R \in X_{qrsp}$ which is perfectoid.
\end{construction}

\subsection{From filtered $\varphi$-modules to crystals}
\label{ss:filtcrys}

In this subsection, we explain how to convert a filtered $\varphi$-module over $K$ into a crystal over $\Prism_\bullet \langle I/p \rangle[1/p]$. The basic idea is to take the constant crystal over $\Prism\{I/p\}[1/p]$, and then modify its $\varphi$-pullback to $\Prism\langle I/p \rangle[1/p]$ at the Hodge-Tate divisor using the filtration. 

\begin{construction}[From filtered $\varphi$-modules over $K$ to crystals over $\Prism_\bullet \langle I/p \rangle{[1/p]}$]
\label{FiltPhiCrys}
Let $(D,\varphi_D,\mathrm{Fil}^*)$ be a filtered $\varphi$-module over $K$, i.e., $D$ is a finite dimensional $K_0$-vector space equipped with an isomorphism $\varphi_D:\varphi^* D \simeq D$, and $\mathrm{Fil}^*$ is a (finite, exhaustive, separated) filtration on the $K$-vector space $D_K = D \otimes_{W(k),\varphi} K$; here the extra Frobenius twist is present for compatibility with the prismatic (as opposed to crystalline) picture. We shall attach an $F$-crystal $\mathcal{M}(D)_{\langle I/p \rangle} \in \mathrm{Vect}^{\varphi}(X_{qrsp},\Prism_\bullet \langle I/p \rangle[1/p])$ to the filtered $\varphi$-module  $(D,\varphi_D,\mathrm{Fil}^*)$.

First, we attach an $F$-crystal $\mathcal{M}_1 := \mathcal{M}(D,\varphi)_{\{I/p\}}$ over $\Prism_\bullet\{I/p\}[1/p]$ to the $\varphi$-module $(D,\varphi_D)$. Regard $(D,\varphi_D)$ as an $F$-crystal of $\mathcal{O}_\Prism[1/p]$-modules on $\mathrm{Spec}(k)$.  Using the structure map $W(k) \to \Prism_\bullet\{I/p\}$, we can regard $\Prism_\bullet\{I/p\}$ as a diagram of prisms over $k$. Consequently, we obtain an $F$-crystal $\mathcal{M}_1 := \mathcal{M}(D,\varphi)_{\{I/p\}}$ over $\Prism_\bullet\{I/p\}[1/p]$ from $(D,\varphi_D)$ via base change, i.e.,
\[ \mathcal{M}_1 = D \otimes_{W(k)} \Prism_\bullet\{I/p\},\] 
with the (unit) $F$-crystal structure $\varphi_{\mathcal{M}_1}: \varphi^* \mathcal{M}_1 \simeq \mathcal{M}_1$ induced from $\varphi_D$.

Next, consider the finite projective $\Prism_\bullet \langle I/p \rangle[1/p]$-module $\mathcal{M}_2 := \mathcal{M}_1 \otimes_{\Prism_\bullet \{I/p\}, \tilde{\varphi}} \Prism_\bullet \langle I/p \rangle$. The map $\varphi_{\mathcal{M}_1}$ induces an isomorphism  $\varphi_{\mathcal{M}_2}: \varphi^* \mathcal{M}_2 \simeq \mathcal{M}_2$, so the pair $(\mathcal{M}_2, \varphi_{\mathcal{M}_2})$ is a unit $F$-crystal over $\Prism_\bullet \langle I/p \rangle[1/p]$. To obtain the desired $F$-crystal, we modify $\mathcal{M}_2$ along the locus $I=0$ using the filtration on $D_K$ to obtain the desired $F$-crystal $\mathcal{M}_3 = \mathcal{M}(D)_{\langle I/p \rangle}$ over $\Prism_\bullet \langle I/p \rangle[1/p]$. More precisely, we apply Beauville-Laszlo glueing along the Cartier divisor defined by $I$ to the vector bundles
\begin{itemize}
\item $\mathcal{M}_2[1/I] \in \mathrm{Vect}(X_{qrsp}, \Prism_\bullet \langle I/p \rangle[1/p,1/I])$.
\item $\mathrm{Fil}^0(D_K \otimes_K \mathbb{B}_{\dR}) \in \mathrm{Vect}(X_{qrsp},\mathbb{B}_{\dR}^+)$. 
\end{itemize}
along the isomorphism
\begin{equation}
\label{BLBdR}
 \mathcal{M}_2[1/I] \otimes_{\Prism_\bullet \langle I/p \rangle} \mathbb{B}_{\dR} \simeq D \otimes_{W(k),\varphi} \mathbb{B}_{\dR}   \simeq \mathrm{Fil}^0(D_K \otimes_K \mathbb{B}_{\dR})[1/I]
 \end{equation}
in $\mathrm{Vect}(X_{qrsp},\mathbb{B}_{\dR})$; here we implicitly use Lemma~\ref{RationalPresent} to identify $\mathbb{B}_{\dR}^+$ with $\Prism_\bullet \langle I/p \rangle[1/p]^{\wedge}_I$. Thus, we obtain a vector bundle $\mathcal{M}_3 \in \mathrm{Vect}(X_{qrsp},\Prism_\bullet \langle I/p \rangle[1/p])$ equipped with isomorphisms
\[ \mathcal{M}_3[1/I] \simeq \mathcal{M}_2[1/I] \in \mathrm{Vect}(X_{qrsp},\Prism_\bullet \langle I/p \rangle[1/p,1/I]) \quad \text{and} \quad (\mathcal{M}_3)^\wedge_I \simeq \mathrm{Fil}^0(D_K \otimes_K \mathbb{B}_{\dR}) \in \mathrm{Vect}(X_{qrsp},\mathbb{B}_{\dR}^+) \]
compatible with \eqref{BLBdR} over $\mathbb{B}_{\dR}$. The map $\varphi_{\mathcal{M}_2}[1/I]$ then yields an isomorhism $\varphi_{\mathcal{M}_3}:\varphi^* \mathcal{M}_3 [1/I] \simeq \mathcal{M}_3[1/I]$, yielding the desired $F$-crystal.
\end{construction}

For completeness, we also remark that a partial inverse to the previous construction is described in the proof of Theorem~\ref{KisinFullyFaithfulBK}.  Next, we observe that the Frobenius pullback trick allows us to extend the preceding construction to larger loci. 

\begin{remark}[Analytic continuation of $F$-crystals to the open unit disc]
\label{ACFCrys}
Fix an $F$-crystal $(\mathcal{M},\varphi_{\mathcal{M}})$ over $\Prism_\bullet \langle I/p \rangle[1/p]$ on $X_{qrsp}$; the main example for us is the output of Construction~\ref{FiltPhiCrys}. One can use the contracting property of $\varphi$ to extend $(\mathcal{M},\varphi_{\mathcal{M}})$ uniquely to an $F$-crystal $(\mathcal{N},\varphi_{\mathcal{N}})$ over $\Prism_\bullet \langle \varphi^n(I)/p \rangle[1/p]$ on $X_{qrsp}$ for any $n \geq 1$. Let us explain how to construct this extension for $n=1$; the extension for larger $n$ is constructed entirely analogously. The underlying vector bundle $\mathcal{N} \in \mathrm{Vect}(X_{qrsp},\Prism_\bullet \langle \varphi(I)/p \rangle[1/p])$ is defined glueing the vector bundles
\begin{enumerate}
\item $\tilde{\varphi}^* \mathcal{M}|_{\Prism_\bullet \langle \varphi(I)/p \rangle[1/p, 1/I]} \in \mathrm{Vect}(X_{qrsp}, {\Prism_\bullet \langle \varphi(I)/p \rangle[1/p, 1/I]})$; note that this bundle extends $\varphi^* \mathcal{M}$ (and thus also $\mathcal{M}$, via $\varphi_{\mathcal{M}}$) over $\Prism_\bullet \langle I/p \rangle[1/p,1/I]$.
\item $\mathcal{M}^{\wedge}_I \in \mathrm{Vect}(X_{qrsp},\Prism_\bullet \langle \varphi(I)/p \rangle[1/p]^{\wedge}_I)$; here we implicitly use that the $I$-adic completion of $\Prism_\bullet[1/p]$ coincides with that of $\Prism_\bullet \langle \varphi^n(I)/p \rangle[1/p]$ for all $n \geq 0$ by Lemma~\ref{RationalPresent} below.
\end{enumerate}
along the evident natural isomorphism of their base changes to $\Prism_\bullet \langle \varphi(I)/p \rangle[1/p]^{\wedge}_I[1/I]$. It is clear from the above description that $\mathcal{N}$ extends $\mathcal{M}$ over $\Prism_R \langle I/p \rangle[1/p]$, whence $\varphi^* \mathcal{N} = \tilde{\varphi}^* \mathcal{M}$. The desired $F$-crystal structure $\varphi_{\mathcal{N}}:(\varphi^* \mathcal{N})[1/I] \simeq \mathcal{N}[1/I]$ arises by observing that both sides identify with $(\tilde{\varphi}^* \mathcal{M})[1/I]$: this follows for the target by the definition in (1) and for the source by the preceding sentence.
\end{remark}

The following lemma was used above. 

\begin{lemma}
\label{RationalPresent}
The natural map gives an isomorphism
\[ \mathbb{B}_{\dR}^+ := \Prism_\bullet[1/p]^{\wedge}_I \simeq \Prism_\bullet \langle \varphi^n(I)/p \rangle[1/p]^{\wedge}_I\]
for all $n \geq 0$.
\end{lemma}
\begin{proof}
As $\Prism_R$ is a transversal prism for $R \in X_{qrps}$, it is enough to show the following: if $(A,I)$ is a transversal prism, then the natural map
\[ A[1/p]^{\wedge}_I \to A \langle \varphi^n(I)/p \rangle[1/p]^{\wedge}_I \]
is an isomorphism for all $n \geq 0$. Fix some such $n \geq 0$. We may assume $I=(d)$ is principal.  Using the presentation
\[ A \langle \varphi^n(d)/p \rangle = A[x]/(px-\varphi^n(d))^{\wedge}_p,\]
it is sufficient to show that the natural map gives an isomorphism
\[ A[1/p]/d \to \mathrm{Kos}(A[x]; px-\varphi^n(d), d)^{\wedge}_p[1/p],\]
where $\mathrm{Kos}$ denotes a Koszul complex, i.e.~in this case the derived reduction of $A[x]$ modulo $px-\varphi^n(d)$ and $d$. Now $\varphi^n(d) = d^{p^n} + ph$ for some $h \in A$, so the right side above simplifies to 
\[ \mathrm{Kos}(A[x]; px-\varphi^n(d), d)^{\wedge}_p[1/p] \simeq \mathrm{Kos}(A[x]; px-ph, d)^{\wedge}_p[1/p].\]
Relabelling $y=x-h$ this further simplifies to
\[  \mathrm{Kos}(A[y]; py, d)^{\wedge}_p[1/p] \simeq \mathrm{Kos}(A[y]; y,d)^{\wedge}_p[1/p] \simeq A/d[1/p],\]
as wanted; here the first isomorphism arises by noting that $py$ is a nonzerodivisor in $A/d[y]$ by transversality of $(A,I)$ and that the kernel $yA[y]/pyA[y]$ of $A[y]/py \to A[y]/y = A$ is killed by $p$ (and thus disappears on applying $(-)^{\wedge}_p[1/p]$). 
\end{proof}

\subsection{Boundedness of descent data at the boundary}
\label{ss:BoundDescent}

Given an $F$-crystal $\mathcal{M}'$ over $\Prism_{\mathcal{O}_C}$ with descent data on $\mathcal{M'}\langle I/p \rangle[1/p]$, we shall explain how the descent data automatically extends to $\mathcal{M}'[1/p]$ using crucially results from \cite{AMMNBFS} to establish the following extension result for $p$-adic Tate twists. 

In the statement and proof of the following proposition, we are taking fixed points $-^{\varphi=1}$ in the usual (nonderived) sense.

\begin{proposition}[Analytic continuation for $p$-adic Tate twists]
\label{TateTwistRational}
Let $R$ be a qrsp $p$-torsionfree $\mathcal{O}_C$-algebra. The natural map
\[ \Prism_R\{n\}^{\varphi=1}[1/p] \to \Prism_R \langle I/p \rangle\{n\}^{\varphi=1}[1/p]\]
is an isomorphism.
\end{proposition}

We note that $\varphi$ is not quite a selfmap of $\Prism_R\{n\}$ -- it only takes values in $\Prism_R\{n\}[1/I]$ -- but being $\varphi$-fixed still makes sense.

\begin{proof}
Recall that the $p$-adic Tate twist attached to $R$ is defined as follows:
\[ \mathbb{Q}_p(n)(R) := \mathrm{fib}\left(\mathrm{Fil}^n_N \Prism_R\{n\} \xrightarrow{\varphi_n-1} \Prism_R\{n\}\right)[1/p],\]
using the Nygaard filtration $\mathrm{Fil}^\bullet_N \Prism_R\subset \Prism_R$ from \cite[Definition 12.1]{BhattScholzePrisms}. Thus, there is a natural injective map
\[ H^0(\mathbb{Q}_p(n)(R)) \to \Prism_R\{n\}^{\varphi=1}[1/p].\]
In fact, we observe that this map is an isomorphism: containment in the correct level of the Nygaard filtration is automatic for $\varphi$-fixed elements of $\Prism_R\{n\}$ as the Nygaard filtration on $\Prism_R$ may be defined (using that $R$ is qrsp) as the $\varphi$-preimage of the $I$-adic filtration on $\Prism_R$ (see \cite[\S 12]{BhattScholzePrisms}).

In particular, it is enough to show that the composition
\[ a: H^0(\mathbb{Q}_p(n)(R)) \to  \Prism_R \langle I/p \rangle\{n\}^{\varphi=1}[1/p]\]
is an isomorphism. For this, recall that since $\Prism_{R/p} = \Prism_R\{I/p\}$, the $p$-adic Tate twist attached to $R/p$ may be written as 
\[ H^0(\mathbb{Q}_p(n)(R/p)) = \left(\Prism_R\{I/p\}\{n\}\right)^{\varphi=1}[1/p],\]
where we implicitly use that inverting $p$ kills any Nygaard graded quotient of $\Prism_{R/p}$. (In fact, as $R/p$ is qrsp $\mathbf{F}_p$-algebra, the complex $\mathbb Q_p(n)(R/p)$ is in fact concentrated in degree $0$ by \cite[Proposition 8.20]{BMS2}; but we do not need this.) Moreover, the natural map 
\[ H^0(\mathbb{Q}_p(n)(R)) \xrightarrow{b} H^0(\mathbb{Q}_p(n)(R/p))\]
 factors as
\[ b: H^0(\mathbb{Q}_p(n)(R)) \xrightarrow{a} \Prism_R \langle I/p \rangle\{n\}^{\varphi=1}[1/p] \xrightarrow{c} H^0(\mathbb{Q}_p(n)(R/p)).\]
 As the maps $a$, $b$ and $c$ are all injective, it suffices to show that $\mathrm{im}(c)$ maps to $0$ in $\mathrm{coker}(b)$. To understand this cokernel, we use the Beilinson fibre sequence \cite[Theorem 6.17]{AMMNBFS}
 \[ \mathbb{Q}_p(n)(R)\rightarrow \mathbb{Q}_p(n)(R/p) \rightarrow \left(L\Omega_R / \mathrm{Fil}^n_H L\Omega_R\right)\{n\} [1/p]\xrightarrow{[1]}\]
which in particular induces an exact sequence
\[ H^0(\mathbb{Q}_p(n)(R))\xrightarrow{b} H^0(\mathbb{Q}_p(n)(R/p)) \xrightarrow{d} H^0(\left(L\Omega_R / \mathrm{Fil}^n_H L\Omega_R\right)\{n\}[1/p]).\]
Our task is then to show that $d \circ c = 0$. Using our chosen $\mathcal{O}_C$-algebra structure to trivialize the Breuil-Kisin twists $\Prism_{\mathcal O_C}\{1\}\cong \frac{1}{q-1}\Prism_{\mathcal O_C}$, we can make the map $c$ and $d$ explicit. First, $c$ is the natural map
\[ \left(\frac{1}{(q-1)^n} \Prism_R \langle I/p \rangle\right)^{\varphi=1}[1/p] \to  \left(\frac{1}{(q-1)^n}\Prism_R\{I/p\}\right)^{\varphi=1}[1/p]. \]
Moreover, by the arguments of \cite[6.18--6.21]{AMMNBFS}, the map $d$ is up to nonzero scalar compatible with the isomorphism $L\Omega_R\cong A_{\crys}(R/p) = \Prism_{R/p}$ and the forgetful map $\mathbb{Q}_p(R/p)\to \Prism_{R/p}[1/p]$. Thus, it remains to show that the image of the natural map
\[
\left(\frac{1}{(q-1)^n} \Prism_R \langle I/p \rangle\right)^{\varphi=1}[1/p]\to \frac{1}{(q-1)^n} \Prism_R\{I/p\}[1/p]\cong \frac{1}{(q-1)^n} \Prism_{R/p}[1/p]\cong \frac{1}{(q-1)^n} L\Omega_R[1/p]
\]
is contained in $\frac{1}{(q-1)^n} \mathrm{Fil}^n_H L\Omega_R [1/p]$. But the Frobenius on $\Prism_R$ induces an injective map
\[
(L\Omega_R/\mathrm{Fil}^n_H L\Omega_R)[1/p]\hookrightarrow \mathbb B_{\dR}^+(R)/I^n,
\]
as follows from the discussion of the Nygaard filtration\footnote{Indeed, as $R$ is qrsp, the Nygaard filtration $\mathrm{Fil}^\bullet_N \Prism_R$ on $\Prism_R$ can be defined as the preimage of the $I$-adic filtration $I^\bullet \Prism_R$ under the Frobenius. Consequently, the Frobenius on $\Prism_R$ gives an injective map $\Prism_R/\mathrm{Fil}^n \Prism_R \to \Prism_R / I^n \Prism_R$. Inverting $p$ then gives an injective map $(\Prism_R/\mathrm{Fil}^n_N \Prism_R)[1/p] \to (\Prism_R/I^n \Prism_R)[1/p] \simeq \mathbb{B}_{\dR}^+(R)/I^n \mathbb{B}_{\dR}^+(R)$. It now remains to observe that $(\Prism_R/\mathrm{Fil}^n_N \Prism_R)[1/p] \simeq L\Omega_R/\mathrm{Fil}^n_H L\Omega_R[1/p]$ by comparing graded pieces in the absolute de Rham comparison.}. Dividing by $(q-1)^n$ and noting that $\varphi(q-1)$ is a generator of $I \mathbb{B}_{dR}^+(R)$, we get an injection
\[
 \frac{1}{(q-1)^n} (L\Omega_R/\mathrm{Fil}^n_H L\Omega_R)[1/p]\hookrightarrow I^{-n} \mathbb B_{\dR}^+(R)/\mathbb B_{\dR}^+(R)
\]
induced by the Frobenius. The composite map
\[
\left(\frac{1}{(q-1)^n} \Prism_R \langle I/p \rangle\right)^{\varphi=1}[1/p]\to I^{-n} \mathbb B_{\dR}^+(R)/\mathbb B_{\dR}^+(R)
\]
is, by the $\varphi$-invariants on the left, simply induced by the natural map $\frac{1}{(q-1)^n} \Prism_R\langle I/p\rangle[1/p] \to I^{-n} \mathbb B_{\dR}^+(R)$, which actually takes values in $\mathbb B_{\dR}^+(R)$ as $q-1$ is invertible in $\mathbb B_{\dR}^+(R)$, as desired; here we used implicitly Lemma~\ref{RationalPresent}. But this implies that the displayed map above is $0$, which then implies by the preceding reasoning that the map
\[ \left(\frac{1}{(q-1)^n} \Prism_R \langle I/p \rangle\right)^{\varphi=1}[1/p] \to  \frac{1}{(q-1)^n} (L\Omega_R/\mathrm{Fil}^n_H L\Omega_R)[1/p]\]
is $0$, as wanted.
\end{proof}

For future use, we collect some properties of the prism of $\mathcal O_C\widehat{\otimes}_{\mathcal O_K} \mathcal O_C$, using crucially Proposition~\ref{TateTwistRational} for the first one. 

\begin{lemma}[Properties of $\Prism_{Y \times_X Y}$]
\label{FixedRational}
Let $R=\mathcal O_C\widehat{\otimes}_{\mathcal O_K} \mathcal O_C$. Write $q_1,q_2 \in \Prism_R$ for the image of $q \in \Prism_{\mathcal{O}_C}$ along the maps $\Prism_{\mathcal{O}_C} \to \Prism_R$ induced by the first and second structure maps $\mathcal{O}_C \to R$.
\begin{enumerate}
\item The natural map
\[ \Prism_R[1/(p(q_1-1)(q_2-1))]^{\varphi=1} \to \Prism_R \langle I/p \rangle[1/(p(q_1-1)(q_2-1))]^{\varphi=1}\]
is an isomorphism on $H^0$.

\item The element $(q_1-1)(q_2-1)$ is invertible in $\Prism_R \langle p/(q_1-1)^p \rangle[1/p]$.

\item The natural maps give a short exact sequence 
\[ 0 \to \Prism_R \to \Prism_R \langle (q_1-1)^p/p \rangle \oplus \Prism_R \langle p/(q_1-1)^p \rangle \to \Prism_R \langle p/(q_1-1)^p, (q_1-1)^p/p \rangle \to 0.\]
\end{enumerate}
\end{lemma}

\begin{proof} 
\begin{enumerate}
\item Note that we have a $\varphi$-equivariant isomorphisms $\Prism_{\mathcal{O}_C}\{n\} \simeq \frac{1}{(q-1)^n} \Prism_{\mathcal{O}_C}$ for all $n \geq 0$. By base change, we obtain natural isomorphisms
\[ \Prism_R\{2n\} \simeq \frac{1}{(q_1-1)^n} \Prism_R\{n\} \simeq \frac{1}{ (q_1-1)^n (q_2-1)^n } \Prism_R.\]
As the direct limit over $n$ of the terms on the right gives $\Prism_R[1/(q_1-1)(q_2-1)]$, it suffices to see that for all $n\geq 0$, the natural map
\[\Prism_R\{n\}^{\varphi=1}[1/p]\to \Prism_R\langle I/p\rangle\{n\}^{\varphi=1}[1/p] \]
is an isomorphism on $H^0$, which follows from Proposition~\ref{TateTwistRational}.

\item Since $(q_1-1)^p \mid p$ in $\Prism_R \langle p/(q_1-1)^p \rangle$, it is clear that $(q_1-1)$ is invertible in $\Prism_R \langle p/(q_1-1)^p \rangle[1/p]$. To show $(q_2-1)$ is invertible, we shall check the stronger statement that 
\[ A := \Prism_R \langle p/(q_1-1)^p \rangle/(q_2-1)\] 
is an $\mathbf{F}_p$-algebra.  We first observe that $[p]_{q_1} = [p]_{q_2} \cdot u$ for a unit $u \in \Prism_R$: this follows from the realization of $\Prism_R$ as a suitable prismatic envelope of $\Prism_{\mathcal{O}_C} \otimes \Prism_{\mathcal{O}_C}$. As $[p]_{q_1} = (q_1-1)^{p-1} \mod p\Prism_R$ and similarly for $[p]_{q_2}$, we obtain an equation of the form
\[ (q_1-1)^{p-1} = (q_2-1)^{p-1} \cdot u  + pv \in \Prism_R\]
for $u,v \in \Prism_R$. This allows us to write
\begin{equation}
\label{prismaticq}
 (q_1-1)^p = p(q_1-1)x \mod (q_2-1)\Prism_R.
\end{equation}
Our task is to show that
\[ A = \Prism_R \langle p/(q_1-1)^p \rangle/(q_2-1) = \left(\Prism_R[z]/(z (q_1-1)^p - p)\right)^{\wedge}_{(q_1-1)}/ (q_2-1) \]
is an $\mathbf{F}_p$-algebra. Thanks to \eqref{prismaticq}, we have
\[ A = \left(\left(\Prism_R/(q_2-1)\right)[z]/(zp(q_1-1)x - p)\right)^{\wedge}_{(q_1-1)} = \left(\left(\Prism_R/(q_2-1)\right)[z]/(p\epsilon - p)\right)^{\wedge}_{(q_1-1)},\]
where $\epsilon = z(q_1-1)x$. But $q_1-1$ is topologically nilpotent (due to the completion operation), so the same holds for $\epsilon$, whence $\epsilon-1$ is a unit. The above presentation simplifies to
\[ A = \left(\left(\Prism_R/(q_2-1)\right)[z]/p\right)^{\wedge}_{(q_1-1)},\]
which is clearly an $\mathbf{F}_p$-algebra. 

\item First, consider the regular ring $A := \mathbf{Z}_p[u]$. Since blowing up at a regular ideal does not change $\mathcal{O}$-cohomology,  the Mayer-Vietoris sequence for the standard charts of the blowup $X$ of $\mathrm{Spec}(A)$ at the ideal $(p,u)$ gives an exact sequence
\[ 0 \to A \simeq R\Gamma(X,\mathcal{O}_X) \to A[u/p] \oplus A[p/u] \to A[p/u,u/p] \to 0\]
of $A$-modules. Moreover, each of the $4$ terms that appears is a noetherian $A$-algebra, and the maps in the sequence arises as linear combinations of $A$-algebra maps. Now consider the composition $A \to \Prism_{\mathcal{O}_C} \xrightarrow{i_1} \Prism_R$ given by sending $u$ to $(q_1-1)^p$. Each map in this composition is $(p,u)$-completely flat, and hence so is the composition. The $(p,u)$-completed base change of the above sequence along $A \to \Prism_R$ then gives an exact triangle
\[ \Prism_R \to \Prism_R \langle (q_1-1)^p/p \rangle \oplus \Prism_R \langle p/(q_1-1)^p \rangle \to \Prism_R \langle p/(q_1-1)^p, (q_1-1)^p/p \rangle \]
in the derived category. Since $A \to \Prism_R$ is $(p,u)$-completely flat, each term of the above triangle is $(p,u)$-completely flat over a noetherian ring, and is thus concentrated in degree $0$. Thus, the above triangle is in fact an exact sequence, as wanted.
\qedhere
\end{enumerate}
\end{proof}

The main result of this subsection is the following:

\begin{proposition}[Boundedness of descent data at the boundary of the open unit disc]
\label{ACDescent}
Let $\mathcal{M} \in \mathrm{Vect}^{\varphi}(\Prism_{\mathcal{O}_C})$ be an $F$-crystal over $\Prism_{\mathcal{O}_C}$. Assume that the base change $\mathcal{M}\langle I/p \rangle[1/p] \in \mathrm{Vect}^{\varphi}(\Prism_{\mathcal{O}_C}\langle I/p \rangle[1/p])$ is provided with descent data with respect to $Y \to X$ (i.e., we are given a lift to $\mathrm{Vect}^{\varphi}(X_{qrsp},\Prism_\bullet \langle I/p \rangle[1/p])$. Then this descent data extends uniquely to $\mathcal{M}[1/p] \in  \mathrm{Vect}^{\varphi}(\Prism_{\mathcal{O}_C}[1/p])$.
\end{proposition}
\begin{proof}
Write $R = \mathcal{O}_C \widehat{\otimes}_{\mathcal{O}_K} \mathcal{O}_C$, so $p_1^* \mathcal{M}[1/p]$ and $p_2^* \mathcal{M}[1/p]$ are vector bundles over $\Prism_R[1/p]$. By the assumption on $\mathcal{M}$, we have an isomorphism 
\[ \alpha_{\langle I/p \rangle}:p_1^* \mathcal{M} \langle I/p \rangle [1/p] \simeq p_2^* \mathcal{M} \langle I/p \rangle[1/p]\]
in $\mathrm{Vect}^{\varphi}( \Prism_R \langle I/p \rangle[1/p])$ satisfying the cocycle condition. Our task is to extend this to an isomorphism
\[ \alpha:p_1^* \mathcal{M}[1/p] \simeq p_2^* \mathcal{M}[1/p]\] 
in $\mathrm{Vect}(\Prism_R[1/p])$; the cocycle condition as well as the $\varphi$-equivariance will then be automatic from that for $\alpha_{\langle I/p \rangle}$ as $\Prism_\bullet[1/p] \to \Prism_\bullet \langle I/p \rangle[1/p]$ is injective on $X_{qrsp}$. In fact, it suffices to check that $\alpha_{\langle I/p \rangle}$ carries $p_1^* \mathcal{M}[1/p]$ into $p_2^* \mathcal{M}[1/p]$; applying the same reasoning to the inverse will prove the claim. 

First, we observe that Remark~\ref{ACFCrys} gives descent data over $\Prism_R \langle \varphi(I)/p \rangle[1/p]$. In particular, since $\varphi(I) \subset p \Prism_R \langle (q_1-1)^p/p \rangle$, there is a unique $\varphi$-equivariant map 
\[ \alpha_{\langle (q_1-1)^p/p \rangle}:p_1^* \mathcal{M} \langle (q_1-1)^p/p \rangle [1/p] \to p_2^* \mathcal{M} \langle (q_1-1)^p/p \rangle\] 
of finite projective modules over $\Prism_R \langle (q_1-1)^p/p \rangle[1/p]$ extending $\alpha_{\langle I/p \rangle}$ over $\Prism_R \langle I/p \rangle[1/p]$. 

Next, we extend to $\Prism_R \langle p/(q_1-1)^p \rangle[1/p]$. For this, let $T = T(\mathcal{M})$ be the finite free $\mathbf{Z}_p$-module obtained from the \'etale realization of $\mathcal{M}$. By \cite[Lemma 4.26]{BMS1}, there is a $\varphi$-equivariant identification 
\[ \mathcal{M'}[\frac{1}{q-1}] \simeq T \otimes_{\mathbf{Z}_p} \Prism_{\mathcal{O}_C} [\frac{1}{q-1}]\]
of vector bundles over $\Prism_{\mathcal{O}_C}[\frac{1}{q-1}]$ (where $T$ has the trivial $\varphi$-action).  Choosing a basis for $T$, the isomorphism $\alpha_{\langle I/p \rangle}$ gives  a matrix with coefficients in $H^0$ of
 \[ \Prism_R \langle I/p \rangle[1/(p(q_1-1)(q_2-1))]^{\varphi=1}.\] 
 Lemma~\ref{FixedRational} (1) shows that the coefficients in fact lie in $H^0$ of
 \[ \Prism_{R}[1/(p(q_1-1)(q_2-1))]^{\varphi=1}.\] 
By Lemma~\ref{FixedRational} (2), the coefficients then also lie in 
\[ \Prism_{R} \langle p/(q_1-1)^p \rangle[1/p].\] 
These coefficients define a map 
\[ \alpha_{\langle p/(q_1-1)^p \rangle}:p_1^* \mathcal{M}\langle p/(q_1-1)^p \rangle[1/p] \to p_2^* \mathcal{M}\langle p/(q_1-1)^p \rangle[1/p] \]
of finite projective modules over $\Prism_R \langle p/(q_1-1)^p \rangle[1/p]$ that is compatible with the map $\alpha_{\langle (q_1-1)^p/p \rangle}$ after base change to $\Prism_R \langle p/(q_1-1)^p, (q_1-1)^p/p \rangle [1/p]$. 

Inverting $p$ in the exact sequence appearing in Lemma~\ref{FixedRational} (3) and tensoring with the finite projective $\Prism_R[1/p]$-module $\mathrm{Hom}_{\Prism_R[1/p]}(p_1^* \mathcal{M}[1/p],p_2^* \mathcal{M}[1/p])$, we can patch together the maps $\alpha_{\langle (q_1-1)^p/p \rangle}$ and $\alpha_{\langle p/(q_1-1)^p \rangle}$ constructed above to obtain a map $\alpha:p_1^* \mathcal{M}[1/p] \to p_2^*\mathcal{M}[1/p]$ of finite projective $\Prism_R[1/p]$-modules extending $\alpha_{\langle I/p \rangle}$, as wanted.
\end{proof}

\subsection{Proof of essential surjectivity}
\label{ss:PfEssSur}

We can now prove the promised theorem:

\begin{proof}[Proof of essential surjectivity in Theorem~\ref{MainThm}]
Fix $L \in \mathrm{Rep}^{\crys}_{\mathbf{Z}_p}(G_K)$. Consider the weakly admissible filtered $\varphi$-module $(D,\varphi_D,\mathrm{Fil}^*)$ attached to the $G_K$-representation $L[1/p]$. We shall attach a prismatic $F$-crystal $\mathcal{M}(D)_{int}$ over $\Prism_\bullet$ to $(D,\varphi_D,\mathrm{Fil}^*)$ on $X_{qrsp}$ with \'etale realization given by $L[1/p]$ after inverting $p$; we then check that $\mathcal{M}(D)_{int}$ can be constructed to have \'etale realization $L$ on the nose.

First, we attach an $F$-crystal $\mathcal{M}(D)_{bdd}$ over $\Prism_\bullet[1/p]$ on $X_{qrsp}$ to $(D,\varphi_D,\mathrm{Fil}^*)$, and also construct an extension of the value $\mathcal{M}(D)_{bdd}(Y)$ to an $F$-crystal $\mathcal{M}'$ over $\Prism_{\mathcal{O}_C}$.  Consider the compatible system of $F$-crystals  
\[ \mathcal{M}(D) := \{ \mathcal{M}(D)_{\langle \varphi^n(I)/p \rangle} \in \mathrm{Vect}(X_{qrsp}, \Prism_\bullet \langle \varphi^n(I)/p \rangle[1/p])\}_{n \geq 1} \]
coming from Remark~\ref{ACFCrys}. The value $\mathcal{M}(D)(Y)$ can be regarded as an $F$-crystal on $\mathrm{Spa}(\Prism_{\mathcal{O}_C}) - \{p=0\}$. To construct the objects mentioned in the first sentence of this paragraph, by Proposition~\ref{ACDescent}, it suffices to explain why the $F$-crystal $\mathcal{M}(D)(Y)$ over $\mathrm{Spa}(\Prism_{\mathcal{O}_C}) - \{p=0\}$ extends to an $F$-crystal $\mathcal{M}'$ over $\Prism_{\mathcal{O}_C}$. For this\footnote{The argument we give is an $A_{\inf}$-variant of the analogous result in Kisin's \cite{Kisin}. It relies ultimately on Berger's observation (see \cite[\S IV.2]{BergerPDE} and \cite[Theorem 1.3.8]{Kisin}) translating weak admissibility of $D$ into a property of $\mathcal{M}(D)(Y)$ via Kedlaya's slope filtration results \cite{KedlayaMonodromy}.}, by gluing on the adic space and using \cite[Theorem 14.2.1]{Berkeley} to handle the difference between $F$-crystals on $\mathrm{Spec}(\Prism_{\mathcal{O}_C})$ and $\mathrm{Spa}(\Prism_{\mathcal{O}_C})-\{x_k\}$, it suffices to show that the base changed $\varphi$-module $\mathcal{M}(D)(Y)_{\mathscr{R}}$ over the Robba ring $\mathscr{R}$ for $\Prism_{\mathcal{O}_C}$\footnote{We refer to \cite[Definition 1.8.1]{FFCurve} for the definition. This ring is sometimes also called the extended Robba ring $\widetilde{\mathcal{R}}$, e.g., as in \cite[Definition 13.4.3]{Berkeley}.} extends (as a $\varphi$-module) to the integral Robba ring $\mathscr{R}^{\text{int}}$ (called $\widetilde{\mathcal{R}}^{\text{int}}$ in \cite[Definition 12.3.1]{Berkeley}). It suffices to show that $\mathcal{M}(D)(Y)_{\mathscr{R}}$ is trivial as a $\varphi$-module over $\mathscr{R}$. Now $\varphi$-modules over $\mathscr{R}$ can be identified with vector bundles on the Fargues--Fontaine curve $X_{FF}$ by \cite[Corollary 11.2.22]{FFCurve}. By the weak admissibility of $(D,\varphi_D,\mathrm{Fil}^*)$ and  \cite[Proposition 10.5.6]{FFCurve} and matching constructions, the vector bundle on $X_{FF}$ attached to $\mathcal{M}(D)(Y)_{\mathscr{R}}$ is semistable of slope $0$. As all such vector bundles are trivial by \cite[Theorem 8.2.10 (1)]{FFCurve}, the desired claim follows.

Next, we explain why the descent data on $\mathcal{M}'[1/p]$ can be used to construct a $G_K$-equivariant structure on $\mathcal{M}'$. Note that $\mathcal{M}'[1/I]^{\wedge}_p \in \mathrm{Vect}^{\varphi}(\Prism_{\mathcal{O}_C}[1/I]^{\wedge}_p)$ has the form $T \otimes_{\mathbf{Z}_p} \Prism_{\mathcal{O}_C}[1/I]^{\wedge}_p$ for a finite free $\mathbf{Z}_p$-module $T$ with trivial $\varphi$-action (Corollary~\ref{LocSysLaurentFCrysIsog}). Inverting $p$ shows that 
\[ \mathcal{M}(D)_{bdd}(Y) \otimes_{\Prism_{\mathcal{O}_C}[1/p]} \Prism_{\mathcal{O}_C}[1/I]^{\wedge}_p[1/p] = \mathcal{M}'[1/I]^{\wedge}_p[1/p] = T[1/p] \otimes_{\mathbf{Q}_p} \Prism_{\mathcal{O}_C}[1/I]^{\wedge}_p[1/p].\]
As the left side is the value of $Y$ of an $F$-crystal on $X$, it carries a natural $G_K$-equivariant structure. Consequently, the right side $T[1/p] \otimes_{\mathbf{Q}_p}  \Prism_{\mathcal{O}_C}[1/I]^{\wedge}_p \in \mathrm{Vect}^{\varphi}( \Prism_{\mathcal{O}_C}[1/I]^{\wedge}_p[1/p])$ is naturally a $G_K$-equivariant object of 
\[ \mathrm{Vect}(\mathbf{Z}_p) \otimes_{\mathbf{Z}_p} \mathbf{Q}_p \stackrel{(-)^{\varphi=1}}{\simeq} \mathrm{Vect}^{\varphi}(\Prism_{\mathcal{O}_C}[1/I]^{\wedge}_p) \otimes_{\mathbf{Z}_p} \mathbf{Q}_p \subset \mathrm{Vect}^{\varphi}( \Prism_{\mathcal{O}_C}[1/I]^{\wedge}_p[1/p]).\]
It follows that $T[1/p]$ is naturally a finite dimensional $\mathbf{Q}_p$-representation of $G_K$. By picking a $G_K$-stable $\mathbf{Z}_p$-lattice $L'$ in this representation and adjusting our choice of $\mathcal{M}'$ along the divisor $\{p=0\} \subset \mathrm{Spec}(\Prism_{\mathcal{O}_C}) - \{x_k\}$ by Beauville-Laszlo glueing (and using Kedlaya's theorem \cite{KedlayaAinf} or \cite[Lemma 4.6]{BMS1} that vector bundles on $\mathrm{Spec}(\Prism_{\mathcal{O}_C}) - \{x_k\}$ extend uniquely to $\mathrm{Spec}(\Prism_{\mathcal{O}_C})$), we can then arrange that $T=L'$ itself is a finite free $\mathbf{Z}_p$-module equipped with a $G_K$-action, or equivalently that $\mathcal{M}' \in \mathrm{Vect}^{\varphi}(\Prism_{\mathcal{O}_C})$ comes equipped with a $G_K$-equivariant structure extending the one on $\mathcal{M}(D)_{bdd}(Y) = \mathcal{M}'[1/p]$ coming from the descent data.

We now claim that the descent data on $\mathcal{M}'[1/p] = \mathcal{M}(D)_{bdd}$ restricts to descent data on $\mathcal{M}'$. In other words, letting $R=\mathcal{O}_C \widehat{\otimes}_{\mathcal{O}_K} \mathcal{O}_C$, the isomorphism 
\[ \alpha:p_1^* \mathcal{M}'[1/p] \simeq p_2^* \mathcal{M}'[1/p]\]
 in $\mathrm{Vect}^{\varphi}(\Prism_R[1/p])$ carries $p_1^* \mathcal{M}'$ isomorphically onto $p_2^* \mathcal{M}'$. In fact, by contemplating the inverse, it suffices to show that $\alpha$ carries $p_1^* \mathcal{M}'$ into $p_2^* \mathcal{M}'$. Now $(p,I) \subset \Prism_R$ is an ideal generated by a regular sequence of length $2$ as $\Prism_R$ is $(p,I)$-completely flat over $\Prism_{\mathcal{O}_C}$ via either structure map. As sections of vector bundles are insensitive to removing closed sets defined by such ideals, it suffices to check our claim on $\mathrm{Spec}(\Prism_R) - V(p,I)$. Since everything is clear after inverting $p$, we may further use Beauville-Laszlo glueing to reduce to checking the statement after $p$-completing, i.e. we must check that the natural isomorphism
\[ \alpha \otimes_{\Prism_R[1/p]} \Prism_R[1/I]^{\wedge}_p[1/p]:p_1^* \mathcal{M}'[1/I]^{\wedge}_p[1/p] \simeq p_2^* \mathcal{M}'[1/I]^{\wedge}_p[1/p]\]
carries $p_1^* \mathcal{M}'[1/I]^{\wedge}_p$ into $p_1^* \mathcal{M}'[1/I]^{\wedge}_p$. Now the base change map
\[ \mathrm{Vect}^{\varphi}(\Prism_R[1/I]^{\wedge}_p) \to \mathrm{Vect}^{\varphi}(\Prism_{R,\perf}[1/I]^{\wedge}_p) \]
is fully faithful (and in fact an equivalence) by the \'etale comparison theorem (see Proposition~\ref{BanachPerfFrobMod}), and therefore also after inverting $p$ on both sides. Thus, it suffices to check our desired containment after pullback to $\Prism_{R,\perf}$. But we know that $\Prism_{R,\perf} \stackrel{a}{\simeq} R\Gamma(\mathrm{Spf}(R)_\eta, A_{\inf}(\mathcal{O}^+)) \stackrel{a}{\simeq} \mathrm{Cont}(G_K, \Prism_{\mathcal{O}_C})$. As inverting $I$ and $p$-completing turns almost isomorphisms to isomorphisms, the desired claim follows from $G_K$-equivariance of the lattice $\mathcal{M}'[1/I]^{\wedge}_p \subset \mathcal{M}'[1/I]^{\wedge}_p[1/p]$.

The previous paragraphs lift $\mathcal{M}' \in \mathrm{Vect}^{\varphi}(\Prism_{\mathcal{O}_C})$ to an object $\mathcal{M}(D)_{int} \in \mathrm{Vect}^{\varphi}(X_{qrsp}, \Prism_\bullet)$; the construction depended on the choice of the $G_K$-stable lattice $T=L'$ in $T[1/p]$ and has the feature that $T(\mathcal{M}(D)_{int}) = L'$. To finish the proof, observe that the argument in Theorem~\ref{FCrystoCrysGal} and the fact that $\mathcal{M}(D)_{int}$ recovers $\mathcal{M}(D)_{\langle I/p \rangle}$ over $\Prism_\bullet \langle I/p \rangle[1/p]$ shows that $T(\mathcal{M}(D)_{int})[1/p] \simeq L[1/p]$. Running the modification argument two paragraphs above with $L'=L$ then shows that we can also arrange that $T(\mathcal{M}(D)_{int}) = L'$, finishing the proof.
\end{proof}

\newpage
\section{Crystalline Galois representations and Breuil-Kisin modules}
\label{sec:BKCrys}

The goal of this section is to relate lattices in crystalline Galois representations to Breuil-Kisin modules over a Breuil-Kisin prism $\mathfrak{S}$ attached to $\mathcal{O}_K$, as in \cite{Kisin}. More precisely, in \S \ref{ss:FFKinfinity}, we give a direct  proof that the \'etale realization functor for $F$-crystals on $\mathfrak{S}$ is fully faithful. Using this result as well as our Theorem~\ref{MainThm}, we recover the main full faithfulness results from \cite{Kisin} in \S \ref{ss:KisinFF}. Finally, in \S \ref{ss:LogConnBK}, we explain why the $F$-crystals over $\mathfrak{S}$ arising from prismatic $F$-crystals over $\mathcal{O}_K$ (or equivalently crystalline Galois representations by Theorem~\ref{MainThm}) admit a natural logarithmic connection over the local ring at the Hodge-Tate point of $\mathfrak{S}$ in characteristic $0$; the argument there shows that the logarithmic connection has a natural integral avatar (in the form of a descent isomorphism over the  ring $\mathfrak{S}^{(1)}$, see Construction~\ref{PrismaticCechNerve}), which we hope shall shed some light on the integrality properties of the connection $N_\nabla$ from \cite[Corollary 1.3.15]{Kisin}.

\begin{notation}[A Breuil-Kisin prism over $\mathcal{O}_K$]
\label{NotBK}
Let $(\mathfrak{S}, E(u)) = (W\llbracket u \rrbracket, (E(u))$ be a Breuil-Kisin prism in $X_\Prism$ attached to the choice of a uniformizer $\pi \in \mathcal{O}_K$, as in Example~\ref{BKAinfPrismOK}. Choose a compatible system of $p$-power roots $\underline{\pi} \in \lim_{x \mapsto x^p} \mathcal{O}_C$ of $\pi$ in $\mathcal{O}_C$. This choice yields a unique $\delta$-$W$-algebra map $\mathfrak{S} \to A_{\inf}$ sending $u$ to $[\underline{\pi}]$ with the composition $\mathfrak{S} \to A_{\inf} \xrightarrow{\tilde{\theta}} \mathcal{O}_C$ being identified with $\mathfrak{S} \to \mathfrak{S}/E(u) \simeq \mathcal{O}_K \subset \mathcal{O}_C$; thus, we get a map $(\mathfrak{S},E(u)) \to (A_{\inf}, \ker(\tilde{\theta}))$ in $X_\Prism$. 
\end{notation}

\subsection{Full faithfulness of the \'etale realization over $\mathfrak{S}$}
\label{ss:FFKinfinity}

The goal of this subsection is to prove the following theorem:

\begin{theorem}[Kisin {\cite[Proposition 2.1.12]{Kisin}}]
\label{EtaleRealizeBK}
The \'etale realization functor
\[ T:\mathrm{Vect}^{\varphi}(\mathfrak{S}) \to \mathrm{Vect}^{\varphi}(\mathfrak{S}[1/E(u)]^{\wedge}_p) \]
is fully faithful.
\end{theorem}

The faithfulness is clear. For full faithfulness, we first argue that morphisms on right are meromorphic over $\mathrm{Spa}(\mathfrak{S})$ and in fact entire away from a finite set of pre-determined points depending only on $K$; we then use the Frobenius structure to get the extension at the missing points. To get the meromorphy in the first step, we use the adic spaces attached to $\mathfrak{S}$ and $A_{\inf}$ as well as the relation between them; the key meromorphy criterion is recorded in Lemma~\ref{BoundPoles}.

Let us begin by recalling the following (standard) coarse classification of points of $\mathrm{Spa}(A_{\inf})$.

\begin{lemma}[Points of $\mathrm{Spa}(A_{\inf})$]
\label{AinfPoint}
Fix a point $y \in \mathrm{Spa}(A_{\inf})$ and write $V=k(y)^+$. Then we must have one of the following possibilities:
\begin{enumerate}
\item The non-analytic point $y_k$: The elements $p$ as well as $q^{1/p^n}-1$ for $n \geq 0$ map to $0$ in $V$. In this case, $(k(y),k(y)^+) = (k,k)$. 
\item The crystalline point $y_{\crys}$: The elements $q^{1/p^n} - 1$ map to $0$ in $V$ for all $n \geq 0$ while $p$ gives a pseudo-uniformizer of $V$. In this case, $(k(y),k(y)^+) = (W(k)[1/p], W(k))$.
\item The \'etale point $y_{et}$: The element $p$ maps to $0$ in $V$ while the elements $q^{1/p^n}-1$ give pseudouniformizers in $V$ for all $n \geq 0$. In this case, $(k(y),k(y)^+) = (C^\flat, \mathcal{O}_{C^\flat})$.
\item The remaining points: Both $p$ as well as $q^{1/p^n}-1$ for $n \gg 0$ give pseudouniformizers in $V$. In this case, we have $p/(q^{1/p^n}-1) \in V$ for $n \gg 0$. 
\end{enumerate}
In particular, we have
\begin{equation}
\label{SpaAinfPointeq}
 \mathrm{Spa}(A_{\inf})^{an} - \{y_{\crys}\} = \bigcup_n \mathrm{Spa}(A_{\inf})\left(\frac{p}{q^{1/p^n}-1}\right).
 \end{equation}
\end{lemma}
\begin{proof}
The last assertion in the lemma follows by observing that the points appearing in (2), (3) or (4) lie in the right side of \eqref{SpaAinfPointeq}. For the rest, it is enough to show that any analytic point $y \in \mathrm{Spa}(A_{\inf})^{an}$ falls into one of the three possibilites described in (2), (3) and (4). Fix one such $y$ for the rest of the proof.

By continuity, each element of the set $S = \{p, \{q^{1/p^n}-1\}_{n \geq 0}\}$ is topologically nilpotent in $V$. Moreover, not all of these elements can be $0$ as the kernel of the valuation of defined by $y$ is not open since $y$ is analytic. Thus, at least one element of $S$ must be a pseudouniformizer in $V$. It is then easy to see that if we are not in case (2) or (3), then both $p$ as well as $q^{1/p^n}-1$ for $n \gg 0$ give pseudouniformizers in $V$: for $m \geq n$, we have $(q^{1/p^m}-1) \mid (q^{1/p^n}-1)$, so if the latter is nonzero the same holds true for the former. In this case, it remains to check the last property in (4), i.e., that $p/(q^{1/p^n}-1) \in V$ for $n \gg 0$. Assume this is false.  Then we must have $q^{1/p^n}-1 \in pV$ for all $n \geq 0$ as $V$ is a valuation ring. Choose some $N \gg 0$ such that $q^{1/p^N}-1 \neq 0$ in $V$. Now the ring $V$ is $p$-adically separated as $p$ is a pseudouniformizer in $V$, so we can choose some $k \geq 0$ such that $q^{1/p^N}-1 \notin p^{k+1} V$. But we know that $q^{1/p^{N+k}} \equiv 1 \mod pV$; raising to the $p^k$-th power then shows that $q^{1/p^N} \equiv 1 \mod p^{k+1}V$, which is a contradiction.
\end{proof}

\begin{notation}[The adic spaces $\mathcal{X}$ and $\mathcal{Y}$]
\label{AdicXY}
Consider the adic space $\mathcal{Y} = \mathrm{Spa}(A_{\inf})^{an} - \{y_{\crys}\}$, where $y_{\crys}$ is as in Lemma~\ref{AinfPoint}. Write $y_0 = y_C \in \mathcal{Y}$ for the Hodge-Tate point, and let $y_n = \varphi^n(y_0)$ for all $n \in \mathbf{Z}$. 

Let $x_{\crys} \in \mathrm{Spa}(\mathfrak{S})^{an}$ be point defined by $u=0$; this is also the image of $y_{\crys}$ under the natural map $\mathrm{Spa}(A_{\inf}) \to \mathrm{Spa}(\mathfrak{S})$, and moreover the inverse image of $\{x_{\crys}\} \subset \mathrm{Spa}(\mathfrak{S})$ is exactly $\{y_{\crys}\} \subset \mathrm{Spa}(A_{\inf})$. Let $\mathcal{X} = \mathrm{Spa}(\mathfrak{S})^{an} - \{x_{\crys}\}$, so we have an induced map $\pi:\mathcal{Y} \to \mathcal{X}$ of adic spaces. Write $x_{et} = \pi(y_{et}) \in \mathcal{X}$, so $x_{et}$ is defined $\mathfrak{S} \to \mathfrak{S}/p[1/u]$. Finally, write $x_n = \pi(y_n)$ for $n \in \mathbf{Z}$; explicitly, these can be described as follows:
\begin{itemize}
\item If $n \geq 0$, then $x_n$ is defined by $\mathfrak{S} \xrightarrow{\varphi^n} \mathfrak{S} \to \mathfrak{S}/E(u) \simeq \mathcal{O}_K \subset K$
\item If $n < 0$, then $x_n$ is defined by $\mathfrak{S} \to \mathfrak{S}/\varphi^n E(u)[1/p] =: K_n$.
\end{itemize}
Finally, write $B = \{x_n\}_{n \geq 1} \subset \mathcal{X}$ for the displayed set of {\em positive} $\varphi$-translates of the Hodge-Tate point.
\end{notation}

\begin{lemma}[Some fibres of $\mathcal{Y} \to  \mathcal{X}$]
\label{SingletonFib}
Consider the subset $B := \{x_m\}_{m \geq 1} \subset \mathcal{X}$ from Notation~\ref{AdicXY}. Then the subset $B' \subset B$ consisting all $x_m \in B$ with $\# \pi^{-1}(x_m) = 1$ is finite. 
\end{lemma}

The proof below shows that $B' = \emptyset$ if $K$ is unramified or more generally if the absolute ramification index $e(K/K_0)$ of $K$ is $< p$. On the other hand, if $K=K_0(p^{1/p})$, then $B' \neq \emptyset$ (see Example~\ref{PolesBadSetEx}).

\begin{proof}
For each $n \geq 1$, one has a factorization
\[ \mathcal{Y} \xrightarrow{\varphi^{-n} \circ \pi} \mathcal{X} \xrightarrow{\varphi^n} \mathcal{X} \]
of $\pi$ with both maps being surjective. It is thus enough to show that $(\varphi^n)^{-1}(x_n) \subset \mathcal{X}$ is a singleton for finitely many integers $n \geq 1$. As $x_n = \varphi^n(x_0)$, it is equivalent to show that there exist only finitely many integers $n \geq 1$ such that $(\varphi^n)^{-1}(x_n) = \{x_0\}$. The point $x_n$ is exactly the vanishing locus of the ideal $I_n := \ker(\mathfrak{S} \xrightarrow{\varphi^n} \mathfrak{S} \to \mathfrak{S}/E(u) = \mathcal{O}_K)$, so $(\varphi^n)^{-1}(x_n)$ is the vanishing locus of $\varphi^n(I_n) \mathfrak{S}$. Thus, our task is to show that there exist only finitely many $n \geq 1$ such that $\varphi^n(I_n) = (E(u))$ as ideals of $\mathfrak{S}[1/p]$. But for any such $n$, the map $L := \mathfrak{S}[1/p]/I_n \xrightarrow{\varphi^n} \mathfrak{S}[1/p]/E(u) = K$ is a finite flat degree $p^n$ extension of rings (by base change from the same property for $\varphi^n:\mathfrak{S} \to \mathfrak{S}$) that lives over the map $\varphi^n:K_0 \to K_0$; this forces $K/L$ to be a totally ramified degree $p^n$ extension of discretely valued fields, which is clearly only possible for finitely many values of $n$, proving the lemma.
\end{proof}

\begin{lemma}[Analytic criteria for membership in {$A_{\inf} \subset A_{\inf}[1/(q-1)]$}]
\label{Ainfentire}
Let $h \in A_{\inf}[1/(q-1)]$. Assume $h$ defines an entire function in a neighbourhood of each $y_n \in \mathcal{Y}$ for $n \geq 1$. Then $h \in A_{\inf}$.
\end{lemma}

In particular, this lemma implies that $1/u \notin A_{\inf}[1/(q-1)]$, whence $(q-1)^n \notin uA_{\inf}$ for any $n \geq 0$, and thus that $(q-1) \notin \sqrt{uA_{\inf}}$; as a contrast, note that the $(p,I)$-adic completion of $\sqrt{uA_{\inf}} \subset A_{\inf}$ coincides with $\ker(A_{\inf} \to W(k))$, which certainly contains $q-1$.

\begin{proof}
Write $h = \frac{g}{(q-1)^m}$ for $g \in A_{\inf}$. We must show that $(q-1)^m \mid g$. If $m=0$, there is nothing to show, so we may assume $m \geq 1$. For each $n \geq 1$, the image of $g = (q-1)^m h \in A_{\inf}$ under the natural map $A_{\inf} \to \mathcal{O}_{\mathcal{Y},y_n}$ lies in the maximal ideal of the target dvr: indeed, $h$ is entire at $y_n$, so the claim follows as $(q-1)^m$ vanishes at $y_n$. Thus, $g$ maps to $0$ under the map
\[ A_{\inf} \to \prod_{n \geq 1} k(y_n).\]
But this map is identified with the natural map
\[ A_{\inf} \simeq \lim_F W(\mathcal{O}_C) \xrightarrow{pr} W(\mathcal{O}_C) \xrightarrow{a \mapsto (\varphi^n(a))_{n \geq 0}} \prod_{n \geq 0} \mathcal{O}_C \subset \prod_{n \geq 0} C\]
defined by Witt vector functoriality. The composite $A_{\inf} \to W(\mathcal{O}_C)$ appearing above has kernel exactly $(q-1)$ by \cite[Lemma 3.23]{BMS1}, while the rest of the maps are injective, so it follows that $(q-1) \mid g$. We may then replace the expression $h = \frac{g}{(q-1)^m}$ with $h = \frac{ g/(q-1) }{ (q-1)^{m-1}}$ and continue inductively to prove the lemma.
\end{proof}

\begin{lemma}[Detecting meromorphy over $\mathfrak{S}$]
\label{BoundPoles}
Fix $f \in \mathfrak{S}[\frac{1}{u}]^{\wedge} \cap A_{\inf}[\frac{1}{q-1}]$ (where the intersection is as subrings of $W(C^\flat)$). Then $f$ is meromorphic over $\mathfrak{S}$, with poles contained in the set $B'$ from Lemma~\ref{SingletonFib}.
\end{lemma}
\begin{proof}
Fix some $c \geq 0$ such that $(q-1)^c \cdot f \in A_{\inf}$. 

First, we prove that $f$ is meromorphic on $\mathcal{X}$ with poles contained in $B$. For this, we may work with each affinoid open neighbourhood $U$ of $x_{et} \in \mathcal{X}$ separately. For any such $U$, the intersection $B \cap U$ is finite, so our task is to show that $f$ becomes entire on $U$ after multiplication by an element of $\mathfrak{S}$ with zeroes only at $B \cap U$. But $\pi^{-1}(U \cap B) \cap V(q-1) \subset \mathcal{Y}$ is a finite subset of $\{y_n\}_{n \geq 1}$. For any $y \in \pi^{-1}(U \cap B)$, the extension $\mathcal{O}_{\mathcal{X},\pi(y)} \to \mathcal{O}_{\mathcal{Y},y}$ is an extension of discrete valuation rings with finite ramification index. We may thus choose $g \in \mathfrak{S}$ whose valuation in $\mathcal{O}_{Y,y}$ exceeds that of $(q-1)^c$ for each $y \in \pi^{-1}(U \cap B)$. By our choice of $c$, it follows that $gf \in \mathfrak{S}[1/u]^{\wedge}_p$ defines an entire function on $\pi^{-1}(U) \subset \mathcal{Y}$. By \cite[Definition 2.2.13, Lemma 2.3.5, Remark 2.3.8]{KedlayaSlope}, it follows that $gf$ is analytic on $U$. As this holds true for all $U$, we conclude that $f$ is meromorphic on $\mathcal{X}$ with poles in $B$.

Next, we observe that $f$ only has poles in $B' \subset B$. Indeed, for any $x_n \in B$, the preimage $\pi^{-1}(x_n)$ intersects $V(q-1)$ in the singelton $\{\varphi^n(y_C)\}$. If $x_n \in B-B'$, the $\pi^{-1}(x_n)$ must have at least one point $y'  \in \mathcal{Y} - V(q-1)$ by definition of $B'$. But $f$ is entire on $\mathcal{Y} - V(q-1)$, so $f$ is entire at $y'$, and thus also at $x_n = \pi(y')$ (e.g., by comparing the map on local rings). 

Combining the previous two paragraphs, we can find some $g \in \mathfrak{S}$ with zeroes only at points of the finite set $B'$ such that $gf$ is entire on $\mathcal{X}$. It remains to check that $gf$ is entire at $x_{\crys}$ as well, i.e., lies in $\mathfrak{S}$. For this, we first observe that $gf \in A_{\inf}$ by Lemma~\ref{Ainfentire}. It then suffices to show that $\mathfrak{S} = \mathfrak{S}[1/u]^{\wedge}_p \cap A_{\inf}$ as subrings of $W(C^\flat)$, i.e., that the map
\[ \mathfrak{S}  \to \mathfrak{S}[1/u]^{\wedge}_p \times^h_{W(C^\flat)} A_{\inf} \]
induces an isomorphism on $H^0$. As $W(C^\flat)/\mathfrak{S}[1/u]^{\wedge}_p$ is $p$-torsionfree, one can check the above statement after reducing modulo $p$, where it reduces to showing the bijectivity of 
\[ k\llbracket u \rrbracket \to k((u)) \cap \mathcal{O}_{C^\flat},\]
which is clear for valuative reasons.
\end{proof}

The following example shows  that the functions $f$ appearing in Lemma~\ref{BoundPoles} may indeed have some poles.

\begin{example}
\label{PolesBadSetEx}
Let $K=\mathbf{Q}_p(p^{1/p})$ and $\pi = p^{1/p}$, so $E(u) = u^p-p$. By construction, the embedding $\mathfrak{S} \to A_{\inf}$ carries $E(u)$ to a generator of $\ker(\tilde{\theta})$, i.e., to a unit multiple of $[p]_q$. Twisting by Frobenius, this map sends $u-p$ to a unit multiple of $[p]_{q^{1/p}}$. In particular, we have $(u-p) \mid (q-1) \in A_{\inf}$, whence $f := \frac{1}{u-p} \in A_{\inf}[\frac{1}{q-1}]$. We also have $f \in \mathfrak{S}\langle p/u^2 \rangle[1/u]$ as we can write
\[ \frac{1}{u-p} = \frac{1}{u} \cdot \frac{1}{1-(p/u^2)u} = \frac{1}{u} \cdot \sum_{n \geq 0} u^n (p/u^2)^n.\]
In particular, this function $f \in \mathfrak{S}[\frac{1}{u-p}]$ lies in $\mathfrak{S}\langle p/u^2 \rangle[1/u] \cap A_{\inf}[\frac{1}{q-1}]$ and has a pole of order $1$ at $x_1$. 
\end{example}

We can now prove the desired full faithfulness theorem for the \'etale realization over $\mathfrak{S}$:

\begin{proof}[Proof of Theorem~\ref{EtaleRealizeBK}]
By passing to a suitable internal $\mathrm{Hom}$, it suffices to show the following: for any $M \in \mathrm{Vect}^{\varphi}(\mathfrak{S})$, the natural map
\[ M^{\varphi=1} \to \left(M[1/E(u)]^{\wedge}_p\right)^{\varphi=1}\]
is bijective. Injectivity is clear as $M$ is finite free over $\mathfrak{S}$. For surjectivity, fix some $\alpha \in \left(M[1/E(u)]^{\wedge}_p\right)^{\varphi=1}$. By \cite[Lemma 4.26]{BMS1}, we know that the image of $\alpha$ in $(M \otimes_{\mathfrak{S}} A_{\inf})[1/E(u)]^{\wedge}_p \simeq M \otimes_{\mathfrak{S}} W(C^\flat)$ lies inside $M \otimes_{\mathfrak{S}} A_{\inf}[\frac{1}{q-1}]$. Using Lemma~\ref{BoundPoles} as well as the fact that $\mathfrak{S}$ is a UFD, we can then write $\alpha = \beta/g$ with $\beta \in M$ and $g \in \mathfrak{S}$ having zeroes only at points of $B = \{x_n\}_{n \geq 1}$ (and in fact at points of finite set $B' \subset B$, but this will not simplify the argument). We shall check that $\alpha \in M$. Since $g$ only has zeroes in $B$, it suffices to show that $\alpha \in M^{\wedge}_{x_m}$ for all $m \geq 1$. 

Regard $\varphi_M$ as an isomorphism $\varphi_M:(\varphi^* M)[1/I] \simeq M[1/I]$. Iterating $r$ times for $r \geq 1$, this gives isomorphisms 
\[ \varphi^r_M:M \otimes_{\mathfrak{S},\varphi^r} \mathfrak{S}[1/I_r] =: (\varphi^{r,*} M)[1/I_r] \simeq M[1/I_r],\]
where $I_r = I \varphi(I) ... \varphi^{r-1}(I)$. The defining property $\varphi_M(\alpha \otimes 1) = \alpha$ iterated $r$ times then gives the equality
\[ \varphi_M^r(\beta \otimes 1) = \frac{\varphi^r(g)}{g} \beta \in M[1/I_r] \quad \forall r \geq 0.\]
As the ideals $I_r$ for all $r \geq 1$ are invertible at the points $x_m$ for all $m \geq 1$, it follows that
\[  \frac{\varphi^r(g)}{g} \beta \in M_{x_m}^{\wedge}\]
for all $r \geq 1$ and $m \geq 1$. But $\varphi^r(g)$ is invertible at $x_m$ for $r \gg 0$ (depending on $m$): the function $g$ has only finitely many poles, and applying $\varphi(-)$ moves a pole at $x_k$ to a pole at $x_{k-1}$ for all $k \in \mathbf{Z}$.  It follows from the above that $\alpha = \frac{\beta}{g} \in M^{\wedge}_{x_m}$ for all $m \geq 1$, as wanted.
\end{proof}

\subsection{Kisin's full faithfulness results}
\label{ss:KisinFF}

In this section, we prove two full faithfulness results originally shown in \cite{Kisin}. First, we relate crystalline $G_K$-representations of $F$-crystals over $\mathfrak{S}$ using Theorem~\ref{MainThm}:

\begin{theorem}[Kisin {\cite[Corollary 1.3.15]{Kisin}}]
\label{KisinFullyFaithfulBK}
Consider the functor
\[  D_{\mathfrak{S}}:\mathrm{Rep}^{\crys}_{\mathbf{Z}_p}(G_K) \to \mathrm{Vect}^{\varphi}(\mathfrak{S}) \]
obtained by postcomposing the inverse to the equivalence in Theorem~\ref{MainThm} with evaluation on the Breuil-Kisin prism $(\mathfrak{S},E(u)) \in X_\Prism$. This functor is fully faithful.
\end{theorem}

The proof below relies on certain simple lemmas from Kisin's work. Namely, we rely on the arguments in  \cite[\S 1.2.4 - \S 1.2.8]{Kisin} that use Dwork's trick to enlarge the radius of convergence of an isomorphism of $\varphi$-modules on a disc. We do not use the connection $N_\nabla$ from \cite{Kisin}.

\begin{proof}
As $(\mathfrak{S},E(u)) \in X_\Prism$ covers the final object, the functor $D_{\mathfrak{S}}$ is faithful. For fullness, we first prove the statement with $\mathbf{Q}_p$-coefficients by working near the Hodge-Tate point, and then pass to integral coefficients by working at the \'etale point.

First, we construct a natural functor
\[\mathcal{D}:\mathrm{Vect}^{\varphi}(\mathfrak{S}\langle I/p \rangle[1/p]) \to \mathrm{MF}^{\varphi}(K),\]
by undoing Construction~\ref{FiltPhiCrys}. Given $(\mathfrak{M},\varphi_{\mathfrak{M}}) \in \mathrm{Vect}^{\varphi}(\mathfrak{S}\langle I/p\rangle[1/p])$, we set $D(\mathfrak{M}) = \mathfrak{M} \otimes_{\mathfrak{S}\langle I/p \rangle[1/p]} K_0 \in \mathrm{Vect}(K_0)$, where the map $\mathfrak{S}\langle I/p \rangle[1/p] \to K_0$ is induced from $\mathfrak{S}[1/p]/u\mathfrak{S}[1/p] \simeq K_0[1/p]$. As this map is $\varphi$-equivariant, the Frobenius $\varphi_{\mathfrak{M}}$ naturally gives an isomorphism $\varphi_{D(\mathfrak{M})}:\varphi_{K_0}^* D(\mathfrak{M}) \simeq D(\mathfrak{M})$, endowing $D(\mathfrak{M})$ with an $F$-isocrystal structure. Moreover, since $\varphi^* \mathfrak{M}$ is a unit $\varphi$-module over $\mathfrak{S}\langle I/p \rangle[1/p]$, the standard Frobenius trick shows that there is a unique $\varphi$-equivariant isomorphism 
\[ \varphi^\ast D(\mathfrak{M}) \otimes_{K_0} \mathfrak{S}\langle I/p \rangle[1/p] \simeq (\varphi^* \mathfrak{M})\langle I/p \rangle[1/p] \in \mathrm{Vect}^{\varphi}(\mathfrak{S}\langle I/p \rangle[1/p])\] 
lifting the identity after base change to $K_0$ (see \cite[Lemma 1.2.6]{Kisin}). Base changing now to $\mathfrak S[1/p]^\wedge_I$, the Frobenius $\varphi_{\mathfrak{M}}$ gives an isomorphism 
\[ \varphi^\ast D(\mathfrak{M}) \otimes_{K_0} \mathfrak S[1/p]^\wedge_I[1/I] \simeq \mathfrak{M}\otimes_{\mathfrak{S}\langle I/p\rangle[1/p]} \mathfrak S[1/p]^\wedge_I[1/I] \in \mathrm{Vect}(\mathfrak S[1/p]^\wedge_I[1/I]).\] 
Transporting the $I$-adic filtration on $\mathfrak{M}\otimes_{\mathfrak{S}\langle I/p\rangle[1/p]} \mathfrak S[1/p]^\wedge_I$ along this isomorphism and taking its image down along
\[ \varphi^\ast D(\mathfrak{M}) \otimes_{K_0} \mathfrak S[1/p]^\wedge_I \xrightarrow{I=0} \varphi^\ast D(\mathfrak{M}) \otimes_{K_0} K =: D(\mathfrak{M})\otimes_{K_0,\varphi} K\] 
then gives a natural filtration $\mathrm{Fil}^*$ on $D(\mathfrak{M})\otimes_{K_0,\varphi} K$. The triple $(D(\mathfrak{M}),\varphi_{D(\mathfrak{M})},\mathrm{Fil}^*)$ then defines an object of $\mathrm{MF}^{\varphi}(K)$, yielding the promised functor $\mathcal{D}$.  Note that since $\mathcal{D}(\mathfrak{M})$ naturally recovers $\varphi^* \mathfrak{M}$, the functor $\mathcal{D}$ is faithful. Moreover, it is easy to see (see \cite[Proposition 1.2.8]{Kisin}) that the functor $\mathcal{D}(-)$ provides a left-inverse to the composition of the functor $D \mapsto \mathcal{M}(D)_{\langle I/p \rangle}$ from Construction~\ref{FiltPhiCrys} with evaluation over $\mathfrak{S}$.

Now consider the composition
\[ \mathrm{Rep}^{\crys}_{\mathbf{Z}_p}(G_K)[1/p] \xrightarrow{D_{\mathfrak{S}}[1/p]} \mathrm{Vect}^{\varphi}(\mathfrak{S})[1/p] \xrightarrow{- \otimes_{\mathfrak{S}} \mathfrak{S}\langle I/p \rangle[1/p]}\mathrm{Vect}^{\varphi}(\mathfrak{S}\langle I/p \rangle[1/p]) \xrightarrow{\mathcal{D}} \mathrm{MF}^{\varphi}(K).  \]
Unwinding definitions and using the last sentence in the previous paragraph, this composition coincides with Fontaine's functor $D_{\crys}$ and is thus fully faithful. On the other hand, each functor in the above composition is faithful: this was shown in the first paragraph for $D_{\mathfrak{S}}[1/p]$, in the last paragraph for $\mathcal{D}$, and is clear for the base change functor by injectivity of $\mathfrak{S} \to \mathfrak{S}\langle I/p \rangle[1/p]$. This shows that $D_{\mathfrak{S}}[1/p]$ must be fully faithful, proving the theorem up to inverting $p$. 

It now remains to prove that $D_{\mathfrak{S}}$ is itself fully faithful. Fix  $L,L' \in \mathrm{Rep}^{\crys}_{\mathbf{Z}_p}(G_K)$ and a map $\alpha:D_{\mathfrak{S}}(L) \to D_{\mathfrak{S}}(L')$. By the rational version of the theorem, there exists some $n \geq 0$ such that $p^n \alpha = D_{\mathfrak{S}}(a)$ for a unique map $a:L \to L'$ of $G_K$-representations. We must show that $a$ is divisible by $p^n$ as a map of $G_K$-representations. In fact, it suffices to show divisibility merely as a map of $\mathbf{Z}_p$-modules: the resulting map $a/p^n$ is then automatically $G_K$-equivariant as $a$ is so. But the forgetful functor $\mathrm{Rep}^{\crys}_{\mathbf{Z}_p}(G_K) \to \mathrm{Vect}(\mathbf{Z}_p)$ factors over $D_{\mathfrak{S}}$: indeed, postcomposing $D_{\mathfrak{S}}$ with the \'etale realization 
\[ \mathrm{Vect}^{\varphi}(\mathfrak{S}) \to \mathrm{Vect}^{\varphi}(\mathfrak{S}[1/E(u)]^{\wedge}_p) \simeq \mathrm{Rep}_{\mathbf{Z}_p}(G_{K_\infty})\]
gives the obvious restriction map 
\[ \mathrm{Rep}^{\crys}_{\mathbf{Z}_p}(G_K) \to \mathrm{Rep}_{\mathbf{Z}_p}(G_{K_\infty}),\] 
which certainly factors the forgetful functor for the left side. It is then clear that $a$ is divisible by $p^n$ as a map of $\mathbf{Z}_p$-modules, as wanted.
\end{proof}

Using the above as well as Theorem~\ref{EtaleRealizeBK}, we deduce the following full faithfulness result for Galois representations that was conjectured by Breuil and proven by Kisin:

\begin{corollary}[Kisin {\cite[Corollary 2.1.14]{Kisin}}]
\label{KisinBreuil}
The restriction functor
\[ \mathrm{Rep}^{\crys}_{\mathbf{Z}_p}(G_K) \to \mathrm{Rep}_{\mathbf{Z}_p}(G_{K_\infty})\]
is fully faithful. 
\end{corollary}

An alternative direct proof of this result was also given in \cite{BeilinsonRiberioKisin}.

\begin{proof}
We can factor the restriction functor factors as a composition
\[ \mathrm{Rep}^{\crys}_{\mathbf{Z}_p}(G_K) \xrightarrow{D_{\mathfrak{S}}} \mathrm{Vect}^{\varphi}(\mathfrak{S}) \xrightarrow{- \otimes_{\mathfrak{S}} \mathfrak{S}[1/E(u)]^{\wedge}} \mathrm{Vect}^{\varphi}(\mathfrak{S}[1/E(u)]^{\wedge}) \xrightarrow{\simeq} \mathrm{Rep}_{\mathbf{Z}_p}(G_{K_\infty}),\]
where $D_{\mathfrak{S}}$ is the fully faithful functor from Theorem~\ref{KisinFullyFaithfulBK}, the base change functor $(- \otimes_{\mathfrak{S}}  \mathfrak{S}[1/E(u)]^{\wedge})$ is fully faithful by Theorem~\ref{EtaleRealizeBK}, and the last equivalence is from Corollary~\ref{LocSysLaurentFCrysIsog}. As each constituent functor is fully faithful, so is the composition.
\end{proof}

\begin{remark}[Compatibility with Kisin's work]
Strictly speaking, we haven't yet shown that the functor in Theorem~\ref{KisinFullyFaithfulBK} coincides with the one from \cite{Kisin}. To check this, call the latter $D_{\mathfrak{S}}'$. Thanks to Corollary~\ref{KisinBreuil}, to show $D_{\mathfrak{S}} \simeq D_{\mathfrak{S}}'$, it suffices to show that the two functors
\[ \mathrm{Rep}^{\crys}_{\mathbf{Z}_p}(G_K) \xrightarrow{D_{\mathfrak{S}}} \mathrm{Vect}^{\varphi}(\mathfrak{S}) \xrightarrow{- \otimes_{\mathfrak{S}} \mathfrak{S}[1/E(u)]^{\wedge}} \mathrm{Vect}^{\varphi}(\mathfrak{S}[1/E(u)]^{\wedge}) \xrightarrow{\simeq} \mathrm{Rep}_{\mathbf{Z}_p}(G_{K_\infty}),\]
and
\[ \mathrm{Rep}^{\crys}_{\mathbf{Z}_p}(G_K) \xrightarrow{D_{\mathfrak{S}}'} \mathrm{Vect}^{\varphi}(\mathfrak{S}) \xrightarrow{- \otimes_{\mathfrak{S}} \mathfrak{S}[1/E(u)]^{\wedge}} \mathrm{Vect}^{\varphi}(\mathfrak{S}[1/E(u)]^{\wedge}) \xrightarrow{\simeq} \mathrm{Rep}_{\mathbf{Z}_p}(G_{K_\infty})\]
are naturally isomorphic. But these are both identified with restriction along $G_K \subset G_{K_\infty}$ by unwinding definitions, so the compatibility follows. 
\end{remark}

\begin{remark}[Liu's compatibility for different uniformizer choices]
\label{LiuCompat}
Liu's paper \cite{LiuKisinCompat} studies the dependence of the functor $D_{\mathfrak{S}}$ from Theorem~\ref{KisinFullyFaithfulBK} on the choice of the uniformizer $\pi$. To formulate his theorem, fix two uniformizers $\pi$ and $\pi'$ equipped with a compatible system of $p$-power roots $\underline{\pi},\underline{\pi}' \in \mathcal{O}_C^\flat$; write $(\mathfrak{S}_{\pi}, (E))$ and $(\mathfrak{S}_{\pi'},(E'))$ for the corresponding Breuil-Kisin prisms. The choices of $\underline{\pi}, \underline{\pi}'$ determine unique  $\delta$-$W$-algebra maps $\mathfrak{S}_{\pi} \to A_{\inf}$ and $\mathfrak{S}_{\pi'} \to A_{\inf}$. Liu shows the following: for any $L \in \mathrm{Rep}^{\crys}_{\mathbf{Z}_p}(G_K)$, the modules $D_{\mathfrak{S}_{\pi}}(L) \otimes_{\mathfrak{S}_{\pi}} A_{\inf}$ and $D_{\mathfrak{\pi}'}(L) \otimes_{\mathfrak{S}_{\pi'}} A_{\inf}$ are identified in a $(\varphi,G_K)$-equivariant manner; here the implicit $G_K$-action on the base change is not automatic from \cite{Kisin}, and in fact is the main result of \cite{LiuNoteLattice} (but it is already uniquely determined by the given $G_K$-action on the \'etale realization). Such results are now automatic from the prismatic perspective: they follow from the crystal property of evaluations of prismatic $F$-crystals applied to the maps 
\[ (\mathfrak{S}_{\pi}, (E)) \to (A_{\inf},\ker(\tilde{\theta})) \gets (\mathfrak{S}_{\pi'}, (E'))\] 
in $\mathrm{Spf}(\mathcal{O}_K)_\Prism$ together with the observation that $(A_{\inf},\ker(\tilde{\theta}))$ is a $G_K$-equivariant object of $\mathrm{Spf}(\mathcal{O}_K)_\Prism$.
\end{remark}

\subsection{The logarithmic connection}
\label{ss:LogConnBK}

In this section, we explain (Corollary~\ref{LogConnectionBK}) why the $\mathfrak{S}$-module attached to any crystal of vector bundles on $X_\Prism$ (without any Frobenius data) carries a natural logarithmic connection after base change to the local ring $\mathcal{S} := \mathfrak{S}[1/p]^{\wedge}_{(E(u))}$ at the Hodge-Tate point. In fact, the connection is directly obtained from the descent data underlying the crystal: our strategy is to understand the descent data on the $\mathfrak{S}$-module attached to such a crystal to good enough accuracy for obtaining the logarithmic connection.\footnote{One could also construct this log-connection on a larger region stable under the Frobenius map. For prismatic $F$-crystals, the resulting log-connection will then automatically (by functoriality of the construction) commute with the Frobenius. Kisin \cite{Kisin} shows the uniqueness of such a log-connection, so it agrees with his construction.} To do so, we first make the relevant rings more explicit:

\begin{construction}[The prismatic Cech nerve of $\mathfrak{S}$]
\label{PrismaticCechNerve}
Let $\mathfrak{S}^{(\bullet)}$ denote the cosimplicial ring obtained by taking the Cech nerve of $(\mathfrak{S}, E(u))$ in $X_\Prism$. The multiplication map on $\mathfrak{S}$ induces a surjection 
\[ \mu_\Prism:\mathfrak{S}^{(\bullet)} \to \mathfrak{S}\] 
of cosimplicial $\delta$-rings (where the target is a constant diagram); write $J_\Prism^{(\bullet)} \subset \mathfrak{S}^{(\bullet)}$ for the kernel of this map.  By construction, there is a natural map $\mathfrak{S}^{\otimes [\bullet]} \to \mathfrak{S}^{(\bullet)}$ of cosimplicial rings, where $\mathfrak S^{\otimes [i]}$ is the $i$-fold tensor product of $\mathfrak S$ over $W$. Each term $\mathfrak{S}^{(i)}$ is a transversal prism over $X$, i.e.~$\{p,E(u)\}$ gives a regular sequence of length $2$ on each $\mathfrak{S}^{(i)}$. (See \cite{ALBPrismaticDieudonne} for the term ``transversal prism''.)

For future use, let us describe the first two terms explicitly. Clearly $\mathfrak{S}^{(0)} = \mathfrak{S}$. In degree $1$, we have
\[ \mathfrak{S}^{(1)} = W\llbracket u,v\rrbracket \{\frac{u-v}{E(u)}\}^{\wedge}_{(p,E(u))}.\]
Note that since $u = v \mod E(u) \mathfrak{S}^{(1)}$, we also have $E(v) = 0 \mod E(u)\mathfrak{S}^{(1)}$, whence $E(v)/E(u) \in \mathfrak{S}^{(1)}$ is a unit by the irreducibility lemma on distinguished elements, so the right side above can also be described by replacing $E(u)$ with $E(v)$. By construction, we have
\[ \delta^n\left(\frac{u-v}{E(u)}\right) \in J_\Prism^{(1)} \quad \forall n \geq 0.\]
It will also be convenient to give a name to the following non-completed version:
\[ \mathfrak{S}^{(1),nc} := W\llbracket u,v\rrbracket \{\frac{u-v}{E(u)}\}.\]
The surjection $\mu$ restricts to a surjection $\mathfrak{S}^{(1),nc} \to \mathfrak{S}$ with kernel $J_\Prism^{(1),nc}$; moreover, $\mathfrak{S}^{(1)}$ (resp. $J_\Prism$) can be recovered by $(p,E(u))$-completion of $\mathfrak{S}^{(1),nc}$ (resp. $J_\Prism^{(1),nc}$).
\end{construction}

\begin{construction}[The logarithmic Cech nerve on the generic fibre]
Regard the discrete valuation ring $\mathcal{S} := \mathfrak{S}[1/p]^{\wedge}_{E(u)}$ as an adic ring with the $E(u)$-adic topology; endow it with the prelog structure defined by $E(u)^{\mathbf{N}}$. Write $\mathcal{S}^{(\bullet)}$ for its Cech nerve in the log infinitesimal site of $\mathcal{S}$. There is a surjective multiplication map
\[ \mu_{\log}:\mathcal{S}^{(\bullet)}\to \mathcal{S}  \]
with kernel $J_{\log}^{(\bullet)}$; note that $\mathcal{S}^{(\bullet)}$ is $J^{(\bullet)}_{\log}$-adically complete, and hence (by the $E(u)$-adic completeness of $\mathcal{S}^{(\bullet)}/J_{\log}^{(\bullet)} \simeq \mathcal{S}$) also $E(u)$-adically complete. By construction, there is a natural map  $\mathfrak{S}^{\otimes [\bullet]} \to \mathcal{S}^{(\bullet)}$ of cosimplicial rings. Moreover, general nonsense on log infinitesimal cohomology shows that for each $n \geq 0$, each term of the cosimplicial $\mathcal{S}$-module $(J_{\log}^{(\bullet)})^n/(J_{\log}^{(\bullet)})^{n+1}$ is free over $\mathcal{S}^{(\bullet)}_{\log}/J_{\log}^{(\bullet)} \simeq \mathcal{S}$.

For future use, we make the objects explicit in low degrees. Clearly $\mathcal{S}^{(0)} = \mathcal{S}$. In degree $1$, we can identify $\mathcal{S}^{(1)}$ as the  completion of $W\llbracket u,v \rrbracket[1/p, \left(\frac{E(u)}{E(v)}\right)^{\pm 1}]$ along the kernel of the multiplication map to $\mathcal{S}$. 
\end{construction}

The crucial result in this section is the following relation:

\begin{lemma}[Relating the prismatic and logarithmic Cech nerves]
\label{PrismaticLogCech}
There is a natural map 
\[ \mathfrak{S}^{(\bullet)}\langle I/p \rangle[1/p] \to \mathcal{S}^{(\bullet)}_{\log}/(J^{(\bullet)}_{\log})^2\] 
of cosimplicial rings, extending the natural map $\mathfrak{S} \to \mathcal{S}$ in degree $0$. 
\end{lemma}
\begin{proof}
As $p$ is invertible on the target, it suffices to construct a map out of $\mathfrak{S}^{(\bullet)}\langle I/p \rangle$ as in the lemma. For notational ease, we shall construct the extension in degree $1$, i.e., we extend the canonical  map $\mathfrak{S}^{\otimes [1]} \to \mathcal{S}^{(1)}/(J_{\log}^{(1)})^2$ naturally across $\mathfrak{S}^{\otimes [1]} \to \mathfrak{S}^{(1)} \to \mathfrak{S}^{(1)}\langle I/p \rangle$; a similar argument applies to the higher degree terms as well. We shall construct this extension by extending successively along each map in the composition $\mathfrak{S}^{\otimes [1]} \to \mathfrak{S}^{(1),nc} \to \mathfrak{S}^{(1)} \to \mathfrak{S}^{(1)}\langle I/p \rangle$.

To extend to $\mathfrak{S}^{(1),nc}$, we first observe that $u = v \mod E(u) \mathcal{S}^{(1)}$: indeed, since $\left(E(u)/E(v)\right)^{\pm 1}$ exists in $\mathcal{S}^{(1)}$, both $u$ and $v$ are roots of the separable polynomial $E(T) \in K_0[T]$ in the ring $\mathcal{S}^{(1)}/E(u)$ and give same element in the quotient $\mathcal{S}/E(u)$ under the multiplication map, so the equality follows by Hensel's lemma as $\mathcal{S}^{(1)}/E(u)$ is complete along the kernel of the multiplication map to $K = \mathcal{S}/E(u)$. Since $E(u)$ is a nonzerodivisor on $\mathcal{S}^{(1)}$, it follows that $\frac{u-v}{E(u)} \in \mathcal{S}^{(1)}$. Moreover, $\varphi^n E(u)$ gives a unit in $K$ for $n \geq 1$, and hence is also a unit in the pro-infinitesimal thickening $\mathcal{S}^{(1)} \to \mathcal{S} \to K$. Consequently, we have $\{\varphi^n\left(\frac{u-v}{E(u)}\right)\}_{n \geq 0} \in \mathcal{S}^{(1)}$. As $p$ is invertible in $\mathcal{S}^{(1)}$, it follows that we also have $\{\delta^n\left(\frac{u-v}{E(u)}\right)\}_{n \geq 0} \in \mathcal{S}^{(1)}$. As $\mathfrak{S}^{(1),nc}$ is generated over $\mathfrak{S}^{\otimes [1]}$ by these elements subject to the ``obvious'' relations, we obtain an induced map 
\[ \mathfrak{S}^{(1),nc} \to \mathcal{S}^{(1)}\]
factoring the canonical map $\mathfrak{S}^{\otimes [1]} \to \mathcal{S}^{(1)}_{\log}$. 

Next, we extend to $\mathfrak{S}^{(1)}$. For this, observe that the compatibility with the multiplication map shows that extension constructed in the previous paragraph carries $J_\Prism^{(1),nc}$ into $J_{\log}^{(1)}$. Now $\mathfrak{S}^{(1),nc}/J_\Prism^{(1),nc} \simeq \mathfrak{S}$, so $\mathfrak{S}^{(1)}$ is obtained from $\mathfrak{S}^{(1),nc}$ by pushing out along the map from $J_\Prism^{(1),nc}$ to its $(p,E(u))$-completion. So it suffices to show that the map
\[ J_\Prism^{(1),nc} \to J_{\log}^{(1)}/(J_{\log}^{(1)})^2\]
factors over the $(p,E(u))$-completion of the source. Note that the target above is a free $\mathcal{S}$-module and hence $E(u)$-complete. Fixing some $k \geq 1$, it suffices to show that the resulting map
\[ J_\Prism^{(1),nc} \to \left(J_{\log}^{(1)}/(J_{\log}^{(1)})^2\right)/E(u)^k =: Q_k\]
factors over the $p$-completion of the source. We may regard this as a map from non-unital $\mathfrak{S}^{\otimes [1]}$-algebras, where the target $Q_k$ has a square-zero multiplication. Moreover, $Q_k$ is naturally a $p$-adic Banach vector space over $K_0$, while the source is generated by $\{\delta^n\left(\frac{u-v}{E(u)}\right)\}_{n \geq 0}$ as a non-unital $\mathfrak{S}^{\otimes [1]}$-algebra. As the multiplication on $Q_k$ is square-zero, it suffices to show that the image of the set $\{\delta^n\left(\frac{u-v}{E(u)}\right)\}_{n \geq 0}$ in the $p$-adic Banach space $Q_k$ is bounded. In fact, to prove boundedness, we may ignore finitely many terms, so it suffices to prove that $\{\delta^n\left(\frac{u-v}{E(u)}\right)\}_{n \geq c}$ maps to a bounded set in $Q_k$ for some $c = c(k)$. Observe that $\delta^n\left(\frac{u-v}{E(u)}\right)$ and $\frac{1}{p^n}\varphi^n\left(\frac{u-v}{E(u)}\right)$ have the same image in $Q_k$ for all $n \geq 0$: this follows by induction on $n$, using the formula $\delta(x) = \varphi(x)/p - x^p/p$ as well as the observation that any element of $(J_{\log}^{(1)})^p$ maps to $0$ in $Q_k$ since $p \geq 2$. It is therefore sufficient to check that  $\{\frac{1}{p^n}\varphi^n\left(\frac{u-v}{E(u)}\right)\}_{n \geq c}$ gives a bounded set in $Q_k$ for some $c \geq 0$. We then simplify
\[ \frac{1}{p^n}\varphi^n\left(\frac{u-v}{E(u)}\right) = \frac{u^{p^n}-v^{p^n}}{p^n \varphi^n(E(u))} = \frac{(u-v) \cdot (\sum_{i=0}^{p^n-1} u^i v^{p^n-i})}{p^n \varphi^n(E(u))}.\]
Now $\left(\mathcal{S}^{(1)}/(J_{\log}^{(1)})^2\right)/E(u)^k$ is a square-zero extension of $\mathcal{S}/E(u)^k$ by $Q_k$ and $u-v \in Q_k \subset \mathcal{S}^{(1)}/E(u)^k$, so multiplication by $u-v$ gives a map $\mathcal{S}/E(u)^k \to \mathcal{S}^{(1)}/((J_{\log}^{(1)})^2, E(u)^k)$; one checks that this map is bounded. Our problem then translates to check that $\{\frac{\sum_{i=0}^{p^n-1} u^i v^{p^n-i}}{p^n \varphi^n E(u)}\}_{n \geq c}$ has bounded image in $\mathcal{S}/E(u)^k$. But this image is given by setting $u=v$, so we are reduced to checking that $\{\frac{p^n u^{p^n}}{p^n \varphi^n E(u)}\}_{n \geq c}$ = $\{\frac{u^{p^n}}{\varphi^n E(u)}\}_{n \geq c}$ is bounded in the Banach algebra $\mathcal{S}/E(u)^k$ for some $c \geq 0$. As $u$ is power bounded, it is enough to show that $\{\frac{1}{\varphi^n E(u)}\}_{n \geq c}$ is bounded for some $c \geq 0$. But for any distinguished element $d$ in a $p$-local $\delta$-ring $A$, we have\footnote{This follows from the properties of Joyal's operations $\{\delta_n:A \to A\}_{n \geq 1}$ on $A$ extending $\delta = \delta_1$ (see \cite[Remarks 2.13, 2.14]{BhattScholzePrisms}). Indeed, these operations satisfy the following: for any $f \in A$ and $n \geq 1$, one has
\[ \varphi^n(f) = f^{p^n} + p \delta_1(f)^{p^{n-1}} + p^2 \delta_2(f)^{p^{n-2}} + ... + p^n \delta_n(f).\] 
Now if $\delta_1(f) = \delta(f)$ is a unit and $A$ is $p$-local, it follows that $\varphi^n(f) = pu  \mod f^{p^n} A$ for a unit $u \in A$.} 
\[ \varphi^n(d) = pu \mod d^cA\] 
for a unit $u$ and $c \leq p^n$; applying this observation to $A=\mathfrak{S}$ and $d=E(u)$ then shows that all elements of  $\{\frac{1}{\varphi^n E(u)}\}_{n \geq \log_p(k)}$ have the same absolute value in $\mathcal{S}/E(u)^k$, which trivially gives boundedness. Thus, we have constructed the map $\mathfrak{S}^{(1)} \to \mathcal{S}^{(1)}/(J_{\log}^{(1)})^2$ extending the natural map $\mathfrak{S}^{\otimes [1]} \to \mathcal{S}^{(1)} \to \mathcal{S}^{(1)}/(J_{\log}^{(1)})^2$.

Finally, it remains to extend across $\mathfrak{S}^{(1)} \to \mathfrak{S}^{(1)}\langle E(u)/p \rangle$. As $\mathfrak{S}^{(1)}\langle E(u)/p \rangle$ is the $p$-adic completion of $\mathfrak{S}^{(1)}[E(u)/p]$ and $p$ is invertible on $\mathcal{S}^{(1)}/(J_{\log}^{(1)})^2$, we may follow the reasoning used in the previous paragraph to reduce to checking that the set $\{\frac{E(u)^n}{p^n}\}_{n \geq 1}$ is bounded in the $p$-adic Banach algebra $\left(\mathcal{S}^{(1)}/(J_{\log}^{(1)})^2\right)/E(u)^k$  for each $k \geq 1$. But this is obvious as $\frac{E(u)^n}{p^n} = 0$ in this algebra for $n \geq k$. 
\end{proof}

\begin{construction}[Log connections on $\mathcal{S}$]
Write $\mathrm{Vect}^{\nabla,\log}(\mathcal{S})$ for the category of vector bundles $M$ on $\mathcal{S}$ equipped with a (continuous) logarithmic connection $\nabla:M \to M \otimes_{\mathcal{S}} \Omega^1_{\mathcal{S},\log}$. As $\Omega^1_{\mathcal{S},\log}$ is a free $\mathcal{S}$-module of rank $1$ with generator $\frac{du}{E(u)}$, specifying the connection $\nabla$ is equivalent to specifying an operator $N_\nabla = E(u)\frac{d}{du}:M \to M$ satisfying
\[ N_\nabla(f(u) m) = E(u) f'(u) m + f(u) N_\nabla(m)\]
for all $f(u) \in \mathcal{S}$ and $m \in M$. 
\end{construction}

\begin{corollary}[From prismatic crystals to log connections]
\label{LogConnectionBK}
Base changing along $\mathfrak{S}\langle I/p \rangle[1/p] \to \mathcal{S}=\mathfrak S[1/p]^\wedge_I$ lifts to a functor
\[ D_{\nabla}: \mathrm{Vect}(X_\Prism, \mathcal{O}_\Prism\langle \mathcal{I}_\Prism/p \rangle[1/p]) \to \mathrm{Vect}^{\nabla,\log}(\mathcal{S}).\]
\end{corollary}
\begin{proof}
Mimicking the argument in Proposition~\ref{VBDescentPrism}, one first checks that $\mathcal{O}_\Prism \langle \mathcal{I}_\Prism/p \rangle[1/p]$-vector bundles can be described explicitly: the natural map gives an equivalence
\[ \lim \mathrm{Vect}(\mathfrak{S}^{(\bullet)}\langle I/p \rangle[1/p]) \simeq \mathrm{Vect}(X_\Prism, \mathcal{O}_\Prism\langle \mathcal{I}_\Prism/p \rangle[1/p]).\]
Indeed, this reduces to the following observation: if  $A \to B$ is a $(p,I)$-completely faithfully flat map of prisms with $(p,I)$-completed Cech nerve $B^\bullet$, then $B^\bullet \langle I/p \rangle$ is the $p$-completed Cech nerve of the $p$-completely flat ring map $A\langle I/p \rangle \to B\langle I/p \rangle$, and thus we have the descent equivalence 
\[ \mathrm{Vect}(A\langle I/p \rangle[1/p]) \simeq \lim \mathrm{Vect}(B^\bullet \langle I/p \rangle[1/p])\] 
by Theorem~\ref{DMDescent}.

Fix a crystal $\mathcal{E} \in \mathrm{Vect}(X_\Prism, \mathcal{O}_\Prism\langle\mathcal{I}_\Prism/p\rangle)$. By the previous paragraph, this crystal has a value $\mathfrak{M} = \mathcal{E}(\mathfrak{S}) \in \mathrm{Vect}(\mathfrak{S}\langle I/p \rangle[1/p])$. We claim that $\mathfrak{M}_{\mathcal{S}} := \mathfrak{M} \otimes_{\mathfrak{S}} \mathcal{S} \in \mathrm{Vect}(\mathcal{S})$ naturally carries a log connection. By general nonsense on log infinitesimal cohomology, the data of a log connection on $\mathfrak{M}_{\mathcal{S}} \in \mathrm{Vect}(\mathcal{S})$ is exactly a lift of this object to $\lim \mathrm{Vect}(\mathcal{S}^{(\bullet)}_{\log}/(J^{(\bullet)}_{\log})^2)$; but such a lift is provided by base changing $\mathcal{E}(\mathfrak{S}^{(\bullet)}) \in \lim \mathrm{Vect}(\mathfrak{S}^{(\bullet)}\langle I/p \rangle[1/p])$ along the map from Lemma~\ref{PrismaticLogCech}, so the claim follows.
\end{proof}


\bibliographystyle{amsalpha}
\bibliography{PrismaticCrystals}

\end{document}